\documentclass[fleqn,reqno,11pt,a4paper,final]{amsart}

\usepackage[a4paper,left=35mm,right=35mm,top=35mm,bottom=35mm,marginpar=25mm]{geometry} 
\usepackage{amsmath}
\usepackage{amssymb}
\usepackage{amsthm}
\usepackage{amscd}
\usepackage[ansinew]{inputenc}
\usepackage{cite}
\usepackage{bbm}
\usepackage{color}
\usepackage[english=american]{csquotes}
\usepackage[final]{graphicx}
\usepackage{calc}
\usepackage{mathptmx}

\numberwithin{equation}{section}

\newtheoremstyle{thmlemcorr}{10pt}{10pt}{\itshape}{}{\bfseries}{.}{10pt}{{\thmname{#1}\thmnumber{ #2}\thmnote{ (#3)}}}
\newtheoremstyle{thmlemcorr*}{10pt}{10pt}{\itshape}{}{\bfseries}{.}\newline{{\thmname{#1}\thmnumber{ #2}\thmnote{ (#3)}}}
\newtheoremstyle{defi}{10pt}{10pt}{\itshape}{}{\bfseries}{.}{10pt}{{\thmname{#1}\thmnumber{ #2}\thmnote{ (#3)}}}
\newtheoremstyle{remexample}{10pt}{10pt}{}{}{\bfseries}{.}{10pt}{{\thmname{#1}\thmnumber{ #2}\thmnote{ (#3)}}}
\newtheoremstyle{ass}{10pt}{10pt}{}{}{\bfseries}{.}{10pt}{{\thmname{#1}\thmnumber{ A#2}\thmnote{ (#3)}}}

\theoremstyle{thmlemcorr}
\newtheorem{theorem}{Theorem}
\numberwithin{theorem}{section}
\newtheorem{lemma}[theorem]{Lemma}
\newtheorem{corollary}[theorem]{Corollary}
\newtheorem{proposition}[theorem]{Proposition}
\newtheorem{problem}[theorem]{Problem}

\theoremstyle{thmlemcorr*}
\newtheorem{theorem*}{Theorem}
\newtheorem{lemma*}[theorem]{Lemma}
\newtheorem{corollary*}[theorem]{Corollary}
\newtheorem{proposition*}[theorem]{Proposition}
\newtheorem{problem*}[theorem]{Problem}
\newtheorem{conjecture*}[theorem]{Conjecture}

\theoremstyle{defi}
\newtheorem{definition}[theorem]{Definition}

\theoremstyle{remexample}
\newtheorem{remark}[theorem]{Remark}
\newtheorem{example}[theorem]{Example}

\theoremstyle{ass}

\newcommand{\Crm}{\mathrm{C}}

\newcommand{\Lrm}{\mathrm{L}}

\newcommand{\Wrm}{\mathrm{W}}

\newcommand{\Acal}{\mathcal{A}}

\newcommand{\Fcal}{\mathcal{F}}

\newcommand{\Hcal}{\mathcal{H}}

\newcommand{\Jcal}{\mathcal{J}}

\newcommand{\Lcal}{\mathcal{L}}
\newcommand{\Mcal}{\mathcal{M}}

\newcommand{\Scal}{\mathcal{S}}

\newcommand{\Fbf}{\mathbf{F}}

\newcommand{\Mbf}{\mathbf{M}}

\newcommand{\Rbf}{\mathbf{R}}

\newcommand{\Ybf}{\mathbf{Y}}

\newcommand{\Abb}{\mathbb{A}}

\newcommand{\Cbb}{\mathbb{C}}

\newcommand{\Gbb}{\mathbb{G}}
\newcommand{\Hbb}{\mathbb{H}}

\newcommand{\Jbb}{\mathbb{J}}

\newcommand{\Sbb}{\mathbb{S}}

\DeclareMathOperator{\id}{id}

\DeclareMathOperator{\img}{img}

\DeclareMathOperator*{\wlim}{w-lim}
\DeclareMathOperator*{\wslim}{w*-lim}

\DeclareMathOperator{\supmod}{sup}
\DeclareMathOperator{\diverg}{div}
\DeclareMathOperator{\curl}{curl}
\DeclareMathOperator{\dist}{dist}
\DeclareMathOperator{\rank}{rank}

\DeclareMathOperator{\spn}{span}

\DeclareMathOperator{\supp}{supp}
\newcommand{\ee}{\mathrm{e}}
\newcommand{\ii}{\mathrm{i}}
\newcommand{\set}[2]{\left\{\, #1 \ \ \textup{\textbf{:}}\ \ #2 \,\right\}}

\newcommand{\setb}[2]{\bigl\{\, #1 \ \ \textup{\textbf{:}}\ \ #2 \,\bigr\}}
\newcommand{\setB}[2]{\Bigl\{\, #1 \ \ \textup{\textbf{:}}\ \ #2 \,\Bigr\}}
\newcommand{\setBB}[2]{\biggl\{\, #1 \ \ \textup{\textbf{:}}\ \ #2 \,\biggr\}}

\newcommand{\norm}[1]{\|#1\|}

\newcommand{\normb}[1]{\bigl\|#1\bigr\|}

\newcommand{\abs}[1]{|#1|}

\newcommand{\absb}[1]{\bigl|#1\bigr|}

\newcommand{\absBB}[1]{\biggl|#1\biggr|}

\newcommand{\floor}[1]{\lfloor #1 \rfloor}

\newcommand{\dpr}[1]{\langle #1 \rangle}

\newcommand{\dprb}[1]{\bigl\langle #1 \bigr\rangle}

\newcommand{\dprBB}[1]{\biggl\langle #1 \biggr\rangle}

\newcommand{\ddpr}[1]{\langle\!\langle #1 \rangle\!\rangle}

\newcommand{\ddprb}[1]{\bigl\langle\hspace{-2.5pt}\bigl\langle #1 \bigr\rangle\hspace{-2.5pt}\bigr\rangle}

\newcommand{\cl}[1]{\overline{#1}}
\newcommand{\di}{\mathrm{d}}
\newcommand{\dd}{\;\mathrm{d}}

\newcommand{\N}{\mathbb{N}}
\newcommand{\B}{\mathbb{B}}
\newcommand{\R}{\mathbb{R}}
\newcommand{\C}{\mathbb{C}}

\newcommand{\loc}{\mathrm{loc}}
\newcommand{\sym}{\mathrm{sym}}

\newcommand{\ONE}{\mathbbm{1}}

\newcommand{\toweak}{\rightharpoonup}
\newcommand{\toweakstar}{\overset{*}\rightharpoondown}

\newcommand{\toup}{\uparrow}
\newcommand{\todown}{\downarrow}

\newcommand{\embed}{\hookrightarrow}

\newcommand{\conv}{\star}
\newcommand{\BigO}{\mathrm{\textup{O}}}
\newcommand{\SmallO}{\mathrm{\textup{o}}}
\newcommand{\sbullet}{\begin{picture}(1,1)(-0.5,-2.5)\circle*{2}\end{picture}}
\newcommand{\frarg}{\,\sbullet\,}

\newcommand{\MCF}{\mathbf{MCF}}
\newcommand{\WF}{\mathrm{WF}}

\newcommand{\eps}{\epsilon}

\newcommand{\term}[1]{\textbf{#1}}

\newcommand{\proofstep}[1]{\textit{#1}}

\def\Xint#1{\mathchoice 
{\XXint\displaystyle\textstyle{#1}}%
{\XXint\textstyle\scriptstyle{#1}}%
{\XXint\scriptstyle\scriptscriptstyle{#1}}%
{\XXint\scriptscriptstyle\scriptscriptstyle{#1}}%
\!\int} 
\def\XXint#1#2#3{{\setbox0=\hbox{$#1{#2#3}{\int}$} 
\vcenter{\hbox{$#2#3$}}\kern-.5\wd0}} 
\def\dashint{\,\Xint-}

\newcommand{\restrict}{\begin{picture}(10,8)\put(2,0){\line(0,1){7}}\put(1.8,0){\line(1,0){7}}\end{picture}}

\renewcommand{\eps}{\varepsilon}
\renewcommand{\epsilon}{\varepsilon}
\renewcommand{\phi}{\varphi}

\begin{document}


\title[Microlocal compactness forms]{Directional oscillations, concentrations, and\\compensated compactness via\\microlocal compactness forms}

\author{Filip Rindler}
\address{Mathematics Institute, University of Warwick, Coventry CV4 7AL, United Kingdom, and University of Cambridge (on leave), Gonville \& Caius College, Trinity Street, Cambridge CB2 1TA, United Kingdom.}
\email{F.Rindler@warwick.ac.uk}

\begin{abstract}
This work introduces microlocal compactness forms (MCFs) as a new tool to study oscillations and concentrations in $\mathrm{L}^p$-bounded sequences of functions. Decisively, MCFs retain information about the location, value distribution, and direction of oscillations and concentrations, thus extending at the same time the theories of (generalized) Young measures and H-measures. In $\mathrm{L}^p$-spaces oscillations and concentrations precisely discriminate between weak and strong compactness, and thus MCFs allow one to quantify the difference in compactness. The definition  of MCFs involves a Fourier variable, whereby also differential constraints on the functions in the sequence can be investigated easily---a distinct advantage over Young measure theory. Furthermore, pointwise restrictions are reflected in the MCF as well, paving the way for applications to Tartar's framework of compensated compactness; consequently, we establish a new weak-to-strong compactness theorem in a ``geometric'' way. After developing several aspects of the abstract theory, we consider three applications: For lamination microstructures, the hierarchy of oscillations is reflected in the MCF. The directional information retained in an MCF is harnessed in the relaxation theory for anisotropic integral functionals. Finally, we indicate how the theory pertains to the study of propagation of singularities in certain systems of PDEs. The proofs combine measure theory, Young measures, and harmonic analysis.
\vspace{4pt}

\noindent\textsc{MSC (2010): 28B05 (primary); 35B05, 35A27, 49J45, 35L67.} 

\noindent\textsc{Keywords:} Microlocal compactness form, MCF, oscillation, concentration, microlocal analysis, relaxation, microstructure, laminate, compensated compactness, hyperbolic system, propagation of singularities.

\vspace{4pt}

\noindent\textsc{Date:} \today{} (version 3.0).
\end{abstract}




\maketitle



\section{Introduction}

One of the major challenges in the analysis of nonlinear PDEs is to understand oscillations and concentrations in weakly converging sequences of functions since these phenomena distinguish weak from strong convergence in $\Lrm^p$-spaces. In its most simple form, this observation is already expressed in Vitali's classical convergence theorem, by which a norm-bounded, non-oscillating (contained in the requirement of convergence in measure), non-concentrating (equiintegrable) sequence is seen to converge strongly. This connection between \enquote{abstract} compactness and \enquote{concrete} oscillation and concentration effects also supports the view that questions of compactness have a \enquote{physical} meaning and should be studied for their own sake, cf.~\cite{Tart79CCAP,Tart09GTH}.

Another motivation for such a study is that often sequences of interest are constrained by a linear differential constraint (curl-freeness for a sequence of gradients for instance) and/or pointwise constraints. Then, compensated compactness theory, see for example~\cite{Mura78CPC,Mura79CPC2,Tart79CCAP,Gera88CCRD,Gera91MDM,JoMeRa95TCCN,Tart09GTH}, considers whether these additional restrictions allow one to improve weak to strong compactness; in fact, in this vein, the term \enquote{compensated compactness} should probably be replaced by \enquote{compactness by compensation} (which is also closer to the French original \enquote{compacit\'{e} par compensation}), but we here stick to the more conventional terminology. At the heart of this investigation is an analysis of \emph{compatibility} between these constraints and different shapes of oscillations and concentrations. If together the constraints rule out any such phenomena, then we indeed have the sought improvement of weak to strong compactness.

Starting with simple scalar defect measures, more and more refined tools have been developed to study these questions. The basic philosophy here is to retain as much information as possible about the sequence in an \enquote{extended limit} because weak convergence always involves a lossy averaging operation. Perhaps the most well-known objects of this kind are Young measures, introduced by Young in~\cite{Youn42GSCV_1,Youn42GSCV_2}. They represent the asymptotic value distribution of oscillations in sequences: A bounded sequence $(u_j) \subset \Lrm^p(\Omega;\C^N)$ ($1 \leq p \leq \infty$) allows for the selection of a subsequence (not relabeled) that generates a Young measure $(\nu_x)_{x\in\Omega}$, defined to be a parametrized family of probability measures on $\C^N$, with the property that
\[
  \int_\Omega F(x,u_j(x)) \dd x  \to
  \int_\Omega \int F(x,z) \dd \nu_x(z) \dd x
\]
for any Carath\'{e}odory integrand $F \colon \Omega \times \C^N \to \R$ such that the sequence $(F(x,u_j(x)))_j$ is equiintegrable. Intuitively, the family $(\nu_x)_x$ contains the asymptotic value distribution of the sequence $(u_j)$ and thus reflects oscillations.

As an important extension, DiPerna and Majda~\cite{DiPMaj87OCWS} compactified the target space to describe concentration effects as well. The generalized Young measure (or DiPerna--Majda measure) generated by $(u_j)$ now adds to the classical Young measure $(\nu_x)_{x\in\Omega}$ a positive Radon measure $\lambda_\nu$ on $\cl{\Omega}$ and another parametrized family $(\nu_x^\infty)_{x\in \cl{\Omega}}$ of probability measures on the unit sphere $\partial \C\B^N$ of $\C^N$ ($\C\B^N$ always denotes the unit ball in $\C^N$). It allows us to compute limits of a larger class of nonlinear quantities with $p$-growth:
\[
  \int_\Omega F(x,u_j(x)) \dd x  \to
  \int_\Omega \int F(x,\frarg) \dd \nu_x \dd x + \int_{\cl{\Omega}} \int F^\infty(x,\frarg) \dd \nu_x^\infty \dd \lambda_\nu(x)
\]
whenever $F$ has at most growth of order $p$ and the $p$-recession function $F^\infty$ (describing the behavior of $F$ at infinity) exists in a sufficiently strong sense:
\[
  F^\infty(x,z) := \lim_{\substack{\!\!\!\! x' \to x \\ \!\!\!\! z' \to z \\ \; t \to \infty}} \frac{F(x',tz')}{t^p},
  \qquad x \in \cl{\Omega}, \; z \in \C^N.
\]
Here, $\lambda_\nu$ and $(\nu_x^\infty)$ respectively describe the location and value distribution of concentration effects in the sequence. Recent works on this topic include~\cite{AliBou97NUIG,KruRou97MDM,Sych00CHGY,KriRin10CGGY,Rind14LPCY,Rind12LSYM}.

A different set of ideas, which could be considered an \enquote{\mbox{$\Lrm^2$-microlocal} analysis}, is expressed in Tartar's theory of H-measures~\cite{Tart90HMNA} or, equivalently, G\'{e}rard's microlocal defect measures~\cite{Gera91MDM}. Switching to a phase-space or microlocal perspective, these objects enable us to quantify not only the location but also the \emph{direction} of oscillations and concentrations. Unfortunately, H-measures are not able to retain the asymptotic value distribution expressed in the Young measure, and are furthermore restricted to the \mbox{$\Lrm^2$-setting}; however, some extensions to an $\Lrm^p$-setting are possible, see~\cite{AntMit11HDEH}.

The question of whether it is possible to combine the complementary approaches of Young measures and H-measures and to find a tool allowing one to study \emph{at the same time} the location, value distribution, and direction of oscillations and concentration effects in $\Lrm^p$-bounded sequences of functions, has remained open for some time (cf.\ Chapters~38,~39 in~\cite{Tart06INSE} and~\cite{Tart95BYM} for some remarks on this problem).

The principal aim of this work is to introduce precisely such a tool, called the \term{microlocal compactness form} (MCF) generated by an $\Lrm^p$-bounded sequence, where $p \in (1,\infty)$. The following are some of its main features:
\begin{itemize}
  \item MCFs allow one to represent limits for a large class of nonlinear integral functionals; in particular, the Young measure is (essentially) contained.
  \item The directional information about oscillations and concentrations in the generating sequence is preserved and the H-measure can be reconstructed (in an $\Lrm^2$-setting).
  \item Differential and pointwise constraints on the sequence are reflected in the generated MCF and in fact hold for \emph{any} generating sequence (contrary to the situation for Young measures, for which generating sequences can have very different properties). Consequently, MCFs provide new \enquote{geometric} proofs in compensated compactness theory.
  \item MCFs can be defined for all $\Lrm^p$-spaces, $p \in (1,\infty)$.
  \item They enjoy the same compactness properties as the weak convergence, i.e.\ $\Lrm^p$-boundedness is enough for the extraction of a generating subsequence.
  \item The $p$-compactness wavefront set of an MCF is defined as the analogue of the wavefront set in classical microlocal analysis, but with respect to weak--strong compactness and not $\Crm^\infty$-regularity like in the classical notion.
  \item When considering a laminate (nested microstructure), its hierarchy is reflected in the generated MCF.
  \item MCFs contain enough information to serve as the foundation of a relaxation theory for integral functionals in the presence of anisotropy.
  \item For some hyperbolic systems, one can derive equations for MCFs that express the propagation of singularities.
\end{itemize}

Since we are interested in directional properties of oscillations and concentrations, the definition of microlocal compactness forms involves the Fourier transform. In fact, it is well-known that by switching to a Fourier point of view we can understand the intuition behind weak convergence phenomena much better: Fourier-transforming a sequence of functions $u_j \toweak u$ in $\Lrm^2$ that oscillate with increasingly higher frequencies (e.g.\ $u_j(x) = \sin(jx)$), one observes that the Fourier transforms $\hat{u}_j$ push mass out to infinity. Applying Parseval's theorem, the duality pairings with a test function $\phi \in \Lrm^2$ can be written as
\[
  \int \phi \cdot \overline{u_j} \dd x = \int \hat{\phi} \cdot \overline{\hat{u}_j} \dd \xi.
\]
Since $\hat{\phi}$ decays at infinity, integrating $\hat{u}_j$ against $\hat{\phi}$ corresponds to \enquote{band-limiting} the sequence $(u_j)$ and consequently, infinitely high frequencies are lost from the weak limit $u$. These considerations also persist for concentration effects, which are highly localized in space, and hence, by the uncertainty principle, are strongly smeared out in Fourier space. The band-limiting effect of a test function then implies that this Fourier mass is again lost in the limit. We conclude from this informal discussion that a tool aiming to improve the study of oscillations and concentrations must look at \enquote{infinite frequencies} only---microlocal compactness forms will do just that.

Another conclusion that can be drawn from this heuristic point of view on weak convergence is that for \emph{infinite} frequencies it is possible to \enquote{circumvent} the uncertainty principle for the Fourier transform: Localizing in real space corresponds to smearing out the Fourier transform, but when a packet of Fourier mass wanders out to infinity, it does not matter how much smeared out it is (we always look at conical neighborhoods around directions). This simple observation will enable us to construct a theory fulfilling all the aforementioned properties.

After these preliminary considerations, we can now state the main idea of microlocal compactness forms: Consider for a bounded sequence $(u_j) \subset \Lrm^p(\Omega;\C^N)$ the double limit
\begin{equation} \label{eq:omega_lim}
  \lim_{R\to\infty} \lim_{j\to\infty} \int_\Omega h(\frarg,u_j) \cdot \overline{T_{(1-\eta_R)\Psi}[u_j]} \dd x.
\end{equation}
Here, $h \colon \cl{\Omega} \times \C^N \to \C^N$, is a test function with sufficient regularity and we further require that its recession function $h^\infty$ exists in the sense
\[
  h^\infty(x,z) := \lim_{\substack{\!\!\!\! x' \to x \\ \!\!\!\! z' \to z \\ \; t \to \infty}} \frac{h(x',tz')}{t^{p-1}},
  \qquad x \in \cl{\Omega}, \; z \in \C^N.
\]
This choice for the class of test functions will capture the location and value distribution of oscillations and concentrations. Most often, we will not refer to $h$ directly, but only through $f(x,z,q) = h(x,z) \cdot q$ for reasons of consistency with the notation for generalized Young measures. Next, $\Psi \colon \R^d \setminus \{0\} \to \C^{N \times N}$ is a positively $0$-homogeneous Fourier multiplier with sufficient smoothness, serving to represent directional information as well as differential constraints. The family $\{\eta_R\}_{R > 0}$ consists of smooth functions with $\eta_R \equiv 1$ on the ball $B(0,R)$ and support inside $B(0,2R)$, whose purpose it is to cut off all finite frequencies. Finally, by $T_{(1-\eta_R)\Psi}$ we denote the Fourier multiplier operator with symbol $\xi \mapsto (1-\eta_R(\xi)) \Psi(\xi/\abs{\xi})$.

The fundamental existence theorem in Section~\ref{ssc:def} will establish that for all admissible $f,\Psi$ the limit in~\eqref{eq:omega_lim} can be expressed as
\[
 \int_\Omega \dprb{h(x,\frarg) \otimes \overline{\Psi},\omega_x} \dd x + \int_{\cl{\Omega}} \dprb{h^\infty(x,\frarg) \otimes \overline{\Psi},\omega_x^\infty} \dd \lambda_\omega(x) =: \ddprb{f \otimes \overline{\Psi},\omega},
\]
which is a sesquilinear form in $f$ and $\Psi$. Here, $(\omega_x)_{x\in\Omega}$ and $(\omega_x^\infty)_{x\in\cl{\Omega}}$ are parametrized sesquilinear forms and $\lambda_\omega$ is a finite Borel measure on $\cl{\Omega}$. The first part corresponds to oscillations, the second part to concentrations in the generating sequence $(u_j)$. In the course of this paper we will show that this object indeed quantifies the difference between weak and strong compactness of the sequence $(u_j)$ and has all the properties alluded to before. In particular, we will consider applications to compensated compactness theory (Section~\ref{ssc:comp_compact}), laminates (Section~\ref{ssc:laminates}), relaxation (Section~\ref{ssc:relax}), and to the propagation of singularities in hyperbolic systems (Section~\ref{ssc:propagation}).

Besides the theories of Young measures and H-measures, also classical microlocal analysis has served as inspiration for the present work. Following the groundbreaking work by H\"{ormander}, Kohn, Nirenberg (see Chapter~VIII of~\cite{Horm90ALPD1} for historical references), this theory has been very successfully applied to linear PDEs~\cite{Horm94ALPD3,Tayl11PDE2} and to a degree also to nonlinear PDEs~\cite{Tayl91PONP,Tayl11PDE3}. However, for nonlinear equations, microlocal analysis suffers from the defect that it measures $\Crm^\infty$-regularity, which is often not well adapted to nonlinear PDEs.

This paper is organized as follows: After recalling some preliminaries and fixing notation in Section~\ref{sc:setup}, we move to the heart of the matter in Section~\ref{sc:MCF} and prove the main existence theorem as well as many basic properties of MCFs, including a result on how the information of both the Young measure and the H-measure can be extracted from an MCF. We consider concrete oscillations and concentrations in Section~\ref{sc:osc_conc} and give a few basic examples. Then, Section~\ref{sc:diff_constr} examines how differential constraints on a sequence are reflected in the generated MCF and establishes a general compensated compactness theorem. Finally, Section~\ref{sc:appl} concludes the paper by considering the applications mentioned above.

\section*{Acknowledgements}

The author wishes to thank, among others, Giovanni Alberti, Camillo De Lellis, Sebastian Heinz, Dorothee Knees, Jan Kristensen, Alexander Mielke, Ayman Moussa, Stefan M\"{u}ller, and Richard Nickl for stimulating discussions related to the topic of the paper, partly during the Oberwolfach workshops on \enquote{Variational Methods for Evolution} in December 2011 and on the \enquote{Calculus of Variations} in July 2012.

\section{Setup and preliminaries} \label{sc:setup}

\subsection{General notation}

By $B(x_0,r)$ we denote the open ball around a point $x_0$ with radius $r > 0$; the special notation $\C\B^N$ will be used for the unit ball in $\C^N$. The unit sphere in $\C^N$ is $\partial \C\B^N$. We reserve the notation $\Sbb^{d-1}$ for the (real) unit sphere in $\R^d$. The scalar product $a \cdot b$ between two vectors $a,b \in \C^N$ will always be defined as $a \cdot b := \sum_k a^k b^k$ and we write out any complex conjugation in the second variable.

Often, we use the space $\C^{m \times d}$ of all $(m \times d)$-matrices (or restricted to $\R^{m \times d}$) in place of $\C^N$. We equip this matrix space with the Hilbert structure generated by the scalar product $A : B := \sum_{i,j} A^i_j B^i_j$, whereby we obtain the Frobenius norm $\abs{A} = (\sum_{i,j} \abs{A^i_j})^{1/2}$. The open unit ball and unit sphere with respect to this norm are denoted by $\C\B^{m \times d}$ and $\partial \C\B^{m \times d}$, respectively. As usual, $A^* := \overline{A^T}$.

The space of all $\C^N$-valued (vector) Radon measures on a Borel set $U \subset \R^d$ is denoted by $\Mbf(U;\C^N)$, the sets of positive Radon measures and probability measures are $\Mbf^+(U)$ and $\Mbf^1(U)$, respectively. We use the Lebesgue spaces $\Lrm^p(\Omega)$, $\Lrm^p(\Omega;\C^N)$, $\Lrm^1(U,\mu;\C^N)$, \ldots (where $\Omega \subset \R^d$ and $\mu \in \Mbf^+(\Omega)$ is a positive measure on $\Omega$), and the Sobolev spaces $\Wrm^{1,p}(\Omega),\Wrm^{1,p}(\Omega;\R^m)$, \ldots with their usual meanings; the corresponding norms are $\norm{\frarg}_p$, $\norm{\frarg}_{1,p}$, \ldots. Continuous and continuously differentiable functions are contained in the spaces $\Crm(\Omega),\Crm^1(\Omega), \Crm^\infty(\Omega)$ with norms $\norm{\frarg}_\infty, \norm{\frarg}_{1,\infty}$. The spaces \enquote{$\Crm_0$} and \enquote{$\Crm_c$} additionally entail that the contained functions vanish at the boundary/infinity or have compact support.

\subsection{Sphere compactifications} \label{ssc:sphere_compact}

In order to capture the behavior of functions defined on $\C^N$ at infinity we need to consider a suitable compactification of $\C^N$. The sphere compactification to be defined shortly, has turned out to be a good compromise between simplicity and generality and is used widely, see for example Chapter~32 of~\cite{Tart09GTH} or~\cite{FerGer03LZFN} in an H-measure context and~\cite{DiPMaj87OCWS,AliBou97NUIG} in the context of (generalized) Young measures. In contrast to the one-point compactification $\C^N \uplus \{\infty\}$ (\enquote{$\uplus$} refers to the disjoint union) it retains information about the \enquote{directions of infinity}. On the other hand, it is still sufficiently concrete to be useful in our analysis (but more general compactifications up to the Stone--\v{C}ech compactification $\beta \C^N$ could in principle also be of interest).

Formally, the \term{sphere compactification} $\sigma\C^N$ of $\C^N$ is the set
\[
  \sigma\C^N := \C^N \uplus \infty\partial \C\B^N,
  \qquad\text{where}\qquad
  \infty\partial \C\B^N := \setb{ \infty e }{ e \in \partial \C\B^N },
\]
with the topology such that $\sigma \C^N$ is homeomorphic to the closed unit ball $\cl{\C\B^N}$ under 
\[
  z \in \sigma \C^N \mapsto w(z) := \begin{cases}
    \displaystyle\frac{z}{1+\abs{z}}  &\text{if $z \in \C^N$,} \\
    e                                 &\text{if $z = \infty e \in \infty\partial \C\B^N$.}
  \end{cases}
\]

\subsection{Fourier transforms} \label{ssc:Fourier}

We follow~\cite{Graf08CFA,Stei93HA}, where also proofs of the following facts can be found, and define the Fourier transform and the inverse Fourier transform of functions $u,v \in \Lrm^1(\R^d)$ via
\begin{align*}
  \hat{u}(\xi) &= \Fcal[u](\xi) := \int u(x) \ee^{-2 \pi\ii x \cdot \xi} \dd x,
    \qquad \xi \in \R^d, \\
  \check{v}(x) &= \Fcal^{-1}[v](x) := \int v(\xi) \ee^{2 \pi\ii x \cdot \xi} \dd \xi,
    \qquad x \in \R^d.
\end{align*}

We observe the following fact, which will be needed in the proof of Lemma~\ref{lem:osc}: If $w \in \Scal'(\R)$ is a tempered distribution on $\R$, define the \term{standing wave} $v \in \Scal'(\R^d)$ with profile $w$ and direction $n_0 \in \Sbb^{d-1}$ by
\[
  \dprb{f,v} := \int_{\{x^\perp \perp n_0\}} \dprb{w,s \mapsto f(sn_0 + x^\perp)} \dd x^\perp,
  \qquad f \in \Scal(\R^d).
\]
For its Fourier transform $\hat{v} \in \Scal'(\R^d)$ we get $\hat{v}(\xi) = \hat{w}(s) \otimes \delta_0(\xi^\perp)$ for $\xi = sn_0 \oplus \xi^\perp$ ($\xi^\perp \perp n_0$), that is,
\[
  \dprb{f,\hat{v}} = \dprb{\hat{w},s \mapsto f(sn_0)},
  \qquad f \in \Scal(\R^d).
\]
As a particular case we have that if $\hat{w} \in \Lrm_\loc^1(\R)$, then $\hat{v} = \hat{w} \, \Hcal^1 \restrict \R n_0$, where $\R n_0$ is the line through the origin in direction $n_0$. For the proof consider without loss of generality the case $n_0 = e_1$ and choose $f$ to be a tensor product of test functions. Then use the fact that the Fourier transform of a tensor product is the tensor product of the Fourier transforms and conclude by virtue of $\Fcal[\Lcal^d] = \delta_0$.

\subsection{Multipliers} \label{ssc:multipliers}

Let $\psi \in \Crm^{\floor{d/2} + 1}(\R^d \setminus \{0\};\C)$ satisfy Mihlin's condition
\begin{equation} \label{eq:psi_multiplier_est}
  \abs{\partial^\alpha \psi(\xi)} \leq M \abs{\xi}^{-\abs{\alpha}},
  \qquad \xi \in \R^d \setminus \{0\},
\end{equation}
for all multi-indices $\alpha \in \N_0^d$ with $\abs{\alpha} := \abs{\alpha_1} + \cdots + \abs{\alpha_d} \leq \floor{d/2} + 1$ and a constant $M \geq 0$. Define $\norm{\psi}_\Mcal$ to be the smallest $M \geq 0$ for which~\eqref{eq:psi_multiplier_est} holds. It follows from the Mihlin multiplier theorem, see for instance Theorem~5.2.7 in~\cite{Graf08CFA}, that the corresponding multiplier operator
\[
  T_\psi[u] := \Fcal^{-1}[\psi\Fcal[u]],  \qquad u \in \Scal(\R^d),
\]
has a bounded extension from $\Lrm^p(\R^d)$ to itself, where here and in the following $p \in (1,\infty)$, and
\begin{equation} \label{eq:T_psi_est}
  \norm{T_\psi u}_p \leq C_d \max\{p,(p-1)^{-1}\} \norm{\psi}_\Mcal \norm{u}_p,
  \qquad u \in \Lrm^p(\R^d),
\end{equation}
with a dimensional constant $C_d > 0$.

Define the set $\Mcal(\Sbb^{d-1})$ to contain all $\psi \in \Crm^{\floor{d/2} + 1}(\R^d \setminus \{0\};\C)$ that additionally are positively $0$-homogeneous. Owing to this homogeneity, all such $\psi$ automatically satisfy the estimate~\eqref{eq:psi_multiplier_est} with $M = \sup_{\abs{\alpha} \leq \floor{d/2} + 1} \norm{\partial^\alpha \psi}_\infty$. Also, $\norm{\psi}_\Mcal$ is a norm on the space $\Mcal(\Sbb^{d-1})$.

We will also continually use the space $\Mcal(\Sbb^{d-1};\C^{N \times N})$ of $\C^{N \times N}$-valued multipliers, which we usually abbreviate by $\Mcal$. The associated multiplier operators $T_\Psi$, $\Psi \in \Mcal$, are defined as before, but with $\Psi(\xi)$ acting on $\hat{u}(\xi)$ via matrix-vector multiplication (if $u$ is a rapidly decaying function and by duality otherwise). These matrix-multipliers satisfy the same boundedness assertions as before. We consider scalar-valued $\psi \in \Mcal(\Sbb^{d-1})$ to be part of $\Mcal(\Sbb^{d-1};\C^{N \times N})$, identifying $\psi$ with $\Psi := \psi I_{N \times N}$, where $I_{N \times N}$ the $(N \times N)$-identity matrix.

\begin{remark}
Alternatively, one can view $T_\psi$ as the convolution (or singular integral) operator with the tempered distribution $\check{\psi}$. Since $\psi(\frarg/\abs{\frarg})$ is positively $0$-homogeneous, we may appeal to Proposition~2.4.7 in~\cite{Graf08CFA} to represent
\[
  T_\psi = \gamma f + W_\Omega \conv f,
\]
where $\gamma \in \C$ and $W_\Omega$ is given as the tempered distribution
\[
  \dprb{W_\Omega,\phi} := \lim_{\eps \to 0} \int_{\abs{x}\geq \eps} \frac{\Omega(x/\abs{x})}{\abs{x}^d} \phi(x) \dd x
\]
for a $\Crm^\infty$-function $\Omega$ on $\Sbb^{d-1}$ with vanishing integral.
\end{remark}

The following result will be used countless times in the sequel:

\begin{lemma}[Commutation] \label{lem:commutation}
Let $\phi \in \Crm_c(\R^d)$ and $\Psi \in \Mcal$. Define
\[
  M_\phi v := \phi \cdot v,  \qquad
  T_\Psi v := \Fcal^{-1}\bigl[ \Psi(\frarg/\abs{\frarg})\Fcal[v] \bigr],
  \qquad v \in \Lrm^2(\R^d;\C^N),
\]
to be the multiplication and Fourier multiplier operators, respectively. Then, the commutator
\[
  [M_\phi,T_\Psi] := M_\phi T_\Psi - T_\Psi M_\phi,
\]
is smoothing of order 1, that is, $[M_\phi,T_\Psi]$ extends to a bounded operator from $\Wrm^{-1,p}(\R^d;\C^N)$ to $\Lrm^p(\R^d;\C^N)$ for all $p \in (1,\infty)$. In particular, $[M_\phi,T_\Psi]$ is compact as an operator from $\Lrm^p(\R^d;\C^N)$ to itself.
\end{lemma}

\begin{proof}
A direct proof can be found in Lemma~1.7 of~\cite{Tart90HMNA} (also see~\cite{CoRoWe76FTHS} for a more advanced version with $\phi \in \mathrm{VMO}$). For smooth $\phi, \Psi$, the result is in fact well known in the theory of Fourier multipliers and pseudo-differential operators, cf.~\textsection 7.3 in Chapter~VI of~\cite{Stei93HA}. The present result can be reduced to that result by a smoothing argument.
\end{proof}

Another useful compactness result is the following:

\begin{lemma}[Band-limiting] \label{lem:bandlim_compact}
Let $\eta \in \Crm_c^\infty(\R^d)$ and let $\Omega \subset \R^d$ be a bounded open set. Then, the Fourier multiplier operator $T_{\eta}$ is compact from $\Lrm^p(\Omega;\C^N)$ to $\Lrm^p_\loc(\R^d;\C^N)$, that is,
\[
  T_\eta[u_j] \to T_\eta[u]  \quad\text{in $\Lrm^p_\loc(\R^d;\C^N)$}
  \qquad\text{whenever}\qquad
  u_j \toweak u \quad\text{in $\Lrm^p(\Omega;\C^N)$.}
\]
\end{lemma}

\begin{proof}
We consider functions in $\Lrm^p(\Omega;\C^N)$ to be implicitly extended by zero outside $\Omega$. The operator $T_\eta$ can be written as a convolution, i.e.\
\[
  T_\eta[u](x) = \bigl[\check{\eta} \conv u \bigr](x) = \int \check{\eta}(y) u(x-y) \dd y,
  \qquad u \in \Lrm^p(\Omega;\C^N),
\]
with $\check{\eta} \in \Scal(\R^d;\C^N) \subset \Lrm^{p/(p-1)}(\R^d;\C^N)$. From $u_j \toweak u$ in $\Lrm^p(\Omega)$ it then immediately follows that $T_\eta[u_j] \to T_\eta[u]$ pointwise and also that $\norm{T_\eta[u_j]}_\infty \leq \norm{\check{\eta}}_\infty \cdot \norm{u_j}_1 \leq C < \infty$ by Minkowski's inequality. Standard theory further yields $T_\eta[u_j] \toweak T_\eta[u]$ in $\Lrm^p(\R^d;\C^N)$ and hence in $\Lrm^1_\loc$. As the pointwise and the weak limit agree, it follows that $T_\eta[u_j] \to T_\eta[u]$ strongly in $\Lrm^1_\loc$, see for instance Theorem~IV.8.12 in~\cite{DunSch58LO1} (p.~295). By the uniform boundedness of the sequence, this implies convergence in $\Lrm^p_\loc(\R^d;\C^N)$ as well.
\end{proof}

We will also need:

\begin{lemma} \label{lem:TetaR_conv}
Let $\eta \in \Crm^\infty(\R^d;[0,1])$ with $\ONE_{B(0,1)} \leq \eta \leq \ONE_{B(0,2)}$ and define $\eta_R(\xi) := \eta(\xi/R)$ for $\xi \in \R^d$ and any $R > 1$. Then, the Fourier multipliers $T_{\eta_R}$ are uniformly $(\Lrm^p\to\Lrm^p)$-bounded and $T_{\eta_R}[u] \to u$ (strongly) in $\Lrm^p(\Omega;\C^N)$ for all $u \in \Lrm^p(\R^d;\C^N)$ as $R\to\infty$.
\end{lemma}

\begin{proof}
This lemma follows immediately from $T_\eta[u] = \check{\eta}_R \conv u$ and the fact that $\check{\eta}_R(x) = R^d \check{\eta}(Rx)$ is an approximation of the identity as $R\to\infty$, see for example Section~1.2.4 in~\cite{Graf08CFA}.
\end{proof}

We remark that multipliers can also be defined for less smooth (but not merely continuous) symbols, see~\cite{Stef10PORS} for a detailed study.

\subsection{Equiintegrability}

Recall that a family $\{u_j\} \subset \Lrm^p(\R^d;\C^N)$, where $p \in [1,\infty)$ is called \term{$p$-equiintegrable} if it is $\Lrm^p$-bounded and
\[
  \supmod_j \int_A \abs{u_j}^p \dd x  \quad\to\quad  0 \qquad\text{as $\abs{A} \to 0$.}
\]
There are several other equivalent ways of characterizing $p$-equiintegrability, see for example Theorem~2.29 in~\cite{FonLeo07MMCV}. We here only record the following simple fact, which follows from H\"{o}lder's inequality:

\begin{lemma} \label{eq:equiint_product}
Let $\{u_j\} \subset \Lrm^p(\R^d;\C^N)$, $p \in [1,\infty)$, be a $p$-equiintegrable family and $\{v_j\} \subset \Lrm^{p'}(\R^d;\C^N)$ be uniformly $\Lrm^{p'}$-bounded, where $1/p + 1/p' = 1$. Then, the family $\{u_j \cdot v_j\}$ is ($1$-)equiintegrable.
\end{lemma}

\subsection{Test functions} \label{ssc:vertical_test_funct}
Besides \enquote{horizontal} test functions that are used to localize in space, we also need our test functions to depend on the \enquote{vertical} argument, so that they allow us to localize in the target space as well. Spaces of such test functions necessarily depend on the exponent of the $\Lrm^p$-space, in which we are working.

First, for $q \in (0,\infty)$ and functions $h \in \Crm(\C^N)$, $g \in \Crm(\C\B^N)$ (recall that $\C\B^N$ denotes the unit ball in $\C^N$) define
\begin{align}
  S^q h(w) &:= (1-\abs{w})^q h \biggl( \frac{w}{1-\abs{w}} \biggr),  \qquad w \in \C\B^N,   \label{eq:Sp} \\
  S^{-q} g(z) &:= (1+\abs{z})^q g \biggl( \frac{z}{1+\abs{z}} \biggr),  \qquad z \in \C^N.  \notag
\end{align}
It is easy to check that $S^q S^{-q} = \id$ and $S^{-q} S^q = \id$.

Now let $p \in (1,\infty)$ and define the space $\Fbf^p(\Omega;\C^N)$ of \term{test functions}, where $\Omega \subset \R^d$ is open, as follows (it will become clear in Section~\ref{ssc:def} why the definition is chosen as such):
\begin{align*}
  \Fbf^p(\Omega;\C^N) := \setb{ f \in \Crm(\cl{\Omega} \times \C^N \times \C^N;\C) }{ &\text{$f(x,z,q) = h(x,z) \cdot q$ and}\\
  &\text{$S^{p-1}h \in \Crm(\cl{\Omega \times \C\B^N};\C^N)$} }.
\end{align*}
The norm of an element $f \in \Fbf^p(\Omega;\C^N)$ with $f(x,z,q) = h(x,z) \cdot q$ is
\[
  \norm{f}_{\Fbf^p} := \sup \, \setb{ \abs{S^{p-1}h(x,w)} }{(x,w) \in \cl{\Omega \times \C\B^N}}.
\]
Clearly, for every $f(x,z,q) = h(x,z) \cdot q \in \Fbf^p(\Omega;\C^N)$, the associated $h$ has a \term{(strong) $(p-1)$-recession function}
\begin{equation} \label{eq:h_infty}
  h^\infty(x,z) := \lim_{\substack{\!\!\!\! x' \to x \\ \!\!\!\! z' \to z \\ \; t \to \infty}} \frac{h(x',tz')}{t^{p-1}},
  \qquad x \in \cl{\Omega}, \; z \in \C^N.
\end{equation}
Notice that $S^{p-1}h(x,w) = h^\infty(x,w)$ for all $(x,\omega) \in \cl{\Omega} \times \partial \C\B^N$. Furthermore,
\begin{equation} \label{eq:h_est}
  \abs{h(x,z)} \leq \norm{f}_{\Fbf^p}(1 + \abs{z})^{p-1} \leq 2^{p-1} \norm{f}_{\Fbf^p}(1 + \abs{z}^{p-1})
\end{equation}
and
\[
  \abs{h^\infty} \leq \norm{f}_{\Fbf^p}.
\]

Finally, we remark that for $f(x,z,q) = h(x,z) \cdot q \in \Fbf^p(\Omega;\C^N)$ we can consider $h$ also as a function on the extended compactified space $\cl{\Omega} \times \sigma\C^N$. Define this extended function $Eh$ as follows:
\begin{equation} \label{eq:Eh}
\begin{aligned}
  &Eh \colon \cl{\Omega} \times \sigma\C^N \to \C^N, \\
  &Eh(x,z) := (1+\abs{z})^{-(p-1)} \, h(x,z),  \qquad (x,z) \in \cl{\Omega} \times \sigma\C^N.
\end{aligned}
\end{equation}
One can then check that the assumptions on $f$ entail that $Eh$ is continuous (with respect to the relevant topologies) and $\abs{Eh} \leq \norm{f}_{\Fbf^p}$.

\section{Microlocal compactness forms} \label{sc:MCF}

\subsection{Definition and existence of MCFs} \label{ssc:def}

In all of the following let $p \in (1,\infty)$ and $\Omega \subset \R^d$ an open bounded set (for unbounded sets $\Omega$ a largely equivalent theory is available as long as all spaces are replaced by their respective local versions). We will work with parametrized sesquilinear forms
\[
  \omega_x \colon \Crm(\C^N;\C^N) \times \Mcal \to \C,
  \qquad x \in \cl{\Omega},
\]
where as before $\Mcal = \Mcal(\Sbb^{d-1};\C^{N \times N})$. We usually write the application of this sesquilinear form to the pair $(g,\Psi)$ in the \enquote{tensor} notation
\[
  \dprb{g \otimes \overline{\Psi},\omega_x} := \omega_x(g,\Psi).
\]
Let now $\lambda \in \Mbf^+(\cl{\Omega})$ be a positive finite Borel measure. Then, we say that the family $(\omega_x)_{x \in \cl{\Omega}}$ is \term{weakly* measurable} with respect to $\lambda$ if for every fixed $\Psi \in \Mcal$ and every fixed Carath\'{e}odory-type function $g \colon \cl{\Omega} \times \C^N \to \C^N$ with the properties that $x \mapsto g(x,z)$ is $\lambda$-measurable for all fixed $z$ and $z \mapsto g(x,z)$ is continuous for all fixed $x$, the compound function
\[
  x \mapsto \dprb{g(x,\frarg) \otimes \overline{\Psi},\omega_x}
\]
is $\lambda$-measurable. For a family
\[
  \omega_x^\infty \colon \Crm(\partial \C\B^N;\C^N) \times \Mcal \to \C,
  \qquad x \in \cl{\Omega},
\]
we define weak*-measurability in a completely analogous manner.

\begin{definition} \label{def:MCF}
Let $\omega = (\omega_x,\lambda_\omega,\omega_x^\infty)$ be a triple consisting of
\begin{itemize}
\item[(i)] a parametrized family $(\omega_x)_{x \in \Omega}$ of continuous sesquilinear forms
\[
  \qquad \omega_x \colon \Crm(\C^N;\C^N) \times \Mcal \to \C,
\]
\item[(ii)] a positive and finite measure $\lambda_\omega \in \Mbf(\cl{\Omega})$, and
\item[(iii)] a parametrized family $(\omega_x^\infty)_{x \in \cl{\Omega}}$ of continuous sesquilinear forms
\[
  \qquad \omega_x^\infty \colon \Crm(\partial \C\B^N;\C^N) \times \Mcal \to \C.
\]
\end{itemize}
Also assume that $\omega$ satisfies the following properties:
\begin{itemize}
\item[(iv)] The family $(\omega_x)_x$ is weakly* measurable with respect to $\Lcal^d \restrict \Omega$ and the family $(\omega_x^\infty)_x$ is weakly* measurable with respect to $\lambda_\omega$.
\item[(v)] For every fixed $\Psi \in \Mcal$, the function $x \mapsto \dpr{\abs{z}^{p-2} z \otimes \overline{\Psi},\omega_x}$ lies in $\Lrm^1(\Omega;\C)$ and the function $x \mapsto \dpr{z \otimes \overline{\Psi},\omega_x^\infty}$ (with the function $\partial \C\B^N \ni z \mapsto z$) lies in $\Lrm^1(\cl{\Omega},\lambda_\omega;\C)$.
\end{itemize}
Then, $\omega$ is called a \term{$p$-microlocal compactness form (MCF)} on $\Omega$ with target space $\C^N$. The set of all such microlocal compactness forms is denoted by $\MCF^p(\Omega;\C^N)$.
\end{definition}

\begin{definition}
Let $\omega \in \MCF^p(\Omega;\C^N)$, $f \in \Fbf^p(\Omega;\C^N)$ with $f(x,z,q) = h(x,z) \cdot q$, and $\Psi \in \Mcal$. Then, we define the \term{duality product}
\begin{equation} \label{eq:omega_ddpr}
  \ddprb{f \otimes \overline{\Psi},\omega} := \int_\Omega \dprb{h(x,\frarg) \otimes \overline{\Psi},\omega_x} \dd x + \int_{\cl{\Omega}} \dprb{h^\infty(x,\frarg) \otimes \overline{\Psi},\omega_x^\infty} \dd \lambda_\omega(x).
\end{equation}
If we want to stress the domain $\Omega$, we write $\ddpr{\frarg,\frarg}_\Omega$.
\end{definition}

A justification for the above tensor-product notation will be provided in Section~\ref{ssc:ext_repr_distrib} below. Notice that according to the properties of $\omega$, we may also consider $\omega = \omega(\frarg,\frarg)$ as a sesquilinear form on $\Fbf^p(\Omega;\C^N) \times \Mcal$ and then
\[
  \ddprb{f \otimes \overline{\Psi},\omega} = \omega(f,\Psi).
\]

To state the main existence theorem for microlocal compactness forms, we first fix $\eta \in \Crm_c^\infty(\R^d)$ with $\ONE_{B(0,1)} \leq \eta \leq \ONE_{B(0,2)}$. Set $\eta_R(\xi) := \eta(\xi/R)$ for $\xi \in \R^d$ and any $R > 1$. It will be seen shortly that the theory is independent of the choice of $\eta$.

\begin{theorem} \label{thm:omega}
Let $(u_j) \subset \Lrm^p(\Omega;\C^N)$ be a norm-bounded sequence, where $1 < p < \infty$. Then, after selecting a subsequence (not relabeled), there exist a $p$-microlocal compactness form $\omega \in \MCF^p(\Omega;\C^N)$ such that for all $f \in \Fbf^p(\Omega;\C^N)$ with $f(x,z,q) = h(x,z) \cdot q$ and $\Psi \in \Mcal$ it holds that
\begin{equation} \label{eq:omega}
  \ddprb{f \otimes \overline{\Psi},\omega} = \lim_{R\to\infty} \lim_{j\to\infty} \int_\Omega h(\frarg,u_j) \cdot \overline{T_{(1-\eta_R)\Psi}[u_j]} \dd x,
\end{equation}
where $T_{(1-\eta_R)\Psi}$ is the Fourier multiplier operator with symbol $\xi \mapsto (1-\eta_R(\xi))\Psi(\xi/\abs{\xi})$. Moreover, the following properties hold:
\begin{itemize}
\item[(A)] \textbf{Independence.} The MCF $\omega$ is independent of the choice of $\eta$ and the choice of the sequence $R \to \infty$ (that is, every sequence $(R_n)_n \subset \R$ with $R_n \to \infty$ as $n \to \infty$ gives the same result).
\item[(B)] \textbf{Basic estimate.} For all $f,\Psi$ as above it holds that
\begin{equation} \label{eq:omega_est}
  \qquad \absb{\ddprb{f \otimes \overline{\Psi},\omega}} \leq C \bigl( \supmod_j \norm{1 + \abs{u_j}}_p^p \bigr) \cdot \norm{f}_{\Fbf^p} \cdot \norm{\Psi}_\Mcal,
\end{equation}
where $C = C_{d,N} \max\{p,(p-1)^{-1}\}$ with a dimensional constant $C_{d,N}$.
\item[(C)] \textbf{Canonical form.} It is possible to choose
\begin{equation} \label{eq:lambda_omega}
  \qquad \lambda_\omega = \wslim_{j\to\infty} \, \abs{u_j}^p \, \Lcal^d \restrict \Omega.
\end{equation}
\end{itemize}
\end{theorem}

\begin{definition}
In the situation of the preceding theorem we say that the (sub)sequence $(u_j)$ \term{generates} $\omega$.
\end{definition}

\begin{remark}
Let us make a few remarks about the definition of MCFs:
\begin{enumerate}
\item The definition is tailored so that we can localize in both the $x$- and in the $\xi$-variable (at least to some degree). At first sight this appears to contradict the uncertainty principle for the Fourier transform, but as we here only look at oscillations and concentrations that have asymptotically \emph{infinite} frequency, the uncertainty principle does not apply (as was already the case for H-measures).
\item Further to the first point, the purpose of the outer limit $R\to\infty$ is precisely for the microlocal compactness form to \enquote{forget} finite frequencies. This is important both from a conceptual and from a technical point of view, and exploited for example in Lemma~\ref{lem:phi_exchange} below.
\item For $\Lcal^d$-a.e.\ $x \in \Omega$, $\omega_x$ can also be considered as an anti-linear bounded function from $\Mcal$ to the space of $\C^N$-valued Radon measures on $\Cbb^N$ and likewise, for $\lambda_\omega$-a.e.\ $x \in \cl{\Omega}$, $\omega_x^\infty$ can be considered as an anti-linear bounded function from $\Mcal$ into the space of $\C^N$-valued Radon measures on the unit sphere $\partial \C\B^N$. So, MCFs can be considered as \enquote{multiplier-parametrized} measures.
\item While the family $(\omega_x)_x$ is easily seen to be uniquely determined, the choice of $(\omega_x^\infty)_x$ and $\lambda_\omega$ is of course \emph{not} unique: Multiplying $\lambda_\omega$ by any measurable function $\gamma \colon \cl{\Omega} \to (0,\infty)$ and then multiplying $\omega_x^\infty$ correspondingly by $1/\gamma(x)$ yields the same MCF. However, in this paper we will always use the canonical choice~\eqref{eq:lambda_omega}, whereby also $(\omega_x)_x^\infty$ is determined $\lambda_\omega$-almost everywhere.
\item One can also define a local version of MCFs, where $\Lrm^p$ is replaced by $\Lrm^p_\loc$. Then, $\lambda_\omega$ is only $\sigma$-finite and all estimates are only valid when restricted to compact subsets of $\Omega$ (with constants depending on that compact subset).
\item The above definition sheds some light on why we chose to work with test functions $f(x,z,q) = h(x,z) \cdot q \in \Fbf^p(\Omega;\C^N)$ instead of using $h(x,z)$ directly: In the above definition, we plug $\overline{T_{(1-\eta_R)\Psi}[u_j]}$ into the argument $q$, so $h$ by itself is only \enquote{half} the test function. Also, it is our aim to investigate $p$-growth concentrations and $h$ only has $(p-1)$-growth.
\item For the special case $p=2$ see Sections~\ref{ssc:ext_repr_distrib} and~\ref{ssc:MCF_YM_HM}.
\end{enumerate}
\end{remark}

\begin{remark} \label{rem:omega_alt}
Using Parseval's formula, we can rewrite the integral in~\eqref{eq:omega}:
\begin{align*}
  \int_\Omega h(\frarg,u_j) \cdot \overline{T_{(1-\eta_R)\Psi}[u_j]} \dd x
    &= \int_\Omega T_{(1-\eta_R)\Psi^*} \bigl[ h(\frarg,u_j) \bigr] \cdot \overline{u_j} \dd x \\
  &= \dprb{ (1-\eta_R)\Psi^* \Fcal \bigl[ h(\frarg,u_j) \bigr], \overline{\Fcal[u_j]} },
\end{align*}
where $\Psi^*(\xi) := \overline{\Psi(\xi)^T}$ as usual. In the last line one of the Fourier transforms is applied to a function in $\Lrm^q$ with $q \leq 2$ and one to a function in $\Lrm^{q'}$ with $1/q + 1/q' = 1$ so that $q' \geq 2$, which (unless $p = 2$) is only defined as a tempered distribution. The bracket $\dpr{\frarg,\frarg}$ therefore is to be understood in an appropriate distributional sense. If additionally $u_j \in (\Lrm^1 \cap \Lrm^p)(\Omega;\C^N)$ and $\hat{u}_j \in (\Lrm^1 \cap \Lrm^{p'})(\Omega;\C^N)$ with $1/p + 1/p' = 1$, we can further express this as
\begin{align}
  &\int_{\R^d} (1-\eta_R(\xi))\Psi^* \biggl(\frac{\xi}{\abs{\xi}}\biggr) \Fcal \bigl[ h(\frarg,u_j) \bigr](\xi) \cdot \overline{\Fcal[u_j](\xi)} \dd \xi  \label{eq:MCF_L2} \\
  &= \int_{\R^d} \int_\Omega \int_\Omega (1-\eta_R(\xi))\Psi^* \biggl(\frac{\xi}{\abs{\xi}}\biggr) h(x,u_j(x)) \cdot \overline{u_j(y)} \, \ee^{2\pi\ii (y-x) \cdot \xi} \dd y \dd x \dd \xi.  \label{eq:full_integral}
\end{align}
\end{remark}

\begin{proof}[Proof of Theorem~\ref{thm:omega}]
We divide the proof into several steps.

\proofstep{Step 1.} For any $f \in \Fbf^p(\Omega;\C^N)$ with $f(x,z,q) = h(x,z) \cdot q$, we have by H\"{o}lder's inequality,~\eqref{eq:h_est}, and the usual Mihlin multiplier estimates, cf.~\eqref{eq:T_psi_est}, that
\begin{align}
  &\int \absb{h(\frarg,u_j) \cdot \overline{T_{(1-\eta_R)\Psi}[u_j]}} \dd x  \notag\\
  &\qquad \leq \norm{T_{1-\eta_R}}_{\Lrm^p\to\Lrm^p} \cdot \norm{T_\Psi}_{\Lrm^p\to\Lrm^p} \cdot \norm{f}_{\Fbf^p} \cdot \norm{(1+\abs{u_j})^{p-1}}_{p/(p-1)} \cdot \norm{u_j}_p  \notag\\
  &\qquad \leq C_{d,N} \max\{p,(p-1)^{-1}\} \cdot \norm{1+\abs{u_j}}_p^p \cdot \norm{f}_{\Fbf^p} \cdot \norm{\Psi}_\Mcal.  \label{eq:omega_n}
\end{align}
It can be calculated that $\norm{T_{1-\eta_R}}_{\Lrm^p\to\Lrm^p} \leq \norm{T_{1-\eta}}_{\Lrm^p\to\Lrm^p}$ and hence this norm does not depend on $R$. We will show below that $\eta$ can be chosen arbitrarily without changing the definition of MCFs, hence we can pick one and absorb $\norm{T_{1-\eta}}$ into $C_{d,N}$.

Now take countable dense subsets $\{f_m(x,z,q) = h_m(x,z) \cdot q\}_{m\in\N} \subset \Fbf^p(\Omega;\C^N)$, $\{\Psi_n\}_{n\in\N} \subset \Mcal$ (see the proof of Lemma~\ref{lem:countable_test} below) such that all $f_m, \Psi_n$ are smooth, and use a diagonal construction to select a subsequence of $(u_j)$ (not relabeled) such that all the limits
\[
  I_{m,n} := \lim_{j\to\infty} \int h_m(\frarg,u_j) \cdot \overline{T_{\Psi_n}[u_j]} \dd x,  \qquad m,n \in \N,
\]
and
\[
  H_m := \wlim_{j\to\infty} \, h_m(\frarg,u_j)  \qquad\text{in $\Lrm^{p/(p-1)}(\Omega;\C^N)$},  \qquad m \in \N,
\]
exist. Selecting another (not relabeled) subsequence if necessary, we may furthermore assume that
\[
  u_j \toweak u   \qquad\text{in $\Lrm^p(\Omega;\C^N)$.}
\]

We have for every $R > 0$,
\begin{align*}
  &\lim_{j\to\infty} \int h_m(\frarg,u_j) \cdot \overline{T_{(1-\eta_R)\Psi_n}[u_j]} \dd x \\
  &\qquad = \lim_{j\to\infty} \int h_m(\frarg,u_j) \cdot \overline{T_{\Psi_n}[u_j]} \dd x
    - \lim_{j\to\infty} \int h_m(\frarg,u_j) \cdot \overline{T_{\eta_R\Psi_n}[u_j]} \dd x \\
  &\qquad = I_{m,n} - \int H_m \cdot \overline{T_{\eta_R\Psi_n}[u]} \dd x,
\end{align*}
where the last equality holds as $T_{\eta_R \Psi_n}$ is a compact operator by Lemma~\ref{lem:bandlim_compact}. Letting $R\to\infty$, we get by virtue of Lemma~\ref{lem:TetaR_conv},
\[
  \lim_{R\to\infty} \lim_{j\to\infty} \int h_m(\frarg,u_j) \cdot \overline{T_{(1-\eta_R)\Psi_n}[u_j]} \dd x
  = I_{m,n} - \int H_m \cdot \overline{T_{\Psi_n}[u]} \dd x
\]
for any sequence $R\to\infty$. This shows the independence from the choice of the sequence $R\to\infty$ in~(A), at least for all $f_m \otimes \overline{\Psi}_n$, the general assertion then follows by density (see below).

We next show that the value of the expression in~\eqref{eq:omega} does not depend on the choice of $\eta$. For two such choices $\eta, \zeta$ we observe
\begin{equation} \label{eq:eta_indep}
  \lim_{j\to\infty} \int_\Omega h_m(\frarg,u_j) \cdot \overline{T_{(1-\eta_R)-(1-\zeta_R)} T_\Psi[u_j]} \dd x = \int_\Omega H_m \cdot \overline{T_\Psi T_{\zeta_R-\eta_R}[u]} \dd x,
\end{equation}
since $T_{\zeta_R-\eta_R}$ is a convolution operator with a rapidly decaying function and so $T_{\zeta_R-\eta_R}[u_j] \to T_{\zeta_R-\eta_R}[u]$ strongly (cf.\ the proof of Lemma~\ref{lem:bandlim_compact}). Both $\eta$ and $\zeta$ are identically $1$ on the unit ball $B(0,1)$. Thus, $\zeta_R - \eta_R$ vanishes on the ball $B(0,R)$ and so, the right hand side of~\eqref{eq:eta_indep} tends to zero as $R\to\infty$ by Lemma~\ref{lem:TetaR_conv}. This establishes the second part of~(A).

We can now unambiguously define for all $m,n \in \N$,
\[
  \ddprb{f_m \otimes \overline{\Psi}_n,\omega} := \lim_{R\to\infty} \lim_{j\to\infty} \int h_m(\frarg,u_j) \cdot \overline{T_{(1-\eta_R)\Psi_n}[u_j]} \dd x.
\]
From~\eqref{eq:omega_n} we now immediately see~\eqref{eq:omega_est} for $f_m \otimes \overline{\Psi}_n$. By density we can then extend this definition to all $f \otimes \overline{\Psi} \in \Fbf^p(\Omega;\C^N) \otimes \Mcal$. In particular, one can check, using an easy Cauchy sequence argument, that the limits in the definition~\eqref{eq:omega} exist and are independent of the choice of $\eta$ and the sequence $R\to\infty$. Subsequently,~(B) holds for all $f \otimes \overline{\Psi}$.

\proofstep{Step 2.}
In the following we establish the properties (i)--(v) from the definition of microlocal compactness forms. Fix $\Psi \in \Mcal$, set $v_{j,R}(x) := \overline{T_{(1-\eta_R)\Psi}[u_j](x)}$, and define the positive measures
\[
  \mu_{j,R}^\Psi \in \Mbf^+(\cl{\Omega \times \C\B^N \times \C\B^N})
\]
via
\[
  \mu_{j,R}^\Psi := (1+\abs{u_j(x)})^{p-1} \, (1+\abs{v_{j,R}(x)}) \, \Lcal_x^d \restrict \Omega \otimes \delta_{\frac{u_j(x)}{1+\abs{u_j(x)}}} \otimes \delta_{\frac{v_{j,R}(x)}{1+\abs{v_{j,R}(x)}}}.
\]
By H\"{o}lder's inequality and a multiplier estimate similar to~\eqref{eq:omega_n},
\begin{align*}
  &\abs{\mu_{j,R}^\Psi}(\cl{\Omega \times \C\B^N \times \C\B^N}) \\
  &\qquad = \int_\Omega (1+\abs{u_j})^{p-1}(1+\abs{v_{j,R}}) \dd x  \\
  &\qquad \leq \biggl(\int_\Omega (1+\abs{u_j})^p \dd x \biggr)^{(p-1)/p} \cdot \biggl(\int_\Omega (1+\abs{v_{j,R}})^p \dd x \biggr)^{1/p}  \notag \\
  &\qquad \leq C\norm{1+\abs{u_j}}_p^p.
\end{align*}
Thus, for all $f(x,z,q) = h(x,z) \cdot q \in \Fbf^p(\Omega;\C^N)$,
\begin{align}
  &\int_\Omega h(\frarg,u_j) \cdot \overline{T_{(1-\eta_R)\Psi}[u_j]} \dd x \notag\\
  &\qquad = \int_\Omega h(\frarg,u_j) \cdot v_{j,R} \dd x  \notag\\
  &\qquad = \int_\Omega (1+\abs{u_j})^{p-1} (S^{p-1}h)\biggl(x, \frac{u_j}{1+\abs{u_j}} \biggr) \cdot v_{j,R} \dd x  \notag\\
  &\qquad = \int (S^{p-1}h)(x,w) \cdot r \,\dd \mu_{j,R}^\Psi(x,w,r).  \label{eq:mu_int}
\end{align}
Now fix a sequence $R \to \infty$ and select further subsequences such that $\mu_{j,R}^\Psi \toweakstar \mu_R^\Psi$ as $j \to \infty$ (for all $R$ in the chosen sequence) and $\mu_R^\Psi \toweakstar \mu^\Psi$ as $R \to \infty$, where both limits are to be understood in $\Mbf^+(\cl{\Omega \times \C\B^N \times \C\B^N})$. Passing to the limit $j \to \infty$ and then $R \to \infty$ in the last expression,
\begin{equation} \label{eq:omega_mu_Psi}
  \ddprb{f \otimes \overline{\Psi},\omega} = \int (S^{p-1}h)(x,w) \cdot r \,\dd \mu^\Psi(x,w,r).
\end{equation}
From the previous proof step, we also conclude that we in fact did not need to select any further subsequences above.

In the following we suppress any dependence on $\Psi$, which for the moment we consider as fixed. Disintegrating $\mu = \mu^\Psi$, for instance using Theorem~2.28 in~\cite{AmFuPa00FBVF}, we infer the existence of $\gamma \in \Mbf(\cl{\Omega \times \C\B^N})$ and a family $(\sigma_{x,w})_{(x,w) \in \cl{\Omega \times \C\B^N}}$ of probability measures on $\cl{\C\B^N}$ that is weakly* measurable with respect to $\gamma$, and is such that
\[
  \mu = \gamma(\di x \, \di w) \otimes \sigma_{x,w},
\]
i.e.
\[
  \int \Phi(x,w,r) \dd \mu = \int \int \Phi(x,w,r) \dd \sigma_{x,w}(r) \dd \gamma(x,w)
\]
for all $\Phi \in \Crm(\cl{\Omega \times \C\B^N \times \C\B^N})$. Apply the disintegration theorem again to decompose
\[
  \gamma = \kappa(\di x) \otimes \eta_x(\di w)
\]
for $\kappa \in \Mbf(\cl{\Omega})$ and a family $(\eta_x)_{x \in \cl{\Omega}}$ of probability measures on $\cl{\C\B^N}$ that are weakly* measurable with respect to $\kappa$. Moreover,
\[
  \kappa(A) = \mu(A \times \cl{\C\B^N \times \C\B^N})
  \qquad\text{for all Borel sets $A \subset \cl{\Omega}$.}
\]

Taking another $j$-subsequence if necessary, we can define
\[
  \lambda := \wslim_{j\to\infty} \, (1+\abs{u_j})^p \, \Lcal^d \restrict \Omega
  \quad \in \Mbf^+(\cl{\Omega}).
\]
Let $A \subset \cl{\Omega}$ be relatively open in $\cl{\Omega}$. Then, using the same Fourier multiplier estimates as before and standard results in measure theory, we infer
\begin{align}
  \kappa(A) &\leq \liminf_{R\to\infty} \liminf_{j\to\infty} \int_A (1+\abs{u_j})^{p-1}(1+\abs{v_{j,R}}) \dd x  \notag \\
  &\leq \sup_{j,R} \norm{1+\abs{v_{j,R}}}_p \cdot \limsup_{j\to\infty} \, \biggl( \int_{A} (1+\abs{u_j})^p \dd x \biggr)^{(p-1)/p}  \notag \\
  &\leq C \bigl( \supmod_j \norm{1+\abs{u_j}}_p \bigr) \cdot \norm{\Psi}_\Mcal \cdot \lambda(\cl{A})^{(p-1)/p}.  \label{eq:kappa_est}
\end{align}
This estimate implies that $\kappa$ is absolutely continuous with respect to $\lambda$ (note that the Borel sets $A \subset \cl{\Omega}$ with $\lambda(\partial A \setminus \partial \Omega) = 0$ generate the Borel $\sigma$-algebra on $\cl{\Omega}$). Thus, by the Radon--Nikod\'{y}m theorem there exists a density $Q \in \Lrm^1(\cl{\Omega},\lambda)$ such that
\[
  \kappa = Q(x) \, \lambda(\di x).
\]
Denote furthermore the Lebesgue decomposition of $\lambda$ by
\[
  \lambda = R(x) \, \Lcal_x^d \restrict \Omega + \lambda^s,  \qquad\text{where $R \in \Lrm^1(\Omega)$ and $\lambda^s$ is singular to $\Lcal^d$.}
\]
Combining all the previous arguments, we have decomposed $\mu$ as follows:
\begin{align} 
  \mu(\di x \, \di w \, \di r) &= Q(x) \, \lambda(\di x) \otimes \eta_x(\di w) \otimes \sigma_{x,w}(\di r)  \notag \\
  &=  Q(x) \bigl(R(x) \, \Lcal^d_x \restrict \Omega + \lambda^s(\di x) \bigr) \otimes \eta_x(\di w) \otimes \sigma_{x,w}(\di r).   \label{eq:mu_decomp}
\end{align}

Define $V \in \Lrm^1(\cl{\Omega \times \C\B^N},\gamma;\C^N)$, $\nu_x \in \Mbf(\C\B^N;\C^N)$ and also let $\nu_x^\infty \in \Mbf(\partial \C\B^N;\C^N)$ be defined through its action on $\Crm(\partial \C\B^N)$ as follows (using the weak* measurability):
\begin{align*}
  V(x,w) &:= \int_{\cl{\C\B^N}} r \dd \sigma_{x,w}(r), \\
  \dprb{g,\nu_x} &:= Q(x) \, R(x) \, \int_{\C\B^N} (S^{p-1}g) \cdot  V(x,\frarg) \dd \eta_x,
    \hspace{-30pt}&& g \in \Crm_0(\C^N), \\
  \dprb{g^\infty,\nu_x^\infty} &:= Q(x) \, \int_{\partial \C\B^N} g^\infty \cdot V(x,\frarg) \dd \eta_x,
    && g^\infty \in \Crm(\partial \C\B^N).
\end{align*}
Then employ~\eqref{eq:omega_mu_Psi},~\eqref{eq:mu_decomp} and the fact that $S^{p-1}h(x,w) = h^\infty(x,w)$ for all $(x,w) \in \cl{\Omega} \times \partial \C\B^N$, to infer
\begin{align}
  \ddprb{f \otimes \overline{\Psi},\omega} &= \int_{\cl{\Omega}} \int_{\cl{\C\B^N}} \int_{\cl{\C\B^N}} (S^{p-1}h)(x,w) \cdot r \dd \sigma_{x,w}(r) \dd \eta_x(w) \; Q(x) \dd \lambda(x)  \notag \\
  &= \int_{\cl{\Omega}} Q(x) \int_{\C\B^N} (S^{p-1}h)(x,w) \cdot V(x,w) \dd \eta_x(w) \dd \lambda(x)  \notag \\
  &\qquad + \int_{\cl{\Omega}} Q(x) \int_{\partial\C\B^N} (S^{p-1}h)(x,w) \cdot V(x,w) \dd \eta_x(w) \dd \lambda(x)  \notag\\  
  &= \int_{\Omega} \dprb{h(x,\frarg),\nu_x} \dd x
    + \int_{\cl{\Omega}} \dprb{h^\infty(x,\frarg),\nu_x^\infty} \dd \lambda(x)   \label{eq:omega_decomp} \\
  &\qquad + \int_{\cl{\Omega}} Q(x) \int_{\C\B^N} (S^{p-1}h)(x,w) \cdot V(x,w) \dd \eta_x(w) \dd \lambda^s(x). \notag
\end{align}

We will show next that in fact the last integral in~\eqref{eq:omega_decomp} vanishes. To this aim take any $f(x,z,q) := \phi(x) g(z) \cdot q \in \Fbf^p(\Omega;\C^N)$ ($\phi \in \Crm(\cl{\Omega})$, $g \in \Crm(\C^N;\C^N)$) with the property that $S^{p-1}g \in \Crm_c(\C\B^N;\C^N)$ and hence $g^\infty \equiv 0$. The family of functions $\{g(u_j)\}_j$ is uniformly bounded and thus, by virtue of Lemma~\ref{eq:equiint_product}, the functions
\[
  w_{j,R} := g(u_j) \cdot \overline{T_{(1-\eta_R)\Psi}[u_j]}
\]
are ($1$-)equiintegrable (with respect to both $j$ and $R$). Hence, this family is $\Lrm^1$-weakly compact by the Dunford--Pettis theorem. Let $w_{j,R} \toweak w$ in $\Lrm^1(\Omega)$ as $j \to \infty$ and $R \to \infty$ (in this order). Then,
\[
  \ddprb{f \otimes \overline{\Psi},\omega} = \int \phi w \dd x.
\]
This has to equal the expression~\eqref{eq:omega_decomp} for $\ddpr{f \otimes \overline{\Psi},\omega}$ and, differentiating the integral with respect to the domain (varying $\phi$), we get
\[
  Q(x) \int_{\C\B^N} (S^{p-1}g)(w) \cdot V(x,w) \dd \eta_x(w) = 0
  \qquad\text{for $\lambda^s$-a.e.\ $x \in \cl{\Omega}$.}
\]
Finally, varying $S^{p-1} g$ over all of $\Crm_c(\C\B^N;\C^N)$, we may conclude
\[
  Q(x) V(x,w) \, (\eta_x \restrict \C\B^N)(\di w) = 0
  \qquad\text{for $\lambda^s$-a.e.\ $x \in \cl{\Omega}$,}
\]
which shows the claim.

\proofstep{Step 3.}
We now establish that one can replace $\lambda$ by $\lambda_\omega$ as defined in~\eqref{eq:lambda_omega}. For this, observe that if $h(x,z)$ is zero on a neighborhood of $z = 0$, we can replace $\mu_{j,R}^\Psi$ by
\[
  \nu_{j,R}^\Psi := \abs{u_j(x)}^{p-1} \, (1+\abs{v_{j,R}(x)}) \, \Lcal_x^d \restrict \Omega \otimes \delta_{\frac{u_j(x)}{1+\abs{u_j(x)}}} \otimes \delta_{\frac{v_{j,R}(x)}{1+\abs{v_{j,R}(x)}}}.
\]
Analogously to~\eqref{eq:mu_int}, we have for all $f(x,z,q) = h(x,z) \cdot q \in \Fbf^p(\Omega;\C^N)$,
\[
  \int_\Omega h(\frarg,u_j) \cdot \overline{T_{(1-\eta_R)\Psi}[u_j]} \dd x
  = \int \frac{(S^{p-1}h)(x,w)}{\abs{w}^{p-1}}\cdot r \,\dd \nu_{j,R}^\Psi(x,w,r).
\]
Now carry out the same procedure as in the first two steps of the proof; note in passing that~\eqref{eq:kappa_est} still holds with $\lambda$ replaced by $\lambda_\omega$. We define $\bar{\nu}_x, \bar{\nu}_x^\infty$ in the same way as $\nu_x, \nu_x^\infty$. Transforming $(S^{p-1}h)(x,w)/\abs{w}^{p-1}$ via $S^{-(p-1)}$, we arrive at
\[
  \ddprb{f \otimes \overline{\Psi},\omega} = \int_{\Omega} \dprBB{h(x,\frarg) \biggl(\frac{1+\abs{\frarg}}{\abs{\frarg}} \biggr)^{p-1},\bar{\nu}_x} \dd x + \int_{\cl{\Omega}} \dprb{h^\infty(x,\frarg),\bar{\nu}_x^\infty} \dd \lambda_\omega(x)
\]
for different $(\bar{\nu}_x)_x$, $(\bar{\nu}_x^\infty)_x$ of course. We now compare this to~\eqref{eq:omega_decomp} for all integrands $f(x,z,q) = \phi(x)g^\infty(z)(1-\rho(z/K)) \cdot q \in \Fbf^p(\Omega;\C^N)$, where $g^\infty \in \Crm(\C^N \setminus \{0\};\C^N)$ is positively $(p-1)$-homogeneous and $\rho \in \Crm_c^\infty(\R^d;[0,1])$ with $\ONE_{B(0,1)} \leq \rho \leq \ONE_{B(0,2)}$. Letting $K \toup \infty$, we get
\[
  \int_{\cl{\Omega}} \dprb{g^\infty,\nu_x^\infty} \dd \lambda(x)
  = \int_{\cl{\Omega}} \dprb{g^\infty,\bar{\nu}_x^\infty} \dd \lambda_\omega(x)
  \qquad \text{for all $g^\infty \in \Crm(\partial \C\B^N)$.}
\]
Thus, we can just replace $\lambda$ by $\lambda_\omega$ and $(\nu_x^\infty)_x$ by $(\bar{\nu}_x^\infty)_x$ in~\eqref{eq:omega_decomp}; we drop the bar in the following.

\proofstep{Step 4.}
So far the proof has established
\[
  \ddprb{f \otimes \overline{\Psi},\omega} = \int_{\Omega} \dprb{h(x,\frarg),\nu_x} \dd x
    + \int_{\cl{\Omega}} \dprb{h^\infty(x,\frarg),\nu_x^\infty} \dd \lambda_\omega(x)
\]
for all $f \in \Fbf^p(\Omega;\C^N)$ and a fixed $\Psi \in \Mcal$. We now re-introduce the dependence on $\Psi$ and immediately observe that $\nu_x$ and $\nu_x^\infty$ must depend anti-linearly on $\Psi$. Indeed, the left hand side is anti-linear in $\Psi$ by~\eqref{eq:omega} and the measure $\lambda_\omega$ does not depend on $\Psi$ by construction.

Thus, we can in fact write
\[
  \ddprb{f \otimes \overline{\Psi},\omega} = \int_{\Omega} \dprb{h(x,\frarg) \otimes \overline{\Psi}, \omega_x} \dd x
    + \int_{\cl{\Omega}} \dprb{h^\infty(x,\frarg) \otimes \overline{\Psi}, \omega_x^\infty} \dd \lambda_\omega(x),
\]
where now for $\Lcal^d$-a.e.\ fixed $x \in \Omega$, $\omega_x$ is an anti-linear bounded function from $\Mcal$ to the space of $\C^N$-valued Radon measures on $\C\B^N$; for the boundedness we also used the estimate~\eqref{eq:kappa_est}, whereby $Q$ is finite almost everywhere with respect to $\lambda_\omega$. Likewise, for $\lambda_\omega$-a.e.\ $x \in \cl{\Omega}$, $\omega_x^\infty$ is an anti-linear bounded function from $\Mcal$ into the space of $\C^N$-valued Radon measures on the unit sphere $\partial \C\B^N$. This shows properties (i)--(iii) from the definition of MCFs. Property (iv) follows by tracing the weak* measurability properties of the disintegrated measures in the above decomposition, whereas (v) holds by construction. This finishes the proof.
\end{proof}

\begin{remark}
Not every object satisfying Definition~\ref{def:MCF} is generated by at least one sequence, see Example~\ref{ex:MCF_nogen} for a counterexample.
\end{remark}

We close this section by a remark about how \emph{not} to define microlocal compactness forms:

\begin{remark}
Instead of~\eqref{eq:omega}, one might be tempted to define $\omega$ via
\[
  \omega(\phi,h,\Psi) := \lim_{R\to\infty} \lim_{j\to\infty} \int T_{(1-\eta_R)\Psi} \bigl[ \phi h(u_j) \bigr] \dd x,
\]
where $h$ is assumed to have $p$-growth. Unfortunately, $\phi h(u_j)$ in general is only in $\Lrm^1$ and $T_{(1-\eta_R)\Psi}$ is not a bounded operator from $\Lrm^1$ to itself; this follows from the characterization of $\Lrm^1$-Fourier multipliers as \emph{finite} complex-valued measures (see for example Theorem~2.5.8 in~\cite{Graf08CFA}). Thus, for a correct definition we need to weaken the concentration effects inside of $T_\Psi$ and this is exactly what is accomplished by our definition.
\end{remark}

\subsection{Elementary properties}

This section collects a few elementary properties of MCFs, which will be used countless times in the sequel. We start with the following two useful technical lemmas:

\begin{lemma} \label{lem:pq_conc}
Let $(u_j) \subset \Lrm^p(\Omega;\C^N)$ be a norm-bounded sequence that generates a $p$-microlocal compactness form $\omega^p \in \MCF^p(\Omega;\C^N)$, where $\Omega \subset \R^d$ is bounded and $1 < p < \infty$. Then, for all $1 < r < p$ the sequence $(u_j)$ also generates an $r$-microlocal compactness form $\omega^r \in \MCF^r(\Omega;\C^N)$ with
\begin{align*}
  &(\omega^r)_x = (\omega^p)_x  &&\text{for $\Lcal^d$-a.e.\ $x \in \Omega$}
  \qquad\text{and} &&&&&&&&&& \\
  &(\omega^r)_x^\infty = 0 &&\text{for $\lambda_{\omega^r}$-a.e.\ $x \in \cl{\Omega}$.} &&&&&&&&&&
\end{align*}
Furthermore, if $(u_j)$ is norm-bounded in $\Lrm^\infty(\Omega;\C^N)$, then $(\omega^p)_x$ has compact support for $\Lcal^d$-a.e.\ $x \in \Omega$ and $(\omega^p)_x^\infty = 0$ for $\lambda_{\omega^p}$-a.e.\ $x \in \cl{\Omega}$.
\end{lemma}

\begin{proof}
Since $\Lrm^p(\Omega;\C^N)$ embeds into $\Lrm^r(\Omega;\C^N)$ on our bounded $\Omega$, the generation of $\omega^r$ follows directly from Theorem~\ref{thm:omega}. Let $f (x,z,q) = h(x,z) \cdot q \in \Fbf^r(\Omega;\C^N)$ and $\Psi \in \Mcal$. Let us denote by $h^{(r-1)-\infty}(x,z)$ the $(r-1)$-recession function of $h$ as in~\eqref{eq:h_infty}. Clearly, $h^{(p-1)-\infty}$ exists as well and is identically zero. As $f \in\Fbf^r(\Omega;\C^N) \subset \Fbf^p(\Omega;\C^N)$, we immediately get
\begin{align*}
  &\int_\Omega \dprb{h(x,\frarg) \otimes \overline{\Psi},(\omega^r)_x} \dd x + \int_{\cl{\Omega}} \dprb{h^{(r-1)-\infty}(x,\frarg) \otimes \overline{\Psi},(\omega^r)_x^\infty} \dd \lambda_{\omega^r}(x) \\ 
  &\qquad = \ddprb{f \otimes \overline{\Psi},\omega^r}
    = \ddprb{f \otimes \overline{\Psi},\omega^p} \\
  &\qquad= \int_\Omega \dprb{h(x,\frarg) \otimes \overline{\Psi},(\omega^p)_x} \dd x.
\end{align*}
Varying $h$ and $\Psi$, this implies the conclusions of the lemma. The additional statement is evident from the definition since for $f(x,z,q) = h(x,z) \cdot q$ such that $\supp h(x,\frarg) \subset\subset \C^N \setminus R \cl{\C\B^N}$, where $R := \supmod_j \norm{u_j}_\infty$, we have $\ddprb{f \otimes \overline{\Psi},\omega^p} = 0$.
\end{proof}

\begin{lemma} \label{lem:countable_test}
There exists a countable set $\{f_m\}_{m \in \N} \subset \Fbf^p(\Omega;\C^N)$ with
\[
  f_m(x,z,q) = \phi_m(x) g_m(z) \cdot q
  \qquad\text{where $\phi_m \in \Crm^\infty(\cl{\Omega})$, $g_m \in \Crm^\infty(\C^N;\C^N)$,}
\]
and a countable set $\{\Psi_n\}_{n \in \N} \subset \Mcal$ with $\Psi_n \in \Crm^\infty(\Sbb^{d-1};\C^{N \times N})$, such that for any two MCFs $\omega_1, \omega_2 \in \MCF^p(\Omega;\C^N)$,
\[
  \ddprb{f_m \otimes \overline{\Psi}_n,\omega_1} = \ddprb{f_m \otimes \overline{\Psi}_n,\omega_2} \quad\text{for all $m,n \in \N$}
  \qquad\text{implies}\qquad
  \omega_1 = \omega_2.
\]
\end{lemma}

\begin{proof}
In view of the linearity and the fundamental estimate~\eqref{eq:omega_est}, it suffices to construct dense subsets $\{f_m\}_{m \in \N} \subset \Fbf^p(\Omega;\C^N)$ and $\{\Psi_n\}_{n \in \N} \subset \Mcal$ with the stated properties. For the second set we may simply use a smooth dense subset of $\Crm^{\floor{d/2} + 1}(\Sbb^{d-1};\C^{N \times N})$. For the first set, we choose  a countable dense subset $\{\phi_m\}_{m \in \N} \subset \Crm(\cl{\Omega})$ and a countable dense subset $\{\tilde{g}_m\}_{m \in \N} \subset \Crm(\cl{\C\B^N};\C^N)$ consisting only of smooth functions. Then, define $g_m := S^{-(p-1)} \tilde{g}_m$ and set $f_m(x,z,q) := \phi_m(x) g_m(z) \cdot q$.
\end{proof}

The next result concerns the choice of \enquote{good} generating sequences:

\begin{lemma} \label{lem:better_generation}
Let $(u_j) \subset \Lrm^p(\Omega;\C^N)$ generate $\omega \in \MCF^p(\Omega;\C^N)$. Then, there exists a sequence $(v_j) \subset (\Lrm^p \cap \Crm_c^\infty)(\Omega;\C^N)$ that also generates $\omega$.
\end{lemma}

\begin{proof}
Take countable families $\{f_m(x,z,q) = h_m(x,z) \cdot q \}_{m \in \N} \subset \Fbf^p(\Omega;\C^N)$ and $\{\Psi_n\}_{n \in \N} \subset \Mcal$ as in Lemma~\ref{lem:countable_test}. For each fixed $j \in \N$ let $(v_j^{(k)})_k \in (\Lrm^p \cap \Crm_c^\infty)(\Omega;\C^N)$ with $v_j^{(k)} \to u_j$ in $\Lrm^p(\Omega;\C^N)$ and almost everywhere as $k \to \infty$; this can be constructed by mollification. We can furthermore require
\[
  \absBB{\int_\Omega h_m(\frarg,u_j) \cdot \overline{T_{(1-\eta_R)\Psi_n}[u_j]} \dd x - \int_\Omega h_m(\frarg,v_j^{(n)}) \cdot \overline{T_{(1-\eta_R)\Psi_n}[v_j^{(n)}]} \dd x} \leq \frac{1}{j}
\]
whenever $m,n \leq j$. Indeed, the preceding expression can be estimated by
\begin{align*}
  &\absBB{\int_\Omega \bigl[h_m(\frarg,u_j) - h_m(\frarg,v_j^{(k)})\bigr] \cdot \overline{T_{(1-\eta_R)\Psi_n}[u_j]} \\
  &\qquad\qquad + h_m(\frarg,v_j^{(k)}) \cdot \overline{T_{(1-\eta_R)\Psi_n}[u_j-v_j^{(k)}]} \dd x} \\
  &\qquad \leq C \normb{h_m(\frarg,u_j)-h_m(\frarg,v_j^{(k)})}_{p/(p-1)}
    + C \normb{T_{(1-\eta_R)\Psi_n}[u_j-v_j^{(k)}]}_p.
\end{align*}
The first term goes to zero as $k \to \infty$ by Pratt's Theorem since $v_j^{(k)} \to u_j$ a.e.\ and we have the converging $p/(p-1)$-integrable majorants $C (1+\abs{u_j}+\abs{v_j^{(k)}})^{p-1}$. The second expression goes to zero by a standard multiplier estimate. We conclude the proof via a standard diagonal argument.
\end{proof}

Next, the Commutation Lemma~\ref{lem:commutation} implies the following important technical result:

\begin{lemma} \label{lem:phi_exchange}
Let $(u_j) \subset \Lrm^p(\Omega;\C^N)$ generate $\omega \in \MCF^p(\Omega;\C^N)$. Then, for all $f \in \Fbf^p(\Omega;\C^N)$ with $f(x,z,q) = h(x,z) \cdot q$, all $\Psi \in \Mcal$, and all $\phi \in \Crm_c(\R^d)$, it holds that
\begin{align*}
  \ddprb{\phi f \otimes \overline{\Psi},\omega} &= \lim_{R\to\infty} \lim_{j\to\infty} \int_\Omega \phi h(\frarg,u_j) \cdot \overline{T_{(1-\eta_R)\Psi}[u_j]} \dd x \\
  &= \lim_{R\to\infty} \lim_{j\to\infty} \int_\Omega h(\frarg,u_j) \cdot \overline{T_{(1-\eta_R)\Psi}[\overline{\phi} u_j]} \dd x.
\end{align*}
\end{lemma}

\begin{proof}
It suffices to show that
\[
  \lim_{R\to\infty} \lim_{j\to\infty} \int_\Omega h(\frarg,u_j) \cdot \overline{\bigl[\overline{\phi}(x) T_{(1-\eta_R)\Psi}[u_j] - T_{(1-\eta_R)\Psi}[\overline{\phi} u_j]\bigr]} \dd x = 0.
\]
Without loss of generality assume that $h(\frarg,u_j) \toweak H$ in $\Lrm^{p/(p-1)}(\Omega;\C^N)$ and that $u_j \toweak u$ in the space $\Lrm^p(\Omega;\C^N)$. This may involve taking a subsequence, but since we will prove the convergence to zero for all such subsequences, the result then also applies to the original sequence of $j$'s. So, the Commutation Lemma~\ref{lem:commutation} (which also holds with the symbol $(1-\eta_R)\Psi$ instead of $\Psi$) implies
\begin{align*}
  &\overline{\phi} T_{(1-\eta_R)\Psi}[u_j] - T_{(1-\eta_R)\Psi}[\overline{\phi} u_j] \\
  &\qquad = \overline{\phi} T_{(1-\eta_R)\Psi}[u_j - u] - T_{(1-\eta_R)\Psi}[\overline{\phi} (u_j-u)] \\
  &\qquad\qquad + \overline{\phi} T_{(1-\eta_R)\Psi}[u] - T_{(1-\eta_R)\Psi}[\overline{\phi} u] \\
  &\qquad \to 0 + \overline{\phi} T_{(1-\eta_R)\Psi}[u] - T_{(1-\eta_R)\Psi}[\overline{\phi} u]
    \qquad\text{as $j \to \infty$.}
\end{align*}
By Lemma~\ref{lem:TetaR_conv}, the remainder
\[
  \int H \cdot \overline{\bigl[ \overline{\phi} T_{(1-\eta_R)\Psi}[u] - T_{(1-\eta_R)\Psi}[\overline{\phi} u] \bigr]} \dd x
\]
converges to zero as $R\to\infty$.
\end{proof}

The last result in particular implies:

\begin{lemma}[Locality] \label{lem:locality}
Let the sequences $(u_j), (v_j) \subset \Lrm^p(\Omega;\C^N)$ generate $\omega_1, \omega_2 \in \MCF^p(\Omega;\C^N)$, respectively. If $u_j = v_j$ on an open subdomain $D \subset\subset \Omega$ then for all $\phi \in \Crm_c(D)$, $f \in \Fbf^p(\Omega;\C^N)$, $\Psi \in \Mcal$,
\[
  \ddprb{\phi f \otimes \overline{\Psi}, \omega_1} = \ddprb{\phi f \otimes \overline{\Psi}, \omega_2},
\]
and we write \enquote{$\omega_1 = \omega_2$ on $D$}.
\end{lemma}

\begin{proof}
Assume without loss of generality that $\phi \geq 0$, otherwise consider separately the regions where $\phi$ has one sign. Then, by the previous lemma,
\begin{align*}
  \ddprb{\phi f \otimes \overline{\Psi},\omega_1} &= \lim_{R\to\infty} \lim_{j\to\infty} \int_\Omega \phi h(\frarg,u_j) \cdot \overline{T_{(1-\eta_R)\Psi}[u_j]} \dd x \\
  &= \lim_{R\to\infty} \lim_{j\to\infty} \int_\Omega \sqrt{\phi} h(\frarg,u_j) \cdot \overline{T_{(1-\eta_R)\Psi}[\sqrt{\phi} u_j]} \dd x \\
  &= \lim_{R\to\infty} \lim_{j\to\infty} \int_\Omega \sqrt{\phi} h(\frarg,v_j) \cdot \overline{T_{(1-\eta_R)\Psi}[\sqrt{\phi} v_j]} \dd x \\
  &= \cdots = \ddprb{\phi f \otimes \overline{\Psi},\omega_2},
\end{align*}
which proves the claim.
\end{proof}

The following lemma will be used many times in the sequel:

\begin{lemma} \label{lem:eliminate_Rlim}
Let $(u_j) \subset \Lrm^p(\Omega;\C^N)$ generate $\omega \in \MCF^p(\Omega;\C^N)$ and also assume $u_j \toweak u$ in $\Lrm^p(\Omega;\C^N)$. Then, for all $f \in \Fbf^p(\Omega;\C^N)$ with $f(x,z,q) = h(x,z) \cdot q$, and all $\Psi \in \Mcal$, it holds that
\[
  \ddprb{f \otimes \overline{\Psi},\omega} = \lim_{j\to\infty} \int_\Omega h(\frarg,u_j) \cdot \overline{T_\Psi[u_j-u]} \dd x.
\]
\end{lemma}

\begin{proof}
We have
\begin{align*}
  \ddprb{f \otimes \overline{\Psi},\omega} &= \lim_{R\to\infty} \lim_{j\to\infty} \int_\Omega h(\frarg,u_j) \cdot \overline{T_{(1-\eta_R)\Psi}[u_j]} \dd x \\
  &= \lim_{j\to\infty} \int_\Omega h(\frarg,u_j) \cdot \overline{T_\Psi[u_j-u]} \dd x\\
  &\qquad + \lim_{R\to\infty} \lim_{j\to\infty} \int_\Omega h(\frarg,u_j) \cdot \overline{T_{(1-\eta_R)\Psi}[u]} \dd x,
\end{align*}
because $T_{\eta_R\Psi}[u_j] \to T_{\eta_R\Psi}[u]$ strongly in $\Lrm^p(\Omega;\C^N)$, see Lemma~\ref{lem:bandlim_compact}. If $H \in \Lrm^{p/(p-1)}(\Omega;\C^N)$ is the weak limit of $h(\frarg,u_j)$ (up to a subsequence),
\[
  \lim_{R\to\infty} \lim_{j\to\infty} \int_\Omega h(\frarg,u_j) \cdot \overline{T_{(1-\eta_R)\Psi}[u]} \dd x = \lim_{R\to\infty} \int_\Omega H \cdot \overline{T_{(1-\eta_R)\Psi}[u]} \dd x = 0,
\]
where we also used $T_{(1-\eta_R)\Psi}[u] \toweak 0$ for $R\to\infty$, see Lemma~\ref{lem:TetaR_conv}, to conclude.
\end{proof}

The final result in this section shows that $\omega = 0$ is equivalent to strong convergence of any generating sequence, which expresses that the MCF quantifies the difference between weak and strong convergence.

\begin{lemma} \label{lem:omega0_strong_conv}
Let $(u_j) \subset \Lrm^p(\Omega;\C^N)$ generate $\omega \in \MCF^p(\Omega;\C^N)$ and also $u_j \toweak u$ in $\Lrm^p(\Omega;\C^N)$. Then $u_j \to u$ strongly if and only if $\omega = 0$
\end{lemma}

\begin{proof}
Let $f \in \Fbf^p(\Omega;\C^N)$ and $\Psi \in \Mcal$. If $u_j \to u$ strongly in $\Lrm^p(\Omega;\C^N)$, then by the preceding lemma,
\[
  \ddprb{f \otimes \overline{\Psi},\omega} = \lim_{j\to\infty} \int_\Omega h(\frarg,u_j) \cdot \overline{T_{\Psi}[u_j-u]} \dd x = 0.
\]

Now assume $\omega = 0$. According to Proposition~\ref{prop:equiint} in Section~\ref{ssc:conc} below we have that $(u_j)$ is $p$-equiintegrable. Define $u^{R} := \ONE_{\{\abs{u}<R\}} \, u$ and set
\[
  f_R(x,z,q) := \abs{z-u^{R}(x)}^{p-2}(z-u^{R}(x)) \cdot q.
\]
Further, by Lemma~\ref{lem:eliminate_Rlim} and the extended representation result stated in Proposition~\ref{prop:ext_repr} in the next section,
\begin{align*}
  0 = \ddprb{f_R \otimes I,\omega} &= \limsup_{j\to\infty} \int_\Omega \abs{u_j-u^{R}}^{p-2} \,  (u_j-u^{R}) \cdot \overline{(u_j-u)}\dd x \\
  &= \limsup_{j\to\infty} \norm{u_j-u^{R}}_{\Lrm^p(\{\abs{u}<R\})}^p + \SmallO(1/R),
\end{align*}
where $\SmallO(1/R)$ is a term that vanishes as $R \toup \infty$; this comes from the $p$-equiintegrability of $(u_j)$ whereby the integrand in the first line is also equiintegrable and $\abs{\{\abs{u} \geq R\}} \todown 0$ as $R \toup \infty$. Thus, using the equiintegrability again, $\lim_{j\to\infty} \norm{u_j-u}_p^p = 0$.
\end{proof}

\subsection{Extended representation, bilinear forms and compound symbols} \label{ssc:ext_repr_distrib}

In this section we show how MCFs allow us to represent limits for a wider class of integrands, consider MCFs as bilinear forms, and we also briefly touch on the complicated topic of \enquote{compound} symbols, i.e.\ symbols that cannot be written as a tensor product $f \otimes \overline{\Psi}$.

We first define the set of \term{representation integrands}
\begin{align*}
  \Rbf(\Omega;\C^N) := \setb{ &f \colon \cl{\Omega} \times \C^N \times \C^N \to \C }{ \text{$f(x,z,q) = h(x,z) \cdot q$ and}\\
  &\qquad\text{$h$ Carath\'{e}odory with at most $(p-1)$-growth at infinity} \\
  &\qquad\text{and $h^\infty \in \Crm(\cl{\Omega} \times \C^N;\C^N)$ exists in the sense of~\eqref{eq:h_infty} } }.
\end{align*}
In this context recall that a function $h \colon \cl{\Omega} \times \C^N \to \C$ is Carath\'{e}odory if $h(\frarg,z)$ is $\Lcal^d$-measurable for all fixed $z \in \C^N$ and $h(z,\frarg)$ is continuous for all fixed $x \in \cl{\Omega}$. The main result for this class of integrands is that even if only $f \in \Rbf(\Omega;\C^N)$, the representation of limits via the MCF still holds:

\begin{proposition} \label{prop:ext_repr}
Let $(u_j) \subset \Lrm^p(\Omega;\C^N)$ generate $\omega \in \MCF^p(\Omega;\C^N)$. Then, also for $f \in \Rbf(\Omega;\C^N)$ and $\Psi \in \Mcal$ it holds that
\begin{align*}
  &\lim_{R\to\infty} \lim_{j\to\infty} \int_\Omega h(\frarg,u_j) \cdot \overline{T_{(1-\eta_R)\Psi}[u_j]} \dd x \\
  &\qquad = \ddprb{f \otimes \overline{\Psi},\omega} \\
  &\qquad = \int_\Omega \dprb{h(x,\frarg) \otimes \overline{\Psi},\omega_x} \dd x + \int_{\cl{\Omega}} \dprb{h^\infty(x,\frarg) \otimes \overline{\Psi},\omega_x^\infty} \dd \lambda_\omega(x).
\end{align*}
\end{proposition}

\begin{proof}
The fact that the expression on the right hand side is well-defined and equals the left hand side can be proved by a strategy similar to the one employed in the proof of the corresponding result for generalized Young measures, see Proposition~2 in~\cite{KriRin10CGGY}.
\end{proof}

Next, we investigate how an MCF can be considered a bilinear form: We start with the Hilbert space case $p=2$, where the strongest assertions can be made (as is to be expected): Since all $\Lrm^\infty$-functions give rise to a bounded $\Lrm^2$-Fourier multiplier, we get the following improved basic estimate (cf.~\eqref{eq:omega_est}):
\[
  \absb{\ddprb{f \otimes \overline{\Psi},\omega}} \leq C \bigl( \supmod_j \norm{1 + \abs{u_j}}_2^2 \bigr) \cdot \norm{f}_{\Fbf^2} \cdot \norm{\Psi}_\infty
\]
and we can allow $\Psi \in \Crm(\Sbb^{d-1};\C^{N \times N})$. This shows that $\omega$ can be seen as a bilinear form (see~\eqref{eq:Eh} for the definition of $Eh$)
\begin{align*}
  &\omega \colon \Crm(\cl{\Omega} \times \sigma\C^N;\C^N) \times \Crm(\Sbb^{d-1};\C^{N \times N}) \to \C, \\
  &\omega(Eh,\overline{\Psi}) := \ddprb{ f \otimes \overline{\Psi}, \omega}.
\end{align*}
Note that we cannot allow $\Psi \in \Lrm^\infty(\Sbb^{d-1};\C^{N \times N})$ because there is no countable dense subset, which we need to construct a generic limit for all $\Psi$.

By a similar argument for $p>2$ and using the embedding $\Lrm^p(\Omega;\C^N) \embed \Lrm^2(\Omega;\C^N)$ for bounded $\Omega$, we can furthermore consider $\omega$ as a bilinear form 
\[
  \omega \colon X^p \times \Crm(\Sbb^{d-1};\C^{N \times N}) \to \C,
\]
where
\[
  X^p := \setB{ Eh \in \Crm(\cl{\Omega} \times \sigma\C^N;\C^N) }{ \text{$\abs{Eh(x,z)} \leq C(1 + \abs{z})^{(2-p)/2}$, $C > 0$} }.
\]
Note that here we need $(p/2)$-growth of $h$, translating into the above condition on $Eh$, in order for the integral in the definition of the MCF $\omega$ to be bounded.

For $p<2$, however, more regularity in $\Psi$ beyond mere continuity seems to be necessary.

We now turn to the question of \enquote{compound} symbols. Let $f \in \Fbf^p(\Omega;\C^N)$ with $f(x,z,q) = h(x,z) \cdot q$. For a sequence $(u_j) \subset (\Lrm^1 \cap \Lrm^p)(\Omega;\C^N)$ with the property that additionally $\hat{u}_j \in (\Lrm^1 \cap \Lrm^{p'})(\Omega;\C^N)$ for all $j \in \N$ ($1/p + 1/p' = 1$), we can write
\begin{align*}
  &\ddprb{f \otimes \overline{\Psi},\omega} = \lim_{R\to\infty} \lim_{j\to\infty} \int_\Omega h(\frarg,u_j) \cdot \overline{T_{(1-\eta_R)\Psi}[u_j]} \dd x \\
  &= \lim_{R\to\infty} \lim_{j \to \infty} \int_{\R^d} \int_\Omega \int_\Omega (1-\eta_R(\xi))\Psi^* \biggl(\frac{\xi}{\abs{\xi}}\biggr) h(x,u_j(x)) \cdot \overline{u_j(y)} \, \ee^{2\pi\ii (y-x) \cdot \xi} \dd y \dd x \dd \xi,
\end{align*}
see~\eqref{eq:full_integral}. One can therefore consider
\[
  \Psi^* \biggl(\frac{\xi}{\abs{\xi}}\biggr) h(x,z) \cdot q
\]
to be the \enquote{symbol} of the above double Fourier integral (similarly to the theory of pseudo-differential operators).

We can thus extend the class of admissible \enquote{symbols} $f \otimes \overline{\Psi}$ to all $F = F(x,z,q,\xi)$ of the form
\[
  F(x,z,q,\xi) = \Psi^* \biggl(x, \frac{\xi}{\abs{\xi}}\biggr) h(x,z) \cdot q
\]
for $h \colon \cl{\Omega} \times \C^N \to \C^N$ a Carath\'{e}odory function and $\Psi \in \Crm(\cl{\Omega};\Mcal)$ such that
\[
  \Psi^* \biggl(x, \frac{\xi}{\abs{\xi}}\biggr) h^\infty(x,z) \cdot q
\]
is continuous. We call these symbols $F$ \term{compound symbols} and define
\[
  \ddprb{F,\omega} := \int_\Omega \dprb{h(x,\frarg) \otimes \overline{\Psi}(x,\frarg),\omega_x} \dd x + \int_{\cl{\Omega}} \dprb{h^\infty(x,\frarg) \otimes \overline{\Psi}(x,\frarg),\omega_x^\infty} \dd \lambda_\omega(x).
\]

The question whether one can define $\omega$ for even more general compound symbols depending on $x$, $z$, and $\xi/\abs{\xi}$, and how much smoothness is minimally necessary, is rather intricate. In this context it is well known that for symbols of pseudo-differential operators some smoothness is indispensable, even for $(\Lrm^2\to\Lrm^2)$-boundedness; a classical example for this can be found in Chapter~VII of~\cite{Stei93HA}. This necessary regularity also persists for the subclass of symbols that are positively $0$-homogeneous in $\xi$, see Theorem~3 and Proposition~1 of~\cite{Stef10PORS}.

A first, rather abstract, approach to such compound symbols in $x,z,\xi/\abs{\xi}$ is through the theory of tensor products of Banach spaces, see for example~\cite{Ryan02ITPB} or Chapter~VIII of~\cite{DieUhl77VM}. As it turns out there are many ways to define the tensor product of two Banach spaces $X,Y$. For our purposes the most relevant is the \term{projective tensor product} $X \otimes_\pi Y$, which is the completion of $X \otimes Y$ with respect to the \term{projective crossnorm}
\[
  \norm{x}_\pi := \inf \, \setBB{ \sum_{i=1}^n \norm{x_i} \norm{y_i} }{ x = \textstyle\sum_{i=1}^n x_i \otimes y_i }.
\]
With this largest reasonable crossnorm, $X \otimes_\pi Y$ is the smallest tensor product completion of $X \otimes Y$ that yields a Banach space structure. The projective tensor product $X \otimes_\pi Y$ has the special property that its dual $(X \otimes_\pi Y)^*$ is simply the space of bounded bilinear forms on $X \times Y$ and therefore is ideally suited to our purposes. 

In our concrete situation, we have $X = \Crm(\C^N;\C^N)$, $Y = \overline{\Mcal}$ (we use $\overline{\Psi}$ instead of $\Psi$ because of the anti-linearity), and
\begin{align*}
  \omega_x &\in (\Crm(\C^N;\C^N) \otimes_\pi \overline{\Mcal})^*
    \qquad\text{for $\Lcal^d$-almost every $x \in \cl{\Omega}$,} \\
  \omega_x^\infty &\in (\Crm(\partial \C\B^N;\C^N) \otimes_\pi \overline{\Mcal})^*
    \qquad\text{for $\lambda_\omega$-almost every $x \in \cl{\Omega}$.}
\end{align*}
If it is at all possible to define the pairing between $\omega \in \MCF^p(\Omega;\C^N)$ and a compound symbol
\[
  F(x,z,q,\xi) = K(x,z,\xi) \cdot q
  \quad\in\quad  \Crm(\cl{\Omega} \times \C^N \times \C^N \times \Sbb^{d-1};\C),
\]
for which $S^{p-1} K \in \Crm(\cl{\Omega \times \C\B^N} \times \C^N \times \Sbb^{d-1};\C)$, then, for reasons of consistency (cf.~\eqref{eq:omega_ddpr}), this pairing must be
\[
  \ddprb{F,\omega} = \int_\Omega \omega_x \bigl( K(x,\frarg,\frarg) \bigr) \dd x
  + \int_{\cl{\Omega}} \omega_x \bigl( K^\infty(x,\frarg,\frarg) \bigr) \dd \lambda_\omega(x),
\]
where $K^\infty$ is again the $(p-1)$-recession function, defined just like in~\eqref{eq:h_infty}. Then, the above discussion shows that $F$ is an admissible symbol if
\[
  S^{p-1}K(x,\frarg,\frarg) \in \Crm(\cl{\C\B^N};\C^N) \otimes_\pi \overline{\Mcal}
  \qquad\text{for $(\Lcal^d + \lambda_\omega)$-a.e.\ $x \in \cl{\Omega}$.}
\]

Finally, we show how an MCF $\omega \in \MCF^p(\Omega;\C^N)$ can be considered as a distribution, and this will give a second way to treat compound symbols in $x,z,\xi/\abs{\xi}$. Recall the definition of the sphere compactification $\sigma\C^N$ of $\C^N$ in Section~\ref{ssc:sphere_compact} and that of $Eh$ in~\eqref{eq:Eh}, and for the moment consider $\omega$ as a bilinear form taking as its arguments $Eh \in \Crm(\cl{\Omega} \times \sigma\C^N;\C^N)$ and $\overline{\Psi} \in \overline{\Mcal}$. Of course, we may restrict to the subspace of such $Eh$ with infinitely many derivatives and then we can consider $\omega$ as a bilinear form
\begin{align*}
  &\omega \colon \Crm^\infty(\cl{\Omega} \times \sigma\C^N;\C^N) \times \Crm^\infty(\Sbb^{d-1};\C^{N \times N}) \to \C, \\
  &\omega(Eh,\overline{\Psi}) := \ddprb{ f \otimes \overline{\Psi}, \omega}.
\end{align*}
Here, $Eh \in \Crm^\infty(\cl{\Omega} \times \sigma\C^N;\C^N)$ means that $S^{p-1} h \in \Crm^\infty(\cl{\Omega \times \C\B^N};\C^N)$, with one-sided derivatives on $\partial \C\B^N$. In view of the Schwartz kernel theorem, see for instance Chapter~V in~\cite{Horm90ALPD1} or Section~4.6 in~\cite{Tayl11PDE1}, we infer that $\omega$ extends to a (complex-valued) distribution
\[
  \tilde{\omega} \in \Crm^\infty(\cl{\Omega} \times \sigma\C^N \times \Sbb^{d-1};\C^N \times \C^{N \times N})^*.
\]

Using the distribution $\tilde{\omega}$, we can \enquote{extend} the representation to the symbol class $\Crm^\infty(\cl{\Omega} \times \sigma\C^N \times \Sbb^{d-1};\C^N \times \C^{N \times N})$. This approach, however, has a significant drawback, namely that now some $x$-regularity is genuinely needed (as can be seen from the proofs of the Schwartz kernel theorem), so the interpretation of the duality pairing between a symbol and an MCF is less clear.

\subsection{Relationship with Young measures and H-measures} \label{ssc:MCF_YM_HM}

We noted in the introduction that the microlocal compactness form generated by a sequence contains (essentially) all the information of the Young measure and the H-measure (or microlocal defect measure). In this section, we make these statements precise and give more or less explicit ways of extracting the Young measure and H-measure. In the following, $\Omega$ is always a \emph{bounded} Lipschitz domain.

\term{Tartar's H-measure}~\cite{Tart90HMNA} (also see~\cite{Tart09GTH} for variants and historical information) or \term{G\'{e}rard's microlocal defect measure}~\cite{Gera91MDM} generated by a sequence $(u_j) \subset \Lrm^2(\Omega;\C^N)$ with $u_j \toweak 0$ is the matrix-valued measure $\mu = (\mu^k_l) \in \Mbf(\cl{\Omega} \times \Sbb^{d-1};\C^{N \times N})$, where like for ordinary matrices $k$ and $l$ denote the row and column index, respectively, defined via
\[
  \int \phi \otimes \overline{\psi} \dd \mu^k_l := \lim_{j\to\infty} \int_\Omega \Fcal[\phi_1 u_j^k](\xi) \cdot \overline{\Fcal[\phi_2 u_j^l](\xi) \cdot \psi \biggl( \frac{\xi}{\abs{\xi}} \biggr)} \dd \xi
\]
for all $\phi = \phi_1 \overline{\phi}_2$ with $\phi_1, \phi_2 \in \Crm(\Omega)$, all $\psi \in \Crm(\Sbb^{d-1};\C)$, and all $k,l = 1,\ldots,N$.

The fact that the H-measure is completely contained in the microlocal compactness form is essentially a triviality and follows directly from the definition by restricting to test functions of the form $f(x,z,q) = \phi_1(x)\overline{\phi_2}(x) z \cdot q$, taking into account Remark~\ref{rem:omega_alt}, in particular~\eqref{eq:MCF_L2}, and Lemma~\ref{lem:eliminate_Rlim}. Thus, for $\Psi = (\Psi^k_l) \in \Mcal$ (or only $\Psi \in \Crm(\Sbb^{d-1};\C^N)$, see the previous section) we get
\begin{align*}
  \int \phi \otimes \overline{\Psi} \dd \mu &= \sum_{k,l} \int \phi \, \overline{\Psi^k_l} \dd \mu^k_l \\
  &= \lim_{j\to\infty} \, \sum_{k,l} \int_\Omega \Fcal[\phi_1 u_j^k](\xi) \cdot \overline{\Fcal[\phi_2 u_j^l](\xi) \cdot \Psi^k_l \biggl( \frac{\xi}{\abs{\xi}} \biggr)} \dd \xi \\
  &= \lim_{j\to\infty} \, \int_\Omega \Fcal[\phi_1 u_j](\xi) \cdot \overline{\Psi \biggl( \frac{\xi}{\abs{\xi}} \biggr) \Fcal[\phi_2 u_j](\xi)} \dd \xi \\
  &= \ddprb{\phi \otimes \overline{\Psi}, \omega},
\end{align*}
where $\omega \in \MCF^2(\Omega;\C^{N \times N})$ is the MCF generated by $(u_j)$.

In the same vein, we can also write for $f(x,z,q) = h(x,z) \cdot q \in \Fbf^2(\Omega;\C)$ and $\psi \in \Crm(\Sbb^{d-1})$ (for notational simplicity we only consider the case $N=1$)
\[
  \ddprb{f \otimes \overline{\psi}, \omega} = \int \overline{\psi} \dd \gamma^1_2,
\]
where $\gamma = (\gamma^i_j)$ ($i,j = 1,2$) now is the H-measure associated with (a subsequence of) the sequence $(h(\frarg,u_j), u_j)_j$.

For the correspondence of the MCF with the (generalized) Young measure, let us first recall some basic facts and fix notation, see~\cite{DiPMaj87OCWS,AliBou97NUIG,KruRou97MDM,Sych00CHGY,KriRin10CGGY,Rind12LSYM,Rind14LPCY} for details: The \term{(generalized) Young measure} $\nu = (\nu_x,\lambda_\nu,\nu_x^\infty) \in \Ybf^p(\Omega;\C^N)$ consists of the \term{oscillation measure} $(\nu_x)_{x\in\Omega} \subset \Mbf^1(\C^N)$, which is a parametrized (and weakly*-measurable) family of probability measures on $\C^N$, the \term{concentration measure} $\lambda_\nu \in \Mbf(\cl{\Omega})$, and the \term{concentration-angle measure} $(\nu_x)_{x\in\Omega} \subset \Mbf^1(\partial \C\B^N)$, which is a $\lambda_\nu$-weakly* measurable family of probability measures on the unit sphere $\partial \C\B^N$ of $\C^N$. If the Young measure $\nu$ is generated by a sequence $(u_j) \subset \Lrm^p(\Omega;\C^N)$, then we have the representation of limits (cf.\ Theorem~7 in~\cite{KriRin10CGGY}, and~\cite{KruRou97MDM} for the case $p > 1$)
\begin{align*}
  \lim_{j\to\infty} \int_\Omega F(\frarg,u_j) \dd x &= \ddprb{F,\nu} \\
  &= \int_\Omega \int F(x,\frarg) \dd \nu_x \dd x + \int_{\cl{\Omega}} \int F^\infty(x,\frarg) \dd \nu_x^\infty \dd \lambda_\nu(x)
\end{align*}
for all $F \in \Crm(\cl{\Omega} \times \C^N)$ with the property that (cf.~\eqref{eq:Sp})
\[
  S^p F \in \Crm(\cl{\Omega \times \C\B^N}),
\]
whereby $F^\infty$ is well-defined in the sense
\[
  F^\infty(x,z) := \lim_{\substack{\!\!\!\! x' \to x \\ \!\!\!\! z' \to z \\ \; t \to \infty}} \frac{F(x',tz')}{t^p},
  \qquad x \in \cl{\Omega}, \; z \in \C^N.
\]
The classical Young measure is just the oscillation part $(\nu_x)_x$ of the generalized Young measure; in this case the limit representation above only holds if $F(\frarg,u_j)$ is an equiintegrable sequence (in this case $\lambda_\nu = 0$). The fact that the (generalized) Young measure can be computed from the MCF is contained in the following proposition.

\begin{proposition} \label{prop:MCF_YM}
Let $(u_j) \subset \Lrm^p(\Omega;\C^N)$ generate the microlocal compactness form $\omega \in \MCF^p(\Omega;\C^N)$ and also the (generalized) Young measure $\nu \in \Ybf(\Omega;\C^N)$. Assume further that $u_j \toweak u$. Then, the knowledge of $\omega$ and $u$ completely determines $\nu$.
\end{proposition}

\begin{remark}
The fact that besides $\omega$ we also need to know $u(x) = \int z \dd \nu_x(z)$ ($x \in \cl{\Omega}$) is owed to the definition of MCFs being tailored to quantify the \emph{difference} between weak and strong convergence (or compactness). For instance, any strongly converging sequence generates the zero MCF, but obviously in general not the same Young measure.
\end{remark}

\begin{proof}
As recalled before, the generalized Young measure $\nu$ allows us to compute the limits
\[
  \ddprb{F,\nu} = \lim_{j\to\infty} \int_\Omega F(\frarg,u_j) \dd x
\]
for all integrands $F$ as above. It is well known in Young measure theory that all these limits together also determine $\nu$ (for the classical Young measure we only need to test with $F \in \Crm_c(\Omega \times \C^N;\C)$, so that $(F(\frarg,u_j))$ is an equiintegrable sequence). We need to show that all these limits can be expressed through $\omega$ and $u$.

Let $\eps > 0$ and take a cut-off function $\rho_\eps \in \Crm_c^\infty(\C^N)$ such that
\[
  \text{$\rho_\eps \equiv 1$ on $B(0,\eps)$,}  \qquad \supp \rho_\eps \subset B(0,2\eps),
\]
Write
\begin{align*}
  F(x,z) &=\phantom{:} \rho_\eps(z-u(x)) F(x,z) \\
         &\qquad + \biggl[ \bigl(1-\rho_\eps(z-u(x))\bigr) F(x,z) \frac{z-u(x)}{\abs{z-u(x)}^2} \biggr] \cdot \overline{(z-u(x))} \\
  &=: G_\eps(x,z) + h_\eps(x,z) \cdot \overline{(z-u(x))}.
\end{align*}
Then, $f_\eps(x,z) := h_\eps(x,z) \cdot q \in \Rbf^p(\Omega;\C^N)$ and
\begin{equation} \label{eq:YM_repr_split}
  \ddprb{F,\nu} = \lim_{j\to\infty} \int_\Omega G_\eps(\frarg,u_j) \dd x + \lim_{j\to\infty} \int_\Omega h_\eps(\frarg,u_j) \cdot \overline{(u_j-u)} \dd x. 
\end{equation}
By Proposition~\ref{prop:ext_repr} in conjunction with Lemma~\ref{lem:eliminate_Rlim}, we can express the second limit through $\omega$ as follows:
\[
  \lim_{j\to\infty} \int_\Omega h_\eps(\frarg,u_j) \cdot \overline{(u_j-u)} \dd x = \ddprb{f_\eps \otimes I, \omega}.
\]
For the first limit in~\eqref{eq:YM_repr_split} we get, exploiting the continuity of $F$ in its second argument,
\begin{align*}
  \lim_{j\to\infty} \int_\Omega G_\eps(\frarg,u_j) \dd x &= \lim_{j\to\infty} \int_\Omega \rho_\eps(u_j-u) F(\frarg,u) \dd x + \BigO(\eps) \\
  &= \int_\Omega F(\frarg,u) \dd x - \lim_{j\to\infty} \int_\Omega \bigl(1-\rho_\eps(u_j-u)\bigr) F(\frarg,u) \dd x \\
  &\qquad + \BigO(\eps) \\
  &=: \int_\Omega F(\frarg,u) \dd x - \lim_{j\to\infty} J(j,\eps) + \BigO(\eps).
\end{align*}
For the second term we have, using Proposition~\ref{prop:ext_repr} again,
\begin{align*}
  \lim_{j\to\infty} J(j,\eps) &=\phantom{:} \lim_{j\to\infty} \int_\Omega \biggl[ \bigl(1-\rho_\eps(u_j-u)\bigr) F(\frarg,u) \frac{u_j-u}{\abs{u_j-u}^2} \biggr] \cdot \overline{(u_j-u)} \dd x \\
  &=: \lim_{j\to\infty} \int_\Omega k_\eps(\frarg,u_j) \cdot \overline{(u_j-u)} \dd x\\
  &=\phantom{:} \ddprb{g_\eps \otimes I, \omega},
\end{align*}
where $k_\eps \in \Rbf^p(\Omega;\C^N)$ and $g_\eps(x,z,q) := k_\eps(x,z) \cdot q$. Consequently,
\[
  \ddprb{F,\nu} = \int_\Omega F(\frarg,u) \dd x - \ddprb{g_\eps \otimes I, \omega} + \ddprb{f_\eps \otimes I, \omega} + \BigO(\eps).
\]

Let now $(u_j), (v_j) \subset \Lrm^p(\Omega;\C^N)$ with $u_j \toweak u$ and $v_j \toweak u$ both generate $\omega \in \MCF^p(\Omega;\C^N)$ and the Young measures $\nu_1, \nu_2 \in \Ybf^p(\Omega;\C^N)$. Then, the previous arguments show $\ddprb{F,\nu_1} - \ddprb{F,\nu_2} = \BigO(\eps)$ for all $\eps > 0$. Thus, $\nu_1 = \nu_2$.
\end{proof}

\section{Oscillations and concentrations} \label{sc:osc_conc}

This section considers a few standard examples and showcases techniques to explicitly calculate the generated microlocal compactness forms. We also introduce a shorthand notation for the action of MCFs.

\subsection{Oscillations}

Consider the prototypical oscillating sequence
\begin{equation} \label{eq:osc_ex_seq}
  u_j(x) = A \ONE_{(0,\theta)}(jx \cdot n_0 - \floor{jx \cdot n_0}) + B \ONE_{(\theta,1)}(jx \cdot n_0 - \floor{jx \cdot n_0}),  \qquad x \in \Omega,
\end{equation}
where $A,B \in \C^N$, $n_0 \in \Sbb^{d-1}$, $\ONE_{(0,\theta)}$ is the indicator function of the interval $(0,\theta)$, and $\floor{s}$ denotes the largest integer below $s \in \R$. Intuitively, this sequence oscillates between the values $A$ and $B$, with volume fractions $\theta$ and $1-\theta$, respectively, and in direction $n_0$. The following general principle about oscillations shows that all of these facts can be recovered from the MCF generated by the sequence $(u_j)$:

\begin{lemma}[Oscillations] \label{lem:osc}
Let $w \in \Lrm^\infty(\R;\C^N)$ be $1$-periodic (the \enquote{profile function}) and assume that $w_j(s) := w(js)$ ($s \in \R$) generates the homogeneous Young measure $\nu \in \Mbf^1(\C^N)$. Then, the simple oscillation in direction $n_0 \in \Sbb^{d-1}$,
\begin{equation} \label{eq:uj_osc}
  u_j(x) := w(jx \cdot n_0),  \qquad x \in \Omega,
\end{equation}
generates the microlocal compactness form
\[
  \omega = \Lcal^d \restrict \Omega  \otimes \bigl[ \overline{(z-Z_0)} \, \nu(\di z) \bigr] \otimes \overline{\delta}_{\pm n_0}
  \quad \in \MCF^2(\Omega;\C^N),
\]
where $Z_0 := \dashint_0^1 w \dd s$ is the average of $w$ over one period cell and $\overline{\delta}_{\pm n_0} := \overline{\delta}_{-n_0} + \overline{\delta}_{+n_0}$. This notation means that for all $f \in \Fbf^2(\Omega;\C^N)$, $\Psi \in \Mcal$,
\[
  \ddprb{f \otimes \overline{\Psi}, \omega} = \int_\Omega \int h(x,z) \cdot \overline{\bigl[ \Psi(+n_0)+\Psi(-n_0) \bigr] (z-Z_0)} \dd \nu(z) \dd x.
\]
\end{lemma}

Of course, the exponent $p = 2$ here is chosen arbitrarily and can be replaced by any $p \in (1,\infty)$. We also remark that the notation $\overline{\delta}_{\pm n_0}$ indicates that the action of our MCF on $\Psi$ is anti-linear (as opposed to a Dirac measure which would be linear).

\begin{proof}
Since all statements are of a local nature (and we will indeed localize using a compactly supported cut-off function $\phi$), we assume without loss of generality that $\Omega = \R^d$. In general we are thus neglecting effects at the boundary $\partial \Omega$, but it can be seen easily that the sequence $(u_j)$ does not concentrate at all ($\lambda_\omega = 0$), whence this is of no concern.

\begin{figure}[t]
\def\svgscale{0.62}
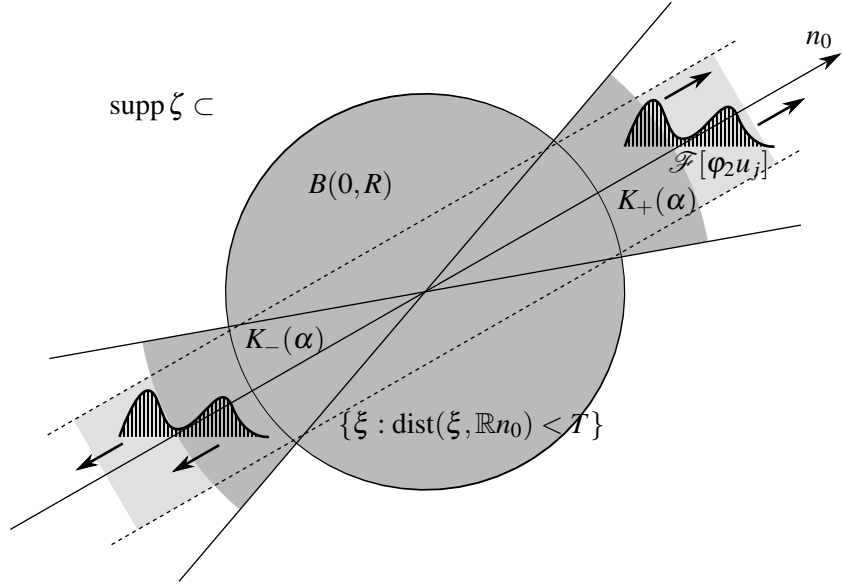
\caption{Construction in the proof of Lemma~\ref{lem:osc}.}
\label{fig:osc_lem_proof}
\end{figure}

Take any $f(x,z,q) = \phi(x) g(z) \cdot q \in \Fbf^p(\R^d;\C^N)$ with $\phi \in \Crm_c^\infty(\R^d)$, $g \in \Crm(\C^N;\C^N)$ and $\Psi \in \Mcal$ from the collection exhibited in Lemma~\ref{lem:countable_test}. It suffices to show the assertion for such $f$, $\Psi$. For $\alpha > 0$ define the cones (see Figure~\ref{fig:osc_lem_proof})
\[
  K_\pm(\alpha) := \setb{ x \in \R^d }{ \text{$\abs{x - (x \cdot n_0)n_0} \leq \alpha \abs{x \cdot n_0}$ and $x \cdot n_0 \gtrless 0$} }.
\]
Now, for $\eps > 0$ arbitrary, choose $\alpha > 0$ so small that
\begin{equation} \label{eq:Psi_close}
  \abs{\Psi(\xi)-\Psi(\pm n_0)} < \eps \qquad\text{for all $\xi \in K_\pm(2\alpha) \cap \Sbb^{d-1}$.}
\end{equation}
Next, take smooth cut-off functions $\rho_\pm \in \Crm^\infty(\Sbb^{d-1})$ with $\rho_\pm(\xi) = 1$ for $\xi \in K_\pm(\alpha) \cap \Sbb^{d-1}$, $\supp \rho_\pm \subset K_\pm(2\alpha) \cap \Sbb^{d-1}$. It always holds that $\supp \rho_+ \cap \supp \rho_- = \emptyset$. 

We will now prove the relations
\begin{align}
  &\ddprb{\phi g \otimes \overline{\rho_\pm \Psi}, \omega} = \lim_{j\to\infty} \int_{\R^d} \phi g(u_j) \cdot \overline{\Psi(\pm n_0)(u_j-Z_0)} \dd x, \label{eq:osc_est_sector}\\
  &\ddprb{\phi g \otimes \overline{(\ONE-\rho_+-\rho_-)\Psi}, \omega} = 0  \label{eq:osc_est_remainder}
\end{align}
for all $\phi = \phi_1 \overline{\phi_2}$ with $\phi_1,\phi_2 \in \Scal(\R^d)$ such that the Fourier transforms $\hat{\phi}_1,\hat{\phi}_2$ have compact support. An approximation argument then yields these relations for our original $\phi \in \Crm_c^\infty(\R^d)$ as well.

We first prove~\eqref{eq:osc_est_remainder}. By Lemma~\ref{lem:phi_exchange} and Parseval's relation we have
\begin{align*}
  &\ddprb{\phi g \otimes \overline{(\ONE-\rho_+-\rho_-)\Psi}, \omega} \\
  &\qquad = \lim_{R\to\infty} \lim_{j\to\infty} \int_{\R^d} \Psi^* \biggl(\frac{\xi}{\abs{\xi}}\biggr) \Fcal \bigl[\phi_1 g(u_j) \bigr](\xi) \cdot \overline{\Fcal[\phi_2 u_j](\xi)} \cdot \zeta_R(\xi) \dd x,
\end{align*}
where
\[
  \zeta_R(\xi) = (1-\eta_R(\xi)) (\ONE-\rho_+-\rho_-)\biggl(\frac{\xi}{\abs{\xi}}\biggr).
\]
The remarks in Section~\ref{ssc:Fourier} about Fourier transforms of standing waves entail that $\hat{u}_j$ (a tempered distribution) has support in the line $\R n_0$. Hence, $\Fcal[\phi_2 u_j] = \hat{\phi}_2 \conv \Fcal[u_j]$ is a $\Crm^\infty$-function with support in a tubular neighborhood $\{\xi : \dist(\xi,\R n_0) < T\}$ of the line $\R n_0$ ($0 < T < \infty$). The function $\zeta$ has support in the set
\[
  \bigl( K_+(\alpha) \cup K_-(\alpha) \cup B(0,R) \bigr)^c,
\]
see Figure~\ref{fig:osc_lem_proof}. Elementary geometry yields that for $\xi \in \supp \zeta$ it holds that $\dist(\xi,\R n_0) \geq C(\alpha) R$ (in fact, $C(\alpha) = (1-\alpha^{-2})^{-1/2}$). Hence we may assume that $R$ is so large that $\supp \zeta$ does not meet $\supp \Fcal[\phi_2 u_j]$ for any $j$. Thus,~\eqref{eq:osc_est_remainder} holds.

For~\eqref{eq:osc_est_sector} use similar arguments as before to see that
\begin{align*}
  &\ddprb{\phi g \otimes \overline{\rho_\pm \Psi}, \omega} \\
  &= \lim_{R\to\infty} \lim_{j\to\infty} \int_{\R^d} \Fcal \bigl[ \phi_1 g(u_j) \bigr](\xi) \cdot \overline{\Psi \biggl(\frac{\xi}{\abs{\xi}}\biggr) \Fcal[\phi_2 u_j](\xi)} \cdot \rho_\pm  \biggl(\frac{\xi}{\abs{\xi}}\biggr) (1-\eta_R(\xi)) \dd \xi \\
  &= \lim_{R\to\infty} \lim_{j\to\infty} \int_{\R^d} \Fcal \bigl[ \phi_1 g(u_j) \bigr](\xi) \cdot \overline{\Psi(\pm n_0) \Fcal[\phi_2 u_j](\xi)} \cdot \rho_\pm  \biggl(\frac{\xi}{\abs{\xi}}\biggr) (1-\eta_R(\xi)) \dd \xi \\
  &\qquad + \ddprb{\phi g \otimes \overline{\rho_\pm \cdot (\Psi(\frarg)-\Psi(\pm n_0))}, \omega}
\end{align*}
and the last error term is of order $\BigO(\eps)$ by~\eqref{eq:Psi_close}. A geometric argument analogous to the one employed for establishing~\eqref{eq:osc_est_remainder} shows that for $j$ sufficiently large, the support of $\Fcal[\phi_2 u_j]$ is contained in $K_\pm(\alpha)$ and hence the factor $\rho_\pm(\xi/\abs{\xi})$ in the last integral is equal to $1$. Finally, apply Parseval's relation and Lemma~\ref{lem:phi_exchange} again to get
\[
  \ddprb{\phi g \otimes \overline{\rho_\pm \Psi}, \omega}
  = \lim_{R\to\infty} \lim_{j\to\infty} \int_{\R^d} \phi g(u_j) \cdot \overline{T_{(1-\eta_R)\Psi(\pm n_0)}[u_j]} \dd x + \BigO(\eps).
\]
An application of Lemma~\ref{lem:eliminate_Rlim} and letting $\eps \to 0$ then yields~\eqref{eq:osc_est_sector}.

Employing~\eqref{eq:osc_est_sector},~\eqref{eq:osc_est_remainder}, we have shown
\begin{align*}
  \ddprb{\phi g \otimes \overline{\Psi}, \omega}
    &= \lim_{j\to\infty} \int_{\R^d} \phi g(u_j) \cdot \overline{[\Psi(+n_0)+\Psi(-n_0)] (u_j-Z_0)} \dd x \\
  &= \int_{\R^d} \int \phi g(z) \cdot \overline{\bigl[ \Psi(+n_0)+\Psi(-n_0) \bigr] (z-Z_0)} \dd \nu(z) \dd x,
\end{align*}
where the last equality follows from the usual Young measure theory, see Section~\ref{ssc:MCF_YM_HM}.
\end{proof}

We can now compute the MCF generated by the example sequence at the beginning of this section. We stress that both the value distribution and the \emph{direction} of the oscillation are reflected in the generated MCF.

\begin{example} \label{ex:osc}
Let $u_j$ be as in~\eqref{eq:osc_ex_seq}. By the preceding Oscillation Lemma the sequence $(u_j)$ generates the MCF
\[
  \omega = \Lcal^d \restrict \Omega \otimes \bigl[ \theta \overline{(A-M)} \delta_A + (1-\theta) \overline{(B-M)} \delta_B \bigr] \otimes \overline{\delta}_{\pm n_0}
  \;\in\;\MCF^2(\Omega;\C^N),
\]
where $M := \theta A + (1-\theta)B$.
\end{example}

For first-order laminates as in the last example, we always encounter a characteristic component of the form $\theta \cl{(A-M)} \delta_A + (1-\theta)\cl{(B-M)} \delta_B$. For higher-order laminates or general nested microstructures, this becomes more complicated. The MCF corresponding to a second-order laminate is computed later in Example~\ref{ex:2nd_order_lam} using Proposition~\ref{prop:lamination}, which investigates higher-order laminations. We will see then that the MCF also reflects the \emph{hierarchy} of laminations at different scales.

\begin{remark}
One can extend Lemma~\ref{lem:osc} to $1$-periodic profile functions $w \in \Lrm^p(\R;\C^N)$ (possibly unbounded). The result, with the space $\MCF^2$ replaced by $\MCF^p$, and the proof remain largely the same except for the following two modifications: Parseval's theorem only holds if we assume that all $u_j$ are of class $\Crm_c^\infty$, see Lemma~\ref{lem:better_generation}. Also, to make the usual Young measure theory applicable one needs to use the fact that the family $(u_j)$ as defined in~\eqref{eq:uj_osc} is equiintegrable.
\end{remark}

\subsection{Concentrations} \label{ssc:conc}

Besides oscillations, MCFs also represent concentration phenomena. The following is a simple example:

\begin{example} \label{ex:conc}
Let $w \in \Lrm^p(\R^d)$ have compact support and satisfy $w \geq 0$. Further, let $Z_0 \in \C^N$ with $\abs{Z_0} = 1$. Define
\[
  u_j(x) := Z_0 \, j^{d/p} w(jx),  \qquad x \in \R^d.
\]
Then, $(u_j)$ generates the MCF
\[
  \omega = \delta_0 \otimes  \overline{Z_0} \delta_{\infty Z_0} \otimes \bar{\mu}  \quad\in\quad \MCF^p(\R^d;\C^N),
\]
where $\bar{\mu}$ is the surface measure (acting on $\overline{\Psi}$)
\begin{equation} \label{eq:conc_meas}
  \bar{\mu} = \biggl( \int_0^\infty \Fcal \bigl[\abs{w}^{p-1}\bigr](t \eta) \, \overline{\hat{w}(t \eta)} \, t^{d-1} \dd t \biggr) \, \bigl( \Hcal^{d-1} \restrict \Sbb^{d-1} \bigr) (\di \eta).
\end{equation}
The above expression for $\omega$ is a shorthand notation for the MCF that acts on $f(x,z,q) = h(x,z) \cdot q \in \Fbf^p(\Omega;\C^N)$, $\Psi \in \Mcal$, as
\[
  \ddprb{f \otimes \overline{\Psi},\omega} = h^\infty(0,Z_0) \cdot \int_{\Sbb^{d-1}} \overline{\Psi(\xi) Z_0} \dd \bar{\mu}(\xi).
\]
We only sketch the proof: Let $f, \Psi$ be as before and consider
\begin{align*}
  J_{j,R} &:= \int h(\frarg,u_j) \cdot \overline{T_{(1-\eta_R)\Psi}[u_j]} \dd x \\
  &\phantom{:}= \int h \bigl( x,Z_0 \, j^{d/p} w(jx) \bigr) \cdot \overline{T_{(1-\eta_R)\Psi} \bigl[ Z_0 \, j^{d/p} w(jx') \bigr](x)} \dd x \\
  &\phantom{:}\approx \int h \bigl( 0,Z_0 \, j^{d/p} w(jx) \bigr) \cdot \overline{T_{(1-\eta_R)\Psi} \bigl[ Z_0 \, j^{d/p} w(jx') \bigr](x)} \dd x
\end{align*}
for $j$ large, where \enquote{$\approx$} means that the error vanishes as $j\to\infty$, $R\to\infty$. Here we used the fact that the support of $w(jx)$ shrinks to zero. Then, we employ $T_{(1-\eta_R)\Psi}[w(jx')](x) = T_{(1-\eta_{R/j})\Psi}[w](jx)$ to calculate further
\begin{align*}
  J_{j,R} &\approx \int \frac{h(0,Z_0 \, j^{d/p} w)}{j^{d(p-1)/p}} \cdot \overline{T_{(1-\eta_{R/j})\Psi}[Z_0 w]} \dd y \\
  &\approx \int h^\infty(0,Z_0) \abs{w}^{p-1} \cdot \overline{T_{(1-\eta_{R/j})\Psi}[Z_0 w]} \dd y.
\end{align*}
For the last expression we used that
\[
  j^{-d(p-1)/p} h(0,Z_0 \, j^{d/p} w(y)) \to h^\infty(0,Z_0) \abs{w(y)}^{p-1}
\]
for almost every $y$ and then in $\Lrm^{p/(p-1)}$ since we have a uniform $\Lrm^{p/(p-1)}$-majorant. Finally, use Parseval's formula to conclude
\begin{align*}
  \ddprb{f \otimes \overline{\Psi}, \omega} &= \lim_{R\to\infty} \lim_{j\to\infty} J_{j,R} \\
  &= h^\infty(0,Z_0) \cdot \int \abs{w}^{p-1} \cdot \overline{T_{\Psi}[Z_0 w]} \dd y \\
  &= h^\infty(0,Z_0) \cdot \int \Fcal[\abs{w}^{p-1}](\xi) \cdot \overline{\Psi(\xi)Z_0 \hat{w}(\xi)} \dd \xi \\
  &= h^\infty(0,Z_0) \cdot \int \Fcal \bigl[ \abs{w}^{p-1} \bigr](\xi) \, \overline{\hat{w}(\xi)} \cdot \overline{\Psi(\xi) Z_0} \dd \xi.
\end{align*}
Transforming to polar coordinates, we arrive at~\eqref{eq:conc_meas}.
\end{example}

We also prove the following interesting converse to Lemma~\ref{lem:pq_conc}:

\begin{proposition}[Equiintegrability] \label{prop:equiint}
Let $(u_j) \subset \Lrm^p(\Omega;\C^N)$ generate $\omega \in \MCF^p(\Omega;\C^N)$. Then, the following are equivalent:
\begin{itemize}
\item[(i)] $(u_j)$ is $p$-equiintegrable.
\item[(ii)] $\dprb{\abs{z}^{p-2}z \otimes I, \omega_x^\infty} = 0$ for $\lambda_\omega$-a.e.\ $x \in \cl{\Omega}$.
\item[(iii)] $\omega_x^\infty = 0$ for $\lambda_\omega$-a.e.\ $x \in \cl{\Omega}$.
\end{itemize}
\end{proposition}

\begin{proof}
We first show (i) $\Leftrightarrow$ (ii). Let $K > 0$, set
\[
  g(t) := \begin{cases}
            0     & \text{if $0 < t < \frac{1}{2}$,} \\
            2t-1  & \text{if $\frac{1}{2} \leq t \leq 1$,} \\
            1     & \text{if $t > 1$,}
          \end{cases}
  \qquad
  g_K(t) := g \biggl( \frac{t}{K} \biggr),
  \qquad
  t \geq 0,
\]
and observe  $g_{2K} \leq \ONE_{[K,\infty)} \leq g_K$. Then define $f_K(x,z,q) = h_K(x,z) \cdot q \in \Fbf^p(\Omega;\C^N)$ through
\[
  h_K(x,z) := \phi(x) g_K(\abs{z}) \cdot \abs{z}^{p-2} z,
\]
where $\phi \in \Crm^\infty(\cl{\Omega})$, and for which we may compute $h_K^\infty(x,z) = \phi(x) \abs{z}^{p-2} z$.

Using Lemma~\ref{lem:eliminate_Rlim} we have
\begin{align*}
  \ddprb{f_{2K} \otimes I,\omega} &= \lim_{j\to\infty} \int_\Omega h_{2K}(\frarg,u_j) \cdot \overline{(u_j-u)} \dd x \\
  &\leq \limsup_{j\to\infty} \int_{\Omega\cap\{\abs{u_j} \geq K\}} \phi \, \abs{u_j}^p \dd x \\
  &\qquad - \lim_{j\to\infty} \int_\Omega \phi \, g_{2K}(\abs{u_j}) \cdot \abs{u_j}^{p-2} u_j \cdot \overline{u} \dd x.
\end{align*}
We claim that
\begin{equation} \label{eq:conc_equiint}
  \lim_{K\toup\infty}\, \limsup_{j\to\infty} \int_\Omega g_{2K}(\abs{u_j}) \cdot \abs{u_j}^{p-2} u_j \cdot \overline{u} \dd x = 0.
\end{equation}
This is equivalent to the equiintegrability of the functions $(\abs{u_j}^{p-1} \cdot \abs{u})_j$, which can be seen as follows: $(\abs{u_j}^{p-1})_j$ is uniformly bounded in $\Lrm^{p/(p-1)}(\Omega)$, thus, up to a subsequence, $\abs{u_j}^{p-1} \toweak U \in \Lrm^{p/(p-1)}(\Omega)$. In particular, $\abs{u_j}^{p-1} \cdot \abs{u} \toweak U \cdot \abs{u}$ weakly in $\Lrm^1(\Omega)$. Hence, $(\abs{u_j}^{p-1} \cdot \abs{u})_j$ is weakly relatively compact in $\Lrm^1(\Omega)$ and consequently, the Dunford--Pettis criterion implies that this sequence is equiintegrable and~\eqref{eq:conc_equiint} holds.

Thus, using an analogous argument for the upper bound, we arrive at
\[
  \ddprb{f_{2K} \otimes I,\omega} \leq \limsup_{j\to\infty} \int_{\Omega\cap\{\abs{u_j} \geq K\}} \phi \, \abs{u_j}^p \dd x 
  \leq \ddprb{f_K \otimes I,\omega}
\]
and together with
\begin{align*}
  \lim_{K\toup\infty}\; \ddprb{f_K \otimes I,\omega}
  &= \lim_{K\toup\infty}\, \int_\Omega \phi(x) \, \dprb{g_K(\abs{z}) \cdot \abs{z}^{p-2} z \otimes I,\omega_x} \dd x \\
  &\qquad + \int_{\cl{\Omega}} \phi(x) \, \dprb{\abs{z}^{p-2} z \otimes I,\omega_x^\infty} \dd \lambda_\omega(x) \\
  &= 0 + \int_{\cl{\Omega}} \phi(x) \, \dprb{\abs{z}^{p-2} z \otimes I,\omega_x^\infty} \dd \lambda_\omega(x)
\end{align*}
the equivalence of (i) and (ii) follows.

Since (iii) clearly implies (ii), it remains to show (i) $\Rightarrow$ (iii). For this, let $f(x,z,q) = h(x,z) \cdot q \in \Fbf^p(\Omega;\C^N)$ and by the $p$-equiintegrability of $(u_j)$ take $K > 0$ so large that
\[
  \supmod_j \int_{\{\abs{u_j} \geq K\}} \abs{h(x,u_j(x))}^{p/(p-1)} \dd x \leq \eps
\]
Thus, if we define
\[
  f_{2K}(x,z,q) := h_{2K}(x,z) \cdot q,
  \quad\text{where}\quad
  h_{2K}(x,z) := (1-g_{2K}(\abs{z})) h(x,z)
\]
with $g_{2K}$ defined as before, then
\[
  \ddprb{f \otimes \overline{\Psi},\omega} = \ddprb{f_{2K} \otimes \overline{\Psi},\omega} + \BigO(\eps)
\]
and letting $K \toup \infty$ while employing $h_{2K}^\infty \equiv 0$, we arrive at
\[
  \int_{\cl{\Omega}} \phi(x) \, \dprb{h^\infty(x,\frarg) \otimes \overline{\Psi},\omega_x^\infty} \dd \lambda_\omega(x) = 0.
\]
Varying $h$ and $\Psi$, the sought implication follows.
\end{proof}

\section{Differential constraints and compensated compactness} \label{sc:diff_constr}

In this section we investigate how differential constraints on the generating sequence are reflected in the corresponding microlocal compactness form.

\subsection{Differential constraints} \label{ssc:diff_constr}

We mostly work in the full-space case $\Omega = \R^d$, which will suffice for our later aims (but see Remark~\ref{rem:domain} below for the case $\Omega \neq \R^d$).

We let $(I-\Delta)^{s/2}$ and $(-\Delta)^{s/2}$ be the Fourier multiplier operators with symbols $(1+4\pi^2\abs{\xi}^2)^{s/2}$ and $(2\pi\abs{\xi})^s$, respectively. Then, $\Wrm^{s,p}(\R^d;\C^N)$ is the space of all distributions $v$ in $\R^d$ such that $(I-\Delta)^{s/2} v \in \Lrm^{p}(\R^d;\C^N)$; as norm we use $\norm{v}_{s,p} := \norm{(I-\Delta)^{s/2} v}_{p}$. In turn, $\Wrm^{-s,p}(\R^d;\C^N)$ is defined to be the dual space to $\Wrm^{s,p'}(\R^d;\C^N)$, where $1/p + 1/p' = 1$.

Let $\Acal$ be a \term{linear constant-coefficient partial differential operator of order $s \in \N$},
\[
  \Acal = \sum_{\abs{\alpha} = s} A^{(\alpha)} \partial^\alpha,
  \qquad \text{$A^{(\alpha)} \in \R^{l \times N}$, $\alpha \in \N_0^d$ with $\abs{\alpha} = \abs{\alpha_1} + \ldots + \abs{\alpha_d} = s$,}
\]
where we have employed the usual multi-index notation. A function $u \in \Lrm^p(\R^d;\C^N)$ with $\Acal u = 0$ in the $\Wrm^{-1,p}$-sense, that is $\int u \cdot \Acal^* \phi \dd x = 0$ for all $\phi \in \Wrm^{1,p'}(\R^d;\C^N)$, where $1/p + 1/p' = 1$ and $\Acal^* = \sum_{\abs{\alpha} = s} [A^{(\alpha)}]^* \partial^\alpha$, is called \term{$\Acal$-free}. The operator $\Acal$ has the \term{symbol}
\[
  \Abb(\xi) := \sum_{\abs{\alpha} = s} A^{(\alpha)} (2\pi\ii \xi)^\alpha
  \quad \in \C^{l \times N}, \qquad \xi \in \R^d.
\]
Moreover, we define the \term{bounded symbol} $\Abb_b$ and the \term{homogeneous symbol} $\Abb_0$ as follows:
\begin{align*}
  \Abb_b(\xi) &:= \frac{1}{(1+4\pi^2\abs{\xi}^2)^{s/2}} \sum_{\abs{\alpha} = s} A^{(\alpha)} (2\pi\ii \xi)^\alpha, \\
  \Abb_0(\xi) &:= \frac{1}{(2\pi\abs{\xi})^s} \sum_{\abs{\alpha} = s} A^{(\alpha)} (2\pi\ii \xi)^\alpha.
\end{align*}
The corresponding operators $\Acal_b := T_{\Abb_b}$ and $\Acal_0 := T_{\Abb_0}$ are bounded from $\Lrm^p$ to itself by the usual multiplier theorems and we have the identities
\[
  \Acal = \Acal_b \circ (I-\Delta)^{s/2} = \Acal_0 \circ (-\Delta)^{s/2},
\]
which we will use to factor $\Acal$ into an $\Lrm^p$-bounded and a differential part.
As an example, consider the operator $\Acal = -\partial/\partial x_j$; then $\Acal_0$ is the $j$th Riesz transform, i.e.\ the multiplier operator with symbol $-\ii \xi_j/\abs{\xi}$, cf.\ Section~4.1.4 in~\cite{Graf08CFA}.

We will sometimes assume the following \term{constant-rank property} on $\Acal$, which was introduced by Murat~\cite{Mura81CCCN} (also see~\cite{FonMul99AQLS}):
\begin{equation} \label{eq:crp}
  \rank \ker \Abb(\xi) = \mathrm{const}
  \qquad \text{for all $\xi \in \Sbb^{d-1}$.}
\end{equation}
It is well-known that many useful operators satisfy this constant-rank property, among them $\diverg$, $\curl$, the static Maxwell operator $\Acal = \bigl(\begin{smallmatrix}\diverg & \diverg \\ 0 & \curl \end{smallmatrix} \bigr)$, and the operator corresponding to the constraints $\partial_1 u^1 + \partial_2 u^2 = 0$, $\partial_1 u^3 + \partial_2 u^4 = 0$ (an example by Tartar).

When we impose differential constraints on sequences $(u_j) \subset \Lrm^p(\R^d;\C^N)$ with $u_j \toweak 0$, we will do so in the usual way in compensated compactness theory~\cite{Mura78CPC,Mura79CPC2,Tart79CCAP}, i.e.\ we require that 
\begin{equation} \label{eq:Au_j_zero}
  \Acal u_j \to 0  \quad\text{in $\Wrm^{-s,p}(\R^d;\C^l)$.}
\end{equation}
Using the decomposition
\[
  \Acal = \Acal_b \circ (I-\Delta)^{s/2},
\]
we get
\begin{align*}
  \normb{\Acal_b u_j}_{p} &= \sup_{\norm{w}_{p'} \leq 1} \absb{\dprb{u_j, \Acal_b^* w}} \\
  &= \sup_{\norm{w}_{p'} \leq 1} \absb{\dprb{u_j, \Acal_b^* \circ (I-\Delta)^{s/2} \circ (I-\Delta)^{-s/2} w}}  \\
  &= \sup_{\norm{v}_{s,p'} \leq 1} \absb{\dprb{u_j, \Acal^* v}}
    = \normb{\Acal u_j}_{-s,p}.
\end{align*}
Hence,~\eqref{eq:Au_j_zero} is equivalent to
\begin{equation} \label{eq:Abu_j_zero}
  \Acal_b u_j \to 0  \quad\text{in $\Lrm^p(\R^d;\C^l)$.}
\end{equation}

We now prove that linear PDE constraints on a sequence are reflected in the generated MCF in a simple way:

\begin{theorem} \label{thm:diff_omega}
Let $(u_j) \subset \Lrm^p(\R^d;\C^N)$ generate $\omega \in \MCF^p(\R^d;\C^N)$, let $\Acal$ be a linear constant-coefficient PDE operator of order $s \in \N$ (not necessarily satisfying the constant-rank property), and
\[
  \Acal u_j \to 0  \qquad\text{in $\Wrm^{-s,p}(\R^d;\C^N)$.}
\]
Then, \enquote{$\Acal \omega = 0$} in the sense that
\begin{equation} \label{eq:Aomega0_ext}
  \ddprb{f \otimes \overline{\Psi \Abb_0}, \omega} = 0
\end{equation}
for all $f \in \Fbf^p(\Omega;\C^N)$ and all $\Psi \in \Crm^{\floor{d/2}+1}(\Sbb^{d-1};\C^{N \times l})$.
\end{theorem}

Here, $\Psi$ has a different signature, because $\Abb_0$ takes values in $\C^{l \times N}$. An $\omega$ satisfying~\eqref{eq:Aomega0_ext} is called \term{$\Acal$-free}. In the special case that $\Acal = \curl$, i.e.\
\begin{equation} \label{eq:A_curl}
  \Acal w = \bigl[ \partial_j w^i_k - \partial_k w^i_j \bigr]_{i,j,k},
  \qquad \text{$w \colon \R^d \to \R^{m \times d}$}
\end{equation}
($i = 1,\ldots,m$, $j,k = 1,\ldots,d$) we also call $\omega$ a \term{gradient MCF}.

\begin{proof}[Proof of Theorem~\ref{thm:diff_omega}]
Since $\Abb_b(\xi)$ and $\Abb_0(\xi)$ are asymptotically the same as $\abs{\xi} \to \infty$, we have
\begin{align*}
  \int h(\frarg,u_j) \cdot \overline{T_{(1-\eta_R) \Psi \Abb_0}[u_j]} \dd x
  &= \int h(\frarg,u_j) \cdot \overline{T_{(1-\eta_R) \Psi \Abb_b}[u_j]} \dd x + E(R) \norm{u_j}_p \\
  &= \int h(\frarg,u_j) \cdot \overline{T_{(1-\eta_R) \Psi}[\Acal_b u_j]} \dd x + E(R) \norm{u_j}_p
\end{align*}
and the error term $E(R)$ vanishes as $R \to \infty$. Using~\eqref{eq:Abu_j_zero}, the last expression converges to zero as $j \to \infty$ and $R \to \infty$ (in this order).
\end{proof}

\begin{remark} \label{rem:domain}
Our results are not directly applicable when $\Omega \neq \R^d$, because we employed the Fourier characterization of $\Acal$-freeness, which is only valid in the whole space. If we want to work in an arbitrary Lipschitz domain $\Omega \subset \R^d$, there are two options:
\begin{enumerate}
\item If $\Acal$ is a first-order operator, then under the additional assumption of the natural boundary condition
\[
  \qquad \Abb(n_\Omega) u|_{\partial \Omega} = 0,
  \quad\text{where $n_\Omega \colon \partial \Omega \to \Sbb^{d-1}$ is the unit outer normal,}
\]
the preceding result holds in $\Omega$ as well, because then $\Acal u = 0$ in $\Omega$ is equivalent to $\Acal u = 0$ in $\R^d$ (both in the $\Wrm^{-1,p}$-sense).

\item Assume that for all $u \in \Lrm^p(\Omega;\C^N)$ there exists an $\Lrm^p$-bounded extension $Eu \in \Lrm^p(\R^d;\C^N)$ such that $Eu|_\Omega = u$, $\norm{Eu}_{p,\R^d} \leq C\norm{u}_{p,\Omega}$, $Eu$ has compact support, and $\Acal (Eu) = 0$ in $\R^d$ whenever $\Acal u = 0$ in $\Omega$. Then it can be seen that the previous result continues to hold. For example, if $\Acal = \curl$, such an extension is available and can be constructed via the potential.
\end{enumerate}
\end{remark}

The following is a converse to Theorem~\ref{thm:diff_omega} in the $\Lrm^2$-case:

\begin{theorem} \label{thm:diff_genseq}
Let $\omega \in \MCF^2(\R^d;\C^N)$ be generated by a sequence $(u_j) \subset \Lrm^2(\R^d;\C^N)$ with $u_j \toweak u$ in $\Lrm^2(\R^d;\C^N)$. Furthermore, let $\Acal$ be a linear constant-coefficient PDE operator of order $s \in \N$ (not necessarily satisfying the constant-rank property) and assume that
\begin{equation} \label{eq:Aomega0_minimal}
  \ddprb{z \cdot q \otimes \overline{\Abb_0^* \Abb_0}, \omega} = 0   \quad \in \C^N,
\end{equation}
which is in particular implied by the condition~\eqref{eq:Aomega0_ext}, and also that $\Acal u = 0$ (in the $\Wrm^{-s,2}$-sense). Then,
\[
  \Acal u_j \to 0  \qquad\text{in $\Wrm^{-s,2}(\R^d;\C^l)$.}
\]
\end{theorem}

\begin{remark}
Before we come to the proof, a few remarks are in order:
\begin{enumerate}
  \item The fact that $\omega$ is generated by \emph{some} sequence is a genuine assumption, see Example~\ref{ex:MCF_nogen} below.
  \item The simplicity of the expression of PDE-constraints should be compared to the situation for (generalized) Young measures, where the characterization problem is in general a difficult one. For classical Young measures, this is the objective of the celebrated Kinderlehrer--Pedregal Theorem~\cite{KinPed91CYMG,KinPed94GYMG,FonMul99AQLS}; the recent extension to generalized Young measures can be found in~\cite{KriRin10CGGY,Rind14LPCY}.
  \item Remarkably, $\Acal u_j \to 0$ in $\Wrm^{-s,2}$ holds for \emph{all} generating sequences. This is a much stronger conclusion than what can be obtained in Young measure theory, where for example a gradient constraint on the Young measure in the sense of the aforementioned Kinderlehrer--Pedregal Theorem is far from implying that all generating sequences (asymptotically) consist only of gradients. In fact, generically there also exist divergence-free generating sequences for gradient Young measures, and divergence-free sequences are in a sense \enquote{orthogonal} to sequences of gradients.
\end{enumerate}
\end{remark}

\begin{proof}[Proof of Theorem~\ref{thm:diff_genseq}]
The conclusion is equivalent to $\Acal_b u_j \to 0$ in $\Lrm^2(\R^d;\C^l)$, and so we compute using Parseval's Theorem,
\begin{align*}
  \int \abs{\Acal_b u_j}^2 \dd x &= \int \bigl( \Abb_b \hat{u}_j \bigr) \cdot \overline{\bigl( \Abb_b \hat{u}_j \bigr)} \dd \xi
    = \int \hat{u}_j \cdot \overline{\bigl( \Abb_b^* \Abb_b \hat{u}_j \bigr)} \dd \xi \\
  &= \int u_j \cdot \overline{T_{(1-\eta_R) \Abb_b^* \Abb_b}[u_j]} \dd x + \int u_j \cdot \overline{T_{\eta_R \Abb_b^* \Abb_b}[u_j]} \dd x.
\end{align*}
Now let first $j \to \infty$ and then $R \to \infty$ to arrive via Lemmas~\ref{lem:bandlim_compact},~\ref{lem:TetaR_conv} at
\[
  \lim_{j\to\infty} \int \abs{\Acal_b u_j}^2 \dd x
  = \ddprb{z \cdot q \otimes \Abb_0^* \Abb_0, \omega} + \int u \cdot \overline{T_{\Abb_b^* \Abb_b}[u]} \dd x.
\]
The first term is zero by~\eqref{eq:Aomega0_minimal} and the second term is equal to $\norm{\Acal_b u}_2$ and hence is also zero by assumption.
\end{proof}

\begin{remark} \label{rem:Afree_proj}
If $\Acal$ satisfies the constant-rank property~\eqref{eq:crp}, we can project the elements of a generating sequence $(u_j)$ with $\Acal u_j \to 0$ in $\Wrm^{-s,2}(\R^d;\C^N)$ onto the set of $\Acal$-free functions, see Lemma~2.14 in~\cite{FonMul99AQLS}, to obtain a generating sequence that is actually $\Acal$-free and not just asymptotically $\Acal$-free.
\end{remark}

We now give two examples for the preceding results:

\begin{example} \label{ex:osc_curl0}
Consider again the situation of Example~\ref{ex:osc}, but let now $A, B \in \R^{m \times d}$. As we saw in that example, the laminar oscillation between $A$ and $B$ in direction $n_0 \in \Sbb^{d-1}$ generates the MCF
\[
  \omega = \Lcal^d \restrict \Omega \otimes \bigl[ \theta \overline{(A-M)} \delta_A + (1-\theta) \overline{(B-M)} \delta_B \bigr] \otimes \overline{\delta}_{\pm n_0}.
\]
Take $\Acal = \curl$ as in~\eqref{eq:A_curl}. Then, condition~\eqref{eq:Aomega0_minimal} for $\omega$ translates into
\begin{align*}
  0 &= \sum_\pm \int_\Omega \theta A \overline{\Abb_0^*(\pm n_0) \Abb_0(\pm n_0) (A-M)} \\
    &\qquad\qquad + (1-\theta) B \overline{\Abb_0^*(\pm n_0) \Abb_0(\pm n_0) (B-M)} \dd x,
\end{align*}
where $\sum_\pm$ denotes the summation over both signs for $\pm n_0$. One can compute
\[
  \ker \Abb_0(\pm n_0) = \setb{ a \otimes n_0 }{ a \in \R^m },
\]
and hence condition~\eqref{eq:Aomega0_minimal} is true if
\begin{equation} \label{eq:curl_matrix_cond}
  A - M = a \otimes n_0  \qquad\text{and}\qquad
  B - M = b \otimes n_0
\end{equation}
for some $a,b \in \R^m$. On the other hand, from~\eqref{eq:Aomega0_ext} for $f(x,z,q) := (z-M) \cdot q$ and $\Psi(\xi) := \Abb_0^*(\xi)$, we get that $\Acal u_j \to 0$ in $\Wrm^{-1,p}$ implies
\begin{align*}
  0 &= \theta \sum_\pm \int_\Omega (A-M) \overline{\Abb_0^*(\pm n_0) \Abb_0(\pm n_0) (A-M)} \dd x \\
  &\qquad + (1-\theta) \sum_\pm \int_\Omega (B-M) \overline{\Abb_0^*(\pm n_0) \Abb_0(\pm n_0) (B-M)} \dd x \\
  &= \theta \sum_\pm \int_\Omega \abs{\Abb_0(\pm n_0) (A-M)}^2 \dd x \\
  &\qquad + (1-\theta) \sum_\pm \int_\Omega \abs{\Abb_0(\pm n_0) (B-M)}^2 \dd x,
\end{align*}
whence~\eqref{eq:curl_matrix_cond} holds. Thus, from Theorems~\ref{thm:diff_omega},~\ref{thm:diff_genseq} we conclude that $\omega$ is a gradient MCF if and only if
\[
  B - A = c \otimes n_0  \qquad\text{for some $c \in \R^m$.}
\]
As expected, this is a generalization of the usual Hadamard jump condition requiring rank-one connectedness across the jump surface (in this context also see~\cite{BalJam87FPMM}).
\end{example}

\begin{example}
Again considering the situation of Example~\ref{ex:osc}, but now with $A, B \in \R^d$ and $\Acal = \diverg$, we get by a similar derivation that $\diverg \omega = 0$ if and only if
\[
  (A - M) \perp n_0  \qquad\text{and}\qquad
  (B - M) \perp n_0.
\]
\end{example}

We end this section with an example showing that certain MCFs cannot be generated by \emph{any} sequence. This raises the problem of characterizing the subclass of MCFs that are generated by at least one sequence. We reserve this for future work.

\begin{example} \label{ex:MCF_nogen}
Let
\[
  a := \begin{pmatrix} 1 \\ 0 \end{pmatrix}, \qquad
  b := \begin{pmatrix} 0 \\ 1 \end{pmatrix}, \qquad
  A := a \otimes a,  \qquad
  B := b \otimes b,
\]
and consider the MCF
\[
  \omega = \Lcal^d \restrict \Omega \otimes \biggl[ \frac{1}{2} A \delta_A \otimes \overline{\delta}_{\pm a} + \frac{1}{2} B \delta_B \otimes \overline{\delta}_{\pm b} \biggr].
\]
A quick way to see that this $\omega$ is not generated by any sequence is as follows: Theorem~\ref{thm:diff_genseq} implies that any generating sequence asymptotically would have to be a sequence of gradients. However, the corresponding Young measure is $\frac{1}{2} \delta_A + \frac{1}{2} \delta_B$, which is not a gradient Young measure. This follows for example from the Kinderlehrer--Pedregal theorem~\cite{KinPed91CYMG,KinPed94GYMG}, because $\rank \, (A-B) = 2$.
\end{example}

\subsection{Compactness wavefront set} \label{ssc:WF}

In this section we introduce a notion of wavefront set of microlocal compactness forms $\omega \in \MCF^p(\Omega;\C^N)$: Recall the sphere compactification $\sigma\C^N$ from Section~\ref{ssc:sphere_compact} and also the extension construction for $Eh$ in~\eqref{eq:Eh}. The \term{($p$-compactness) wavefront set} $\WF(\omega) \subset \cl{\Omega} \times \sigma\C^N \times \Sbb^{d-1}$ is the smallest closed subset $A$ of $\cl{\Omega} \times \sigma\C^N \times \Sbb^{d-1}$ with the property that for any $f(x,z,q) = \phi(x) g(z) \cdot q \in \Fbf^p(\Omega;\C^N)$ with $\phi \in \Crm(\cl{\Omega})$, $g \in \Crm(\C^N;\C^N)$ ($S^{p-1}g \in \Crm(\cl{\C\B^N};\C^N)$), and any $\Psi \in \Mcal$,
\[
  \supp \, \phi \otimes Eg(z) \otimes \overline{\Psi} \subset A^c
  \qquad\text{implies}\qquad
  \ddprb{f \otimes \overline{\Psi}, \omega} = 0.
\]
It can be seen that if $\WF(\omega) = \emptyset$, then $\omega = 0$. Since $\omega$ can also be seen as a distribution, see Section~\ref{ssc:ext_repr_distrib}, it is worth pointing out that the wavefront set considered here is different from the one considered by H\"{o}rmander, cf.\ Chapter~VIII of~\cite{Horm90ALPD1}, applied to this distribution. Another notion with related aims, the \enquote{$\Lrm^2$-compactness wavefront set}, was introduced by G\'{e}rard in~\cite{Gera88CCRD}.

\begin{example}
The oscillatory MCF
\[
  \omega = \Lcal^d \restrict \Omega \otimes \bigl[ \theta \overline{(A-M)} \delta_A + (1-\theta) \overline{(B-M)} \delta_B \bigr] \otimes \overline{\delta}_{\pm n_0}
\]
from Example~\ref{ex:osc} has the wavefront set
\[
  \WF(\omega) = \cl{\Omega} \times \{A,B\} \times \{+n_0,-n_0\}.
\]
\end{example}

\begin{example}
A concentration effect is expressed in the MCF
\[
  \omega = \delta_0 \otimes  \overline{Z_0} \delta_{\infty Z_0} \otimes \bar{\mu}
  \quad\in \MCF^p(\Omega;\C^N),
\]
from Example~\ref{ex:conc}. For its wavefront set we get
\[
  \WF(\omega) = (0,\infty \overline{Z_0}) \times \supp \bar{\mu},
\]
where the support of $\bar{\mu}$, which is given in~\eqref{eq:conc_meas}, depends on the cancellation properties of the profile function $w$'s Fourier transform.
\end{example}

The following lemma describes the effect of pointwise constraints on the wavefront set:

\begin{lemma} \label{lem:WF_pointwise_constr}
Let $(u_j) \subset \Lrm^p(\Omega;\C^N)$ generate $\omega \in \MCF^p(\Omega;\C^N)$ and let $Z(x) \subset \sigma\C^N$ for all $x \in \Omega$ satisfy the \emph{continuity property}
\[
  Z := \setb{ (x,z) \in \cl{\Omega} \times \sigma\C^N }{ z \in Z(x) }
  \qquad\text{is closed in $\cl{\Omega} \times \sigma\C^N$.}
\]
If the pointwise constraint
\[
  u_j(x) \in Z(x)  \qquad \text{for a.e.\ $x \in \Omega$, $j \in \N$,}
\]
holds, then
\[
  \WF(\omega) \subset Z \times \Sbb^{d-1} = \setb{ (x,z,\xi) \in \cl{\Omega} \times \sigma\C^N \times \Sbb^{d-1} }{ z \in Z(x) }.
\]
\end{lemma}

\begin{proof}
Let $f(x,z,q) = \phi(x) g(z) \cdot q \in \Fbf^p(\Omega;\C^N)$ and $\Psi \in \Mcal$ be such that
\[
  \supp \, \phi \otimes Eg \otimes \overline{\Psi} \subset \setb{ (x,z,\xi) \in \cl{\Omega} \times \sigma\C^N \times \Sbb^{d-1} }{ z \notin Z(x) } = Z^c \times \Sbb^{d-1}.
\]
Then,
\[
  \ddprb{f \otimes \overline{\Psi}, \omega} = \lim_{R\to\infty} \lim_{j\to\infty} \int_\Omega \phi \, g(u_j) \cdot \overline{T_{(1-\eta_R)\Psi}[u_j]} \dd x = 0,
\]
because $\phi g(u_j) \equiv 0$.
\end{proof}

The interplay between pointwise and differential constraints, which is related to compensated compactness theory, is more complicated and will be treated in the next section.

\begin{remark}[Comparison to classical wave front set]
Our notion of the $p$-compactness wavefront set is named in analogy to the classical notion from microlocal analysis, as introduced by H\"{o}rmander, see for example in Chapter~VIII of~\cite{Horm90ALPD1}. However, a $p$-MCF measures the difference between weak and strong $\Lrm^p$-compactness and not the directions where $\Crm^\infty$-regularity fails, like the classical wavefront set. Of course, our $p$-compactness wavefront set contains more information than a mere analogue of the classical wavefront set, because the asymptotic value distribution of the generating sequence is also preserved. We remark in passing that in classical microlocal analysis there is no object corresponding to the full MCF and the quantitative information it contains.
\end{remark}

\subsection{Compensated compactness} \label{ssc:comp_compact}

The general philosophy of compensated compactness is to employ additional properties of sequences to improve weak to strong compactness. We here take a slightly broader view and also consider mere restrictions on the direction and value distribution of admissible oscillations and concentrations as \enquote{compensated compactness}. We will focus on the interplay between differential and pointwise constraints and how it affects the wavefront set (and thus compactness).

Let $\Acal$ be a linear constant-coefficient PDE operator of order one as before. From now on we always assume the constant-rank property~\eqref{eq:crp}. Our basic setup follows the general framework of compensated compactness as introduced by Tartar~\cite{Tart79CCAP,Tart83CCMA}: Let $(u_j) \subset \Lrm^p(\Omega;\C^N)$ be a sequence satisfying the \term{differential inclusion}
\[
  \left\{ \begin{aligned}
    \Acal u_j &= 0     \qquad &&\text{in $\Omega$,} \\
    u_j(x) &\in Z(x)          &&\text{for a.e.\ $x \in \Omega$,}
      \qquad\text{$Z \subset \cl{\Omega} \times \sigma \C^N$ closed.}
  \end{aligned} \right.
\]
It turns out that this formulation is very useful for a variety of problems, for instance in the theory of conservation laws~\cite{Tart79CCAP,Tart83CCMA}, rigidity theory~\cite{Kirc03RGM,Rind12LSYM,Rind11LSIF}, and fluid mechanics, see~\cite{DeLSze09EEDI}, where the Euler equation is written in this form (a related reference investigating Young-measure solutions to the Euler equation is~\cite{SzeWie12YMGI}).

The aim now is to investigate how we can use MCFs to efficiently exploit the conditions above in our study of weak--strong compactness. Having already considered pointwise and differential constraints in isolation, we now aim to analyze the interaction between them. The following theorem is a general compensated compactness result:

\begin{theorem}[Compensated Compactness] \label{thm:cc_general}
Let $(u_j) \subset \Lrm^p(\R^d;\C^N)$ generate the microlocal compactness form $\omega \in \MCF^p(\R^d;\C^N)$ and let $\Acal$ be a constant-coefficient PDE operator of order one satisfying the constant-rank property. We furthermore assume:
\begin{itemize}
\item[(i)] $\Acal u_j \to 0$ in $\Wrm^{-1,p}(\R^d;\C^l)$.
\item[(ii)] Let $Z(x) \subset \sigma \C^N$ for all $x \in \cl{\Omega}$ with the property that the orthogonal projection $\Gbb(x) \in \C^{N \times N}$ onto $\spn_\C Z(x)$ is smooth in $x$.
\item[(iii)] $\norm{u_j - \Gbb u_j}_p \to 0$ as $j \to \infty$.
\end{itemize}
Then,
\[
  \WF(\omega) \subset \Xi := \setb{ (x,z,\xi) \in \R^d \times \sigma\C^N \times \Sbb^{d-1} }{ \spn_\C Z(x) \cap \ker \Abb_0(\xi) \neq \{0\} }.
\]
In particular, if $\Xi = \emptyset$, then $\omega = 0$ and $(u_j)$ is strongly compact.
\end{theorem}

A useful special case is:

\begin{corollary} \label{cor:cc_simple}
Let $(u_j) \subset \Lrm^p(\R^d;\C^N)$ generate $\omega \in \MCF^p(\R^d;\C^N)$ and let $\Acal$ be as in Theorem~\ref{thm:cc_general}. Moreover, assume
\[
  \left\{ \begin{aligned}
    \Acal u_j &= 0     \qquad &&\text{in $\Wrm^{-1,p}(\R^d;\C^l)$,} \\
    u_j(x) &\in Z             &&\text{for a.e.\ $x \in \R^d$,}
  \end{aligned} \right.
\]
where $Z \subset \sigma\C^N$ is closed. Then,
\[
  \WF(\omega) \subset \Xi := \setb{ (x,z,\xi) \in \R^d \times \sigma\C^N \times \Sbb^{d-1} }{ \spn_\C Z \cap \ker \Abb_0(\xi) \neq \{0\} }.
\]
In particular, if $\Xi = \emptyset$, then $\omega = 0$ and $(u_j)$ is strongly compact.
\end{corollary}

\begin{remark}
One can also interpret this result as a statement about \enquote{ellipticity}: The directions $\xi$ that do not feature in $\WF(\omega)$ are the directions for which the system \enquote{$\Acal \omega = 0$, $\omega \in Z$} is elliptic, cf.~\cite{Mull99VMMP} for more on this kind of ellipticity.
\end{remark}

\begin{remark}[Necessity of constant-rank condition]
The reason why we require the constant-rank property is that we need to ensure that projections onto $\ker \Abb_0(\xi)$ are continuous, which necessitates that this space has constant dimension as a function of $\xi$. Not much is known in the situation where the constant-rank property is violated, but see~\cite{Mull99ROCI} for a situation in which the constant-rank property is violated, but one can still obtain a result.
\end{remark}

\begin{proof}[Proof of Theorem~\ref{thm:cc_general}]
For all $(x,\xi) \in \R^d \times \Sbb^{d-1}$ with $\spn_\C Z(x,\xi) \cap \ker \Abb_0(\xi) = \{0\}$ we denote by $\Hbb(x,\xi) \colon \C^N \to \C^N$ the non-orthogonal projection with
\[
  \img \Hbb(x,\xi) = \ker \Abb_0(\xi)  \qquad\text{and}\qquad
  \ker \Hbb(x,\xi) = \spn_\C Z(x).
\]
Using some linear algebra in conjunction with (ii), from which it immediately follows that $\rank \spn_\C Z(x) = \mathrm{const}$, and also employing the constant-rank property of $\Acal$, one shows
\[
  \Xi^c = \setb{ (x,z,\xi) \in \R^d \times \sigma\C^N \times \Sbb^{d-1} }{ \norm{\Hbb(x,\xi)} < \infty }.
\]
Moreover, $\Hbb$ is smooth in $\Xi^c$ and positively $0$-homogeneous in $\xi$ (since $\Abb_0$ is).

Let $f(x,z,q) = \phi(x) g(z) \cdot q \in \Fbf^p(\R^d;\C^N)$ and $\Psi \in \Mcal$ with $\supp \, \phi \otimes Eg \otimes \overline{\Psi} \subset \Xi^c_K$ for some $K > 0$, where
\begin{align*}
  \Xi^c_K := \setb{ (x,z,\xi) \in \Xi^c }{ &\norm{\partial^\beta_x \partial^\alpha_\xi \Hbb(x,\xi)} \leq K \text{ for all $\alpha,\beta \in \N_0^d$}\\
  &\text{with $\abs{\alpha}, \abs{\beta} \leq d$} }.
\end{align*}
We will show in the sequel that
\[
  \ddpr{f \otimes \overline{\Psi}, \omega} = 0,
\]
proving the theorem.

Take a smooth cut-off function $\rho \in \Crm^\infty(\R^d \times \Sbb^{d-1}, [0,1])$ such that $\norm{\partial^\beta_x \partial^\alpha_\xi \Hbb(x,\xi)} \leq 2K$ for all $(x,\xi) \in \supp \rho$ and all $\alpha,\beta \in \N_0^d$ with $\abs{\alpha}, \abs{\beta} \leq d$, as well as $\rho(x,\xi) = 1$ whenever $(x,z,\xi) \in \Xi^c_K$ for any $z \in \sigma\C^N$. We consider $\rho$ to be extended to all $\xi \in \R^d \setminus \{0\}$ by positive $0$-homogeneity. Let $T_{\rho\Psi\Hbb}$ be the pseudodifferential operator with \emph{compound} symbol $(x,y,\xi) \mapsto \rho(y,\xi)\Psi(\xi)\Hbb(y,\xi)$, i.e.\ for $u \in \Scal(\R^d)$,
\[
  T_{\rho\Psi\Hbb}[u](x) = \int \!\! \int \rho(y,\xi)\Psi(\xi)\Hbb(y,\xi) u(y) \, \ee^{2\pi(x-y)\cdot\xi} \dd y \dd \xi.
\]
Then, our assumptions imply that $T_{\rho\Psi\Hbb}$ extends to a bounded operator from $\Lrm^p$ to itself, see Chapter~VI of~\cite{Stei93HA} (in particular Section VI.6, which entails that $T_{\rho\Psi\Hbb}$ is also a classical pseudodifferential operator).

We further assume without loss of generality that $u_j \toweak u$ in $\Lrm^p(\R^d;\C^N)$. Then,
\begin{align}
  &\int \phi g(u_j) \cdot \overline{T_{(1-\eta_R)\Psi}[u_j]} \dd x  \notag\\
  &\qquad = \int g(u_j) \cdot \overline{T_{(1-\eta_R)\Psi}[\phi u_j]} \dd x
    + E_1(j,R)  \notag\\
  &\qquad = \int g(u_j) \cdot \overline{T_{(1-\eta_R)\rho\Psi}[\phi u_j]} \dd x
    + E_1(j,R)  \notag\\
  &\qquad = \int \phi g(u_j) \cdot \overline{T_{(1-\eta_R)\rho\Psi}[u_j]} \dd x
    + E_1(j,R) + E_2(j,R)  \notag\\
  &\qquad = \int \phi g(u_j) \cdot \overline{T_{(1-\eta_R)\rho\Psi\Hbb}[u_j]} \dd x
    + E_1(j,R) + E_2(j,R) + E_3(j,R)  \notag\\
  &\qquad = \int \phi g(u_j) \cdot \overline{T_{(1-\eta_R)\rho\Psi\Hbb}[\Gbb u_j]} \dd x  \label{eq:Gcc} \\
  &\qquad\qquad + E_1(j,R) + E_2(j,R) + E_3(j,R) + E_4(j,R)  \notag\\
  &\qquad = 0 + E_1(j,R) + E_2(j,R) + E_3(j,R) + E_4(j,R)  \notag
\end{align}
Here, the last equality follows because $\Hbb(y,\frarg)\Gbb(y) = 0$ by construction. If we can show that
\[
  \lim_{R\to\infty} \lim_{j\to\infty} E_k(j,R) = 0,  \qquad k = 1,2,3,4.
\]
then $\ddpr{f \otimes \overline{\Psi}, \omega} = 0$ and the conclusion of the theorem follows. For $E_1$, $E_2$ this is clear by an argument analogous to the proof of Lemma~\ref{lem:phi_exchange} and for $E_4$ it follows from assumption~(iii).

Concerning $E_3$, we define the linear map $\Jbb(y,\xi) \in \C^{N \times N}$, $(y,\xi) \in \R^d \times \R^d \setminus \{0\}$, by
\[
  \Jbb(y,\xi) w := \begin{cases}
    \rho(y,\xi) \Psi(\xi) \bigl( I_{N \times N} - \Hbb(y,\xi) \bigr) z
        & \text{if $w = \Abb_0(\xi) z$, $z \in \C^N$,}  \\
    0,  & \text{if $w \in (\img \Abb_0(\xi))^\perp$.}
  \end{cases}
\]
From our assumptions we get that the symbol $\R^d \times \R^d \times \R^d \setminus \{0\} \ni (x,y,\xi) \mapsto \Jbb(y,\xi)$ is smooth and positively $0$-homogeneous in $\xi$. Hence, $T_{\Jbb}$ is an $(\Lrm^p \to \Lrm^p)$-bounded pseudodifferential operator, and by construction
\[
  \rho(y,\xi) \Psi(\xi) \bigl( I_{N \times N} - \Hbb(y,\xi) \bigr) = \Jbb(y,\xi) \Abb_0(\xi)
  \quad\text{for all $(y,\xi) \in \R^d \times \R^d \setminus \{0\}$.}
\]
Therefore,
\begin{align*}
  \lim_{j\to\infty} \normb{T_{\rho\Psi - \rho\Psi\Hbb}[u_j] }_p
  &= \lim_{j\to\infty} \normb{T_{\Jbb\Abb_0}[u_j]}_p
    = \lim_{j\to\infty} \normb{T_{\Jbb\Abb_b}[u_j]}_p \\
  &= \lim_{j\to\infty} \normb{T_{\Jbb}[\Acal_b u_j]}_p
    = 0,
\end{align*}
because $\Acal_b u_j \to 0$ in $\Lrm^p(\R^d;\C^l)$. Thus, $\lim_{R\to\infty} \lim_{j\to\infty} E_2(j,R) = 0$. This finishes the proof.
\end{proof}

\begin{remark}
We point out that the preceding results still hold in $\Omega \neq \R^d$ under the same assumptions as in Remark~\ref{rem:domain}. For the second part of that remark, in~\eqref{eq:Gcc} we additionally need to use an argument based on Lemma~\ref{lem:commutation} to ensure that $\Gbb$ is only applied to $\phi u$ with $\supp \phi \subset\subset \Omega$. Then, we only get the result for the \term{restricted wavefront set}:
\[
  \WF_\Omega(\omega) := \setb{ (x,z,\xi) \in \WF(\omega) }{ x \in \Omega } \subset \Xi.
\]
\end{remark}

We can put the preceding result to use in the following examples, which are modeled after the situation occurring in blow-up proofs in the Calculus of Variations, see for example~\cite{Rind12LSYM,Rind11LSIF}:

\begin{example}[Gradients] \label{ex:cc_gradients}
Let $(u_j) \subset \Lrm^p(\Omega;\R^{m \times d})$ with
\[
  \left\{ \begin{aligned}
    \curl \, u_j &= 0         \qquad &&\text{in $\Wrm^{-1,p}$,} \\
    u_j(x) &\in \spn \{M\}        &&\text{for a.e.\ $x \in \Omega$,}
  \end{aligned} \right.
\]
for a fixed matrix $M \in \R^{m \times d}$. The latter condition here means
\[
  u_j(x) = M g_j(x)  \qquad\text{for some $g_j \in \Lrm^p(\Omega;\R)$.}
\]
Then, the following conclusions follow from Corollary~\ref{cor:cc_simple}:
\begin{itemize}
  \item[(i)] If $\rank M \geq 2$, then any MCF $\omega \in \MCF^p(\Omega;\R^{m \times d})$ generated by a subsequence of $(u_j)$ must satisfy
\[
  \qquad \WF_\Omega(\omega) \subset \Xi_\Omega,
\]
where
\[
  \Xi_\Omega := \setb{ (x,z,\xi) \in \Omega \times \sigma\C^{m \times d} \times \Sbb^{d-1} }{ \spn_\C \{M\} \cap \ker \Abb_0(\xi) \neq \{0\} }.
\]
But the set on the right hand side is empty since
\[
  \qquad \ker \Abb_0(\xi) = \set{ a \otimes \xi }{ a \in \C^m }.
\]
Thus, $\omega = 0$ and $(u_j)$ is strongly compact; this result is of course well-known, cf.\ Lemma~2.7 in~\cite{Mull99VMMP} and also Section~3.2 in~\cite{Rind12LSYM}.
  \item[(ii)] If $M = a \otimes n_0$ for $a \in \R^m$, $n_0 \in \Sbb^{d-1}$, then we can still obtain some information:
\[
  \qquad \WF_\Omega(\omega) \subset \setb{ (x,z,\xi) \in \Omega \times \sigma\R^{m \times d} \times \Sbb^{d-1} }{ \xi = \pm n_0 }.
\]
So, while in this case we cannot improve $u_j \toweak u$ to strong convergence, at least we know that oscillations and concentrations can only occur \enquote{in direction $n_0$}.
\end{itemize}
\end{example}

\begin{example}[Linear elasticity]
There exists a second-order constant-rank operator $\Acal$ acting on $\R^{d \times d}_\sym$-valued vector fields that characterizes symmetric gradients, i.e.
\[
  u = \frac{1}{2} \bigl( \nabla v + \nabla v^T \bigr)
  \qquad\text{if and only if}\qquad
  \Acal u = 0,
\]
see Example~3.10~(e) in~\cite{FonMul99AQLS}. Applying our preceding compensated compactness result to the situation of a sequence $(u_j) \subset \Lrm^p(\Omega;\R^{d \times d}_\sym)$ with
\[
  \left\{ \begin{aligned}
    \Acal u_j &= 0         \qquad &&\text{in $\Wrm^{-1,p}$,} \\
    u_j(x) &\in \spn \{M\}        &&\text{for a.e.\ $x \in \Omega$,}
  \end{aligned} \right.
\]
for a fixed matrix $M \in \R^{d \times d}_\sym$. It turns out that
\[
  \ker \Abb_0(\xi) = \setBB{ a \odot \xi = \frac{1}{2} \bigl( a \otimes \xi + \xi \otimes a \bigr) }{ a \in \C^N },  \qquad \xi \in \Sbb^{d-1}.
\]
Then, if $(u_j)$ generates an MCF $\omega \in \MCF^p(\Omega;\R^{d \times d}_\sym)$:
\begin{itemize}
  \item[(i)] If $M$ cannot be written as a symmetric tensor product, then $\omega = 0$ and $(u_j)$ is strongly compact.
  \item[(ii)] If $M = c(a \odot b)$ for some $a,b \in \Sbb^{d-1}$ and $c \in \R$, then
\[
  \qquad \WF_\Omega(\omega) \subset \setb{ (x,z,\xi) \in \Omega \times \sigma\R^{d \times d}_\sym \times \Sbb^{d-1} }{ \text{$\xi = \pm a$ or $\xi = \pm b$} }.
\]
This means that oscillations and concentrations can only occur in direction $a$ \emph{or} $b$ and reflects the inherent rigidity in the system, also see~\cite{Rind11LSIF}, where a similar rigidity was exploited to prove weak* lower semicontinuity of integral functionals in the space BD of functions of bounded deformation.
\end{itemize}
\end{example}

\begin{example}
Let $\Omega \subset \R^2$ and take $\Acal$ to be the constant-rank operator corresponding to Tartar's example
\[
  \partial_1 u^1 + \partial_2 u^2 = 0,  \qquad
  \partial_1 u^3 + \partial_2 u^4 = 0,  \qquad
  u \in \Lrm^p(\Omega;\C^4),
\]
for which ($\xi \in \Sbb^1$)
\[
  \ker \Abb_0(\xi) = \setb{ (a,b,c,d) \in \C^4 }{ (a,b) \perp (\xi_1,\xi_2) \text{ and } (c,d) \perp (\xi_1,\xi_2) }.
\]
Again consider the situation
\[
  \left\{ \begin{aligned}
    \Acal u_j &= 0            \qquad &&\text{in $\Wrm^{-1,p}$,} \\
    u_j(x) &\in \spn \{V\}           &&\text{for a.e.\ $x \in \Omega$,}
  \end{aligned} \right.
\]
where $V = (v_1,v_2,v_3,v_4) \in \C^4$ is fixed. Then:
\begin{itemize}
  \item[(i)] If $(v_1,v_2) \nparallel (v_3,v_4)$, then $\omega = 0$ and $(u_j)$ is strongly compact.
  \item[(ii)] If $(v_1,v_2) \parallel (v_3,v_4)$, then
\[
  \qquad \WF_\Omega(\omega) \subset \setb{ (x,z,\xi) \in \Omega \times \sigma\C^4 \times \Sbb^1 }{ \text{$\xi \perp (v_1,v_2)$} }.
\]
\end{itemize}
\end{example}

\section{Applications} \label{sc:appl}

After having developed aspects of the general theory of microlocal compactness forms, we now move to present several applications.

\subsection{Laminates} \label{ssc:laminates}

The first application concerns laminates (nested microstructures). It will turn out that MCFs not only contain the value distribution and the direction of the various oscillations, but also reflect the hierarchical structure of laminates. As seen in the previous section, they furthermore allow one to quickly read off which differential constraints (e.g.\ curl-freeness) are satisfied by generating sequences.

First, we need the following definition: $\omega \in \MCF^p(\Omega;\C^N)$ is called \term{homogeneous} if $\omega_x, \omega_x^\infty$ are constant in $x$ and $\lambda_\omega = \alpha \Lcal^d \restrict \Omega$ for a constant $\alpha \geq 0$. Equivalently, this is the case if for all $f \in \Fbf^p(\Omega;\C^N)$ of the form $f(x,z,q) = \phi(x) g(z) \cdot q$ and all $\Psi \in \Mcal$ it holds that
\begin{equation} \label{eq:hom_MCF}
  \ddprb{f \otimes \overline{\Psi},\omega} = \dashint_\Omega \phi \dd x \cdot \ddprb{g(z) \cdot q \otimes \overline{\Psi},\omega}.
\end{equation}
For instance, all MCFs generated by the simple (one-directional) oscillations in the Oscillation Lemma~\ref{lem:osc} are homogeneous.

One of the properties of a homogeneous MCF $\omega$ is that its (Lipschitz) domain of definition can be chosen arbitrarily and we simply write $\omega \in \MCF^p_\mathrm{hom}(\C^N)$. A proof sketch of this fact is as follows: Let $(u_j) \subset \Lrm^p(\Omega;\C^N)$ generate a homogeneous MCF $\omega \in \MCF^p(\Omega;\C^N)$ and for simplicity assume $u_j \toweak 0$. Now, let $D \subset \R^d$ be another Lipschitz domain. Choose a Vitali cover $\{a_i + r_i \Omega\}_{i \in \N}$ of $\Lcal^d$-almost all of $D$ consisting of scaled copies of $\Omega$, where $a_i \in D$ and $r_i > 0$. Define $w_j \in \Lrm^p(D;\C^N)$ by
\[
  w_j(x) := u_j \Bigl( \frac{x-a_i}{r_i} \Bigr)  \qquad \text{if $x \in a_i + r_i \Omega$.}
\]
Since $w_j \toweak 0$, Lemma~\ref{lem:eliminate_Rlim} and the following lemma imply that for all $f \in \Fbf^p(D;\C^N)$ with $f(x,z,q) = \phi(x) g(z) \cdot q$, and all $\Psi \in \Mcal$, there exists an error term $E(j)$ with $E(j) \to 0$ as $j \to \infty$ such that
\begin{align*}
  I_j &:= \int_D \phi g(w_j) \cdot \overline{T_\Psi[w_j]} \dd x    \\ 
  &\phantom{:}= \sum_i \int_{a_i + r_i \Omega} \phi(x) g \Bigl(u_j \Bigl( \displaystyle\frac{x-a_i}{r_i} \Bigr) \Bigr) \cdot \overline{T_\Psi \biggl[ u_j \Bigl( \displaystyle\frac{x'-a_i}{r_i} \Bigr) \biggr]} \dd x + E(j).
\end{align*}

We here used the following result, which we factor out because of its independent interest:

\begin{lemma} \label{lem:hom_partition}
Let $(u_j) \subset \Lrm^p(\Omega;\C^N)$ generate the homogeneous microlocal compactness form $\omega \in \MCF^p_\mathrm{hom}(\C^N)$ on the domain $\Omega$ and also $u_j \toweak u$. Moreover, let $\{\Omega_i\}_{i \in \N}$ be a (countable) partition of $\Lcal^d$-almost all of $\Omega$ consisting of open sets $\Omega_i$ ($i \in \N$), $f(x,z,q) = h(x,z) \cdot q \in \Fbf^p(\Omega;\C^N)$, and $\Psi \in \Mcal$. Then,
\[
  \ddprb{f \otimes \overline{\Psi},\omega} = \lim_{j\to\infty} \sum_i \int_{\Omega_i} h(\frarg,u_j) \cdot \overline{T_\Psi[(u_j-u)\ONE_{\Omega_i}]} \dd x 
\]
\end{lemma}

\begin{proof}
By the Locality Lemma~\ref{lem:locality} and the assumption that
\begin{equation} \label{eq:lambda_omega_Leb}
  \lambda_\omega = \wslim_{j\to\infty} \, \abs{u_j}^p \, \Lcal^d \restrict \Omega = \alpha \Lcal^d \restrict \Omega,
\end{equation}
it suffices to show the assertion on the finite union $D_N := \bigcup_{i=1}^N \Omega_i$ for all $N \in \N$. Indeed, from H\"{o}lder's inequality together with~\eqref{eq:lambda_omega_Leb}, we get that the error of restricting to $D_N$ vanishes as $N \toup \infty$.

Let $\eps > 0$ and choose $\zeta_i \in \Crm_c^\infty(\Omega_i)$, $i = 1,\ldots,N$, such that
\[
  \sum_i \limsup_{j\to\infty} \, \normb{h(\frarg,u_j)(\zeta_i - \ONE_{\Omega_i})}_p < \eps
\]
and
\[
  \sum_i \limsup_{j\to\infty} \, \normb{(u_j-u)(\zeta_i - \ONE_{\Omega_i})}_p < \eps.
\]
This is possible again because of~\eqref{eq:lambda_omega_Leb}. Now use Lemmas~\ref{lem:eliminate_Rlim},~\ref{lem:commutation} to derive
\begin{align*}
  \ddprb{f \otimes \overline{\Psi},\omega}_{D_N} &= \lim_{j\to\infty} \int_{D_N} h(\frarg,u_j) \cdot \overline{T_\Psi[u_j-u]} \dd x \\
  &= \lim_{j\to\infty} \sum_i \int_{\Omega_i} \zeta_i h(\frarg,u_j) \cdot \overline{T_\Psi[u_j-u]} \dd x + \BigO(\eps) \\
  &= \sum_i \lim_{j\to\infty} \int_{\Omega_i} h(\frarg,u_j) \cdot \overline{T_\Psi[(u_j-u)\zeta_i]} \dd x + \BigO(\eps) \\
  &= \lim_{j\to\infty} \sum_i \int_{\Omega_i} h(\frarg,u_j) \cdot \overline{T_\Psi[(u_j-u)\ONE_{\Omega_i}]} \dd x + \BigO(\eps)
\end{align*}
Letting $\eps \to 0$, the conclusion of the lemma follows.
\end{proof}

Returning to our previous argument, we get from a short computation using the positive $0$-homogeneity of $\Psi$, and changing variables,
\begin{align*}
  I_j &= \sum_i \int_{a_i + r_i \Omega} \phi(x) g \Bigl(u_j \Bigl( \frac{x-a_i}{r_i} \Bigr) \Bigr) \cdot \overline{T_\Psi[u_j] \Bigl( \frac{x-a_i}{r_i} \Bigr)} \dd x + E(j) \\
  &= \sum_i r_i^d \int_\Omega \phi(a_i + r_i y) g(u_j(y)) \cdot \overline{T_\Psi[u_j](y)} \dd y + E(j).
\end{align*}
Letting $j\to\infty$, we infer from the homogeneity of $\omega$ (with $\omega := \omega_x$ and $\omega^\infty := \omega_x^\infty$),
\begin{align*}
  \lim_{j\to\infty} I_j &= \biggl( \sum_i r_i^d \int_\Omega \phi(a_i + r_i y) \dd y \biggr) \cdot \bigl(\dprb{g \otimes \overline{\Psi},\omega} + \alpha \dprb{g^\infty \otimes \overline{\Psi},\omega^\infty}\bigr) \\
  &= \int_D \phi \dd x \cdot \frac{\ddprb{g(z) \cdot q \otimes \overline{\Psi},\omega}_\Omega}{\abs{\Omega}}.
\end{align*}
On the other hand, if $(w_j)$ generates $\bar{\omega} \in \MCF^p(D;\C^N)$, then $\lim_{j\to\infty} I_j = \ddpr{f \otimes \overline{\Psi}, \bar{\omega}}$. Thus, also $\bar{\omega}$ is homogeneous and
\[
  \frac{\ddprb{g(z) \cdot q \otimes \overline{\Psi},\omega}_\Omega}{\abs{\Omega}} = \frac{\ddprb{g(z) \cdot q \otimes \overline{\Psi},\bar{\omega}}_D}{\abs{D}},
\]
which is what we claimed.

As building blocks for the laminates considered in this section, we will use the Oscillation Lemma~\ref{lem:osc}, also see Example~\ref{ex:osc}. To build up higher-order laminations we will employ the following result:

\begin{proposition}[Laminations] \label{prop:lamination}
Let $A,B \in \C^N$ and assume
\begin{itemize}
  \item[(i)] $u_j \toweak A = \mathrm{const}$ and $v_j \toweak B = \mathrm{const}$ in $\Lrm^p$,
  \item[(ii)] $(u_j)$, $(v_j)$ are $p$-equiintegrable,
  \item[(iii)] $(u_j)$, $(v_j)$ generate the homogeneous MCFs $\omega_1, \omega_2 \in \MCF^p_\mathrm{hom}(\C^N)$ and the homogeneous Young measures $\nu_1,\nu_2 \in \Mbf^1(\C^N)$, respectively.
\end{itemize}
Then, for any $n_0 \in \Sbb^{d-1}$, $\theta \in (0,1)$ there exists an $\Lrm^p$-bounded sequence that generates the homogeneous microlocal compactness form $\bar{\omega} \in \MCF^p_\mathrm{hom}(\C^N)$,
\begin{equation} \label{eq:omega_lamination}
  \bar{\omega} = \theta \omega_1 + (1-\theta) \omega_2 + \Lcal^d \otimes \bigl[ \theta \overline{(A-X)} \nu_1 + (1-\theta) \overline{(B-X)} \nu_2 \bigr] \otimes \overline{\delta}_{\pm n_0},
\end{equation}
where $X := \theta A + (1-\theta) B$.
\end{proposition}

Before we come to the proof, let us interpret the conclusion. The microlocal compactness form $\bar{\omega}$ has two parts: The first is the expected convex combination of $\omega_1$ and $\omega_2$ resulting from the \enquote{fast} oscillations generating $\omega_1$ and $\omega_2$. In contrast to the corresponding result in Young measure theory, however, $\bar{\omega}$ also contains a second part reflecting the \enquote{slowly} oscillating lamination construction in direction $n_0$. In this context recall that the Young measures $\nu_1, \nu_2$ can essentially be calculated from the MCFs $\omega_1, \omega_2$ and the weak limits, see Proposition~\ref{prop:MCF_YM}.

\begin{proof}[Proof of Proposition~\ref{prop:lamination}]
Without loss of generality we assume $n_0 = \ee_d$ and work in $Q = (0,1)^d$, which is possible by the preceding comments on the arbitrary choice of domain for homogeneous MCFs.

\begin{figure}[tb]
\def\svgscale{0.62}
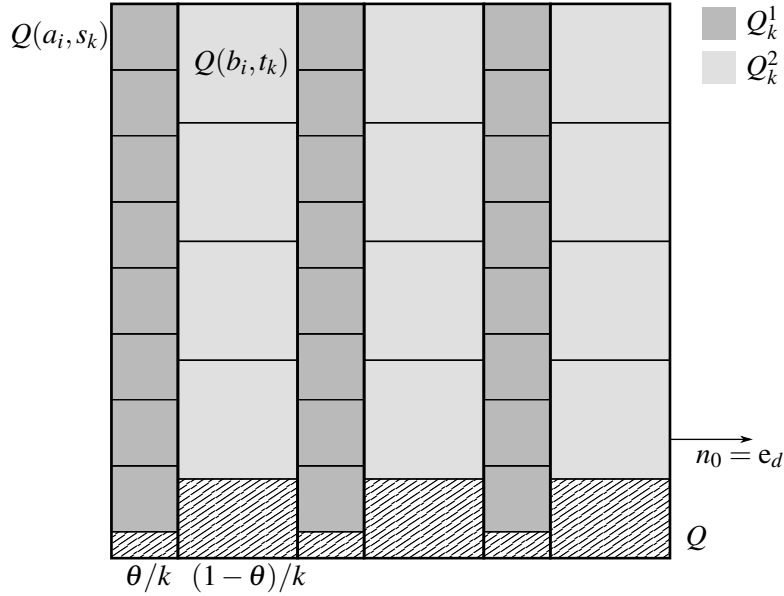
\caption{Construction in the proof of Proposition~\ref{prop:lamination}.}
\label{fig:lamination_proof}
\end{figure}

\proofstep{Step 1.}
For every $k \in \N$ divide $Q$ into $k$ strips of length $1/k$ in direction $n_0 = \ee_d$ as in Figure~\ref{fig:lamination_proof}. Then subdivide every such strip into a strip of length $\theta/k$ and one of length $(1-\theta)/k$. Cover as much as possible of all strips having side length $\theta/k$ with regular cubes $Q(a_i,s_k) := a_i + s_k Q$, where $a_i \in Q$, $i = 1,\ldots, k \cdot \floor{k/\theta}$, and $s_k = \theta/k$. Similarly, cover all strips of side length $(1-\theta)/k$ with cubes $Q(b_i,t_k)$, where $b_i \in Q$, $i = 1,\ldots, k \cdot \floor{k/(1-\theta)}$, and $t_k = (1-\theta)/k$. This of course leaves some fraction of $Q$ uncovered, but the total uncovered volume is less than $1/k$ and so vanishes as $k\to\infty$. Further, set
\[
  D_k^1 := \bigcup_i Q(a_i,s_k),  \qquad
  D_k^2 := \bigcup_i Q(b_i,t_k)
\]
and define
\[
  w_{j,k}(x) := \begin{cases}
    u_j \Bigl( \displaystyle\frac{x-a_i}{s_k} \Bigr)  & \text{if $x \in Q(a_i,s_k)$,} \\
    v_j \Bigl( \displaystyle\frac{x-b_i}{t_k} \Bigr)  & \text{if $x \in Q(b_i,t_k)$,} \\
    0                                    & \text{otherwise.}
  \end{cases}
\]
It is easy to see that (recall $X := \theta A + (1-\theta) B$)
\[
  \bar{w}_k := \wlim_{j\to\infty} w_{j,k} = \begin{cases}
    A  & \text{in $D_k^1$,} \\
    B  & \text{in $D_k^2$,} \\
    0  & \text{otherwise,}
  \end{cases}
  \qquad\qquad
  \wlim_{k\to\infty} \bar{w}_k = X \ONE_Q,
\]
and
\begin{align*}
  \norm{w_{j,k}}^p_p &\leq \theta \norm{u_j}^p_p + (1-\theta) \norm{v_j}^p_p,  \\
  \norm{\bar{w}_k}^p_p &\leq \limsup_{j\to\infty} \, \bigl[ \theta \norm{u_j}^p_p + (1-\theta) \norm{v_j}^p_p \bigr].
\end{align*}

Let $f(x,z,q) = \phi(x) g(z) \cdot q \in \Fbf^p(Q;\C^N)$ with $\phi \in \Crm^\infty(\cl{Q})$, $g \in \Crm^\infty(\C^N)$, and $\Psi \in \Mcal$ be from the collection exhibited in Lemma~\ref{lem:countable_test}. We have
\begin{align}
  &\int_Q \phi g(w_{j,k}) \cdot \overline{T_\Psi[w_{j,k} - X \ONE_Q]} \dd x  \notag\\
  &\qquad = \sum_{m=1,2} \int_{D_k^m} \phi g(w_{j,k}) \cdot \overline{T_\Psi[w_{j,k}-\bar{w}_k]} \dd x \notag\\
  &\qquad\qquad + \sum_{m=1,2} \int_{D_k^m} \phi g(w_{j,k}) \cdot \overline{T_\Psi[\bar{w}_k - X \ONE_Q]} \dd x + E(j,k) \notag\\
  &\qquad =: I_1(j,k) + I_2(j,k) + J_1(j,k) + J_2(j,k) + E(j,k).  \label{eq:int_split}
\end{align}
Here and in all of the following, $E(j,k)$ denotes a generic error term that may change from line to line and satisfies
\[
  \lim_{k\to\infty} \lim_{j\to\infty} E(j,k) = 0.
\]
Above, this property is satisfied because
\begin{align*}
  \abs{E(j,k)} &= \absBB{\int_{Q \setminus (D_k^1 \cup D_k^2)} \phi g(w_{j,k}) \cdot \overline{T_\Psi[w_{j,k} - X \ONE_Q]} \dd x} \\
  &\leq C \norm{\phi}_\infty \cdot \abs{g(0)} \cdot \abs{Q \setminus (D_k^1 \cup D_k^2)}^{(p-1)/p}
  \quad \to \quad 0  \qquad\text{as $k \to \infty$.}
\end{align*}

\proofstep{Step 2.}
We first investigate $I_1$. Applying an argument as in Lemma~\ref{lem:hom_partition}, we get
\begin{align*}
  I_1(j,k) &= \int_{D_k^1} \phi g(w_{j,k}) \cdot \overline{T_\Psi[w_{j,k}-\bar{w}_k]} \dd x \\
  &= \sum_i \int_{Q(a_i,s_k)} \phi(x) g \Bigl(u_j \Bigl(\frac{x-a_i}{s_k}\Bigr)\Bigr) \cdot \overline{T_\Psi\Bigl[ u_j \Bigl(\frac{x'-a_i}{s_k}\Bigr) - A \Bigr](x)} \dd x \\
  &\qquad + E(j,k) \\
  &= \sum_i s_k^d \int_Q \phi(a_i + s_k y) g(u_j(y)) \cdot \overline{T_\Psi[u_j-A](y)} \dd y + E(j,k)
\end{align*}
For the last equality we changed variables and used
\[
  T_\Psi\Bigl[ u_j \Bigl(\frac{x'-a_i}{s_k}\Bigr) \Bigr](x)
  = T_\Psi[u_j]\Bigl(\frac{x-a_i}{s_k}\Bigr),
\]
which can be verified by a simple calculation. Then use the uniform continuity of $\phi$ to get
\begin{align*}
  I_1(j,k) &= \sum_i s_k^d \int_Q \phi(a_i) g(u_j) \cdot \overline{T_\Psi[u_j-A]} \dd y + E(j,k)
\end{align*} 

Passing to the limits $j\to\infty$, $k\to\infty$, we arrive at
\begin{align*}
  \lim_{k\to\infty} \lim_{j\to\infty} I_1(j,k)
    &= \lim_{k\to\infty} \sum_i s_k^d \phi(a_i) \cdot \lim_{j\to\infty} \int_Q g(u_j) \cdot \overline{T_{\Psi}[u_j-A]} \dd y \\
  &= \biggl(\lim_{k\to\infty} \sum_i s_k^d \phi(a_i)\biggr) \cdot \ddprb{g(z) \cdot q \otimes \overline{\Psi}, \omega_1}_Q.
\end{align*}
We further compute the Riemann sums
\[
  \lim_{k\to\infty} \sum_i s_k^d \phi(a_i) = \theta \dashint_Q \phi \dd x.
\]
Thus, also employing~\eqref{eq:hom_MCF},
\begin{equation} \label{eq:I1_lim}
  \lim_{k\to\infty} \lim_{j\to\infty} I_1(j,k) = \theta \, \dashint_Q \phi \dd x \cdot \ddprb{g(z) \cdot q \otimes \overline{\Psi}, \omega_1}_Q
  = \theta \, \ddprb{f \otimes \overline{\Psi}, \omega_1}.
\end{equation}
Analogously, one shows that
\begin{equation} \label{eq:I2_lim}
  \lim_{k\to\infty} \lim_{j\to\infty} I_2(j,k) = (1-\theta) \, \ddprb{f \otimes \overline{\Psi}, \omega_2}.
\end{equation}

\proofstep{Step 3.}
Turning to $J_1$, we first observe
\begin{align*}
  J_1(j,k) &= \int_{D_k^1} \phi g(w_{j,k}) \cdot \overline{T_\Psi[\bar{w}_k - X \ONE_Q]} \dd x \\
  &= \sum_i \int_{Q(a_i,s_k)} \phi(x) g \Bigl(u_j \Bigl(\frac{x-a_i}{s_k}\Bigr)\Bigr) \cdot \overline{T_\Psi[\bar{w}_k - X \ONE_Q](x)} \dd x \\
  &= \sum_i s_k^d \int_Q \phi(a_i + s_k y) g(u_j(y)) \cdot \overline{T_\Psi[\bar{w}_k - X \ONE_Q](a_i + s_k y)} \dd y.
\end{align*}
Now let $j \to \infty$ and represent the limit using the Young measure from assumption~(iii), using the equiintegrability from assumption~(ii):
\begin{align*}
  &\lim_{j\to\infty} J_1(j,k) \\
  &\qquad = \sum_i s_k^d \int_Q \phi(a_i + s_k y) \cdot \biggl[\int g(z) \dd \nu_1(z)\biggr] \cdot \overline{T_\Psi[\bar{w}_k - X \ONE_Q](a_i + s_k y)} \dd y \\
  &\qquad = \int_{D_k^1} \phi(x) \cdot \biggl[\int g(z) \dd \nu_1(z)\biggr] \cdot \overline{T_\Psi[\bar{w}_k - X \ONE_Q](x)} \dd x.
\end{align*}
Clearly, an analogous assertion holds for $J_2(j,k)$. Hence, with
\[
  Y_k(x) := \sum_{m = 1,2} \ONE_{D_k^m}(x) \int g(z) \dd \nu_m(z)
\]
we have
\[
  \lim_{j\to\infty} \bigl( J_1(j,k) + J_2(j,k) \bigr) = \int_Q \phi Y_k \cdot \overline{T_\Psi[\bar{w}_k - X \ONE_Q]} \dd x.
\]

By a geometric argument in Fourier space analogous to the one in the proof of the Oscillation Lemma~\ref{lem:osc}, we may further derive
\begin{align}
  &\lim_{k\to\infty} \lim_{j\to\infty} \bigl( J_1(j,k) + J_2(j,k) \bigr)  \notag\\
  &\qquad = \lim_{k\to\infty} \int_Q \phi Y_k \cdot \overline{(\Psi(+n_0)+\Psi(-n_0)) (\bar{w}_k - X \ONE_Q)} \dd x  \notag\\
  &\qquad = \lim_{k\to\infty} \int_{D_k^1} \phi \cdot \biggl[\int g \dd \nu_1\biggr] \cdot \overline{(\Psi(+n_0)+\Psi(-n_0)) (A - X)} \dd x  \notag\\
  &\qquad\quad + \lim_{k\to\infty} \int_{D_k^2} \phi \cdot \biggl[\int g \dd \nu_2\biggr] \cdot \overline{(\Psi(+n_0)+\Psi(-n_0)) (B - X)} \dd x  \notag\\
  &\qquad = \theta \int_Q \phi \dd x \cdot \int g \dd \nu_1 \cdot  \overline{(\Psi(+n_0)+\Psi(-n_0)) (A - X)} \notag\\
  &\qquad\quad + (1-\theta) \int_Q \phi \dd x \cdot \int g \dd \nu_2 \cdot \overline{(\Psi(+n_0)+\Psi(-n_0)) (B - X)}.  \label{eq:I34_lim}
\end{align}
This corresponds to the action of the microlocal compactness form
\[
  \omega_\mathrm{mix} := \Lcal^d \restrict Q \otimes \bigl[ \theta \overline{(A-X)} \nu_1 + (1-\theta) \overline{(B-X)} \nu_2 \bigr] \otimes \overline{\delta}_{\pm n_0}.
\]

\proofstep{Step 4.}
Now select a diagonal subsequence $(w_\ell) = (w_{j(\ell),k(\ell)})$ of $(w_{j,k})$ such that
\begin{align*}
  &\lim_{\ell\to\infty} \int_Q \phi g(w_\ell) \cdot \overline{T_\Psi[w_\ell - X \ONE_Q]} \dd x \\
  &\qquad = \lim_{k\to\infty} \lim_{j\to\infty} \int_Q \phi g(w_{j,k}) \cdot \overline{T_\Psi[w_{j,k} - X \ONE_Q]} \dd x.
\end{align*}
Combining this with~\eqref{eq:int_split},~\eqref{eq:I1_lim},~\eqref{eq:I2_lim}, and~\eqref{eq:I34_lim}, we have
\begin{align*}
  &\lim_{\ell\to\infty} \int_Q \phi g(w_\ell) \cdot \overline{T_\Psi[w_\ell - X \ONE_Q]} \dd x \\
  &\hspace{-2pt}\qquad = \theta \ddprb{f \otimes \overline{\Psi},\omega_1} + (1-\theta) \ddprb{f \otimes \overline{\Psi},\omega_2} + \ddprb{f \otimes \overline{\Psi},\omega_\mathrm{mix}} = \ddprb{f \otimes \overline{\Psi},\bar{\omega}},
\end{align*}
with $\bar{\omega}$ given in~\eqref{eq:omega_lamination}. An application of Lemma~\ref{lem:eliminate_Rlim} shows that $(w_\ell)$ generates $\bar{\omega}$ and the proof is finished.
\end{proof}

As mentioned before, the preceding result is particularly useful in combination with the Oscillation Lemma~\ref{lem:osc} as we can then compute the MCF generated by laminates of any order. We only illustrate this with the following example of a second-order laminate. Since all sequences will be uniformly bounded, $p \in (1,\infty)$ (and the domain) can be chosen arbitrarily.

\begin{figure}[tb]
\def\svgscale{0.62}
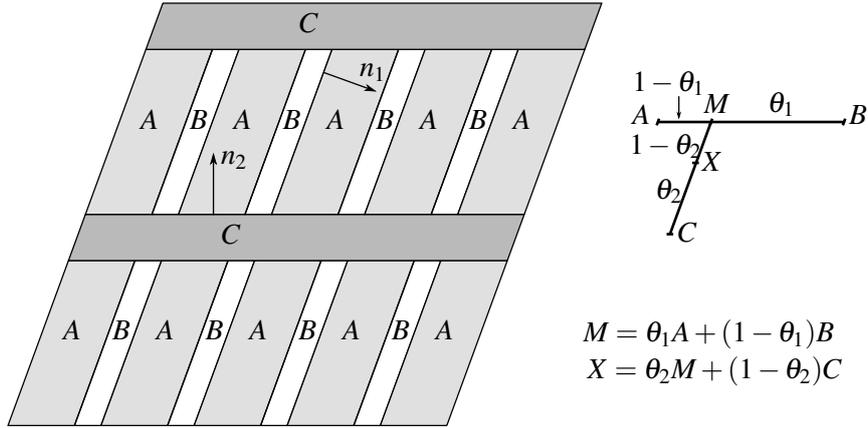
\caption{Second-order lamination in Example~\ref{ex:2nd_order_lam}.}
\label{fig:lamination_ex}
\end{figure}

\begin{example} \label{ex:2nd_order_lam}
Let $A,B,C \in \C^N$, $n_1, n_2 \in \Sbb^{d-1}$ and $\theta_1, \theta_2 \in (0,1)$. We want to construct the laminate shown schematically in Figure~\ref{fig:lamination_ex}. Define
\[
  u_j(x) := \begin{cases}
    A  & \text{if $\floor{j x \cdot n_1} \in (0,\theta_1)$,} \\
    B  & \text{if $\floor{j x \cdot n_1} \in (\theta_1,1)$,}
  \end{cases}
\]
and $v_j(x) := C$ constant. It follows from Lemma~\ref{lem:osc} (also see Example~\ref{ex:osc}) that $(u_j)$ generates the homogeneous MCF
\[
  \omega_1 = \Lcal^d \otimes \bigl[ \theta_1 \overline{(A-M)} \delta_A + (1-\theta_1) \overline{(B-M)} \delta_B \bigr] \otimes \overline{\delta}_{\pm n_1},
\]
where
\[
  M := \theta_1 A + (1-\theta_1) B.
\]
Trivially, $(v_j)$ generates the zero MCF $\omega_2 = 0$. It can also be shown easily that the sequences $(u_j)$ and $(v_j)$ generate the Young measures $\nu_1 = \theta_1 \delta_A + (1-\theta_1) \delta_B$ and $\nu_2 = \delta_C$, respectively. Therefore, with
\[
  X := \theta_2 M + (1-\theta_2) C,
\]
an application of Proposition~\ref{prop:lamination} yields the existence of the MCF
\begin{align*}
  \bar{\omega} &= \theta_2 \omega_1 + (1-\theta_2) \omega_2 + \Lcal^d \otimes \bigl[ \theta_2 \overline{(M-X)} \nu_1 + (1-\theta_2) \overline{(C-X)} \nu_2 \bigr] \otimes \overline{\delta}_{\pm n_2} \\
  &= \Lcal^d \otimes \Bigl\{ \bigl[ \theta_2 \theta_1 \overline{(A-M)} \delta_A + \theta_2 (1-\theta_1) \overline{(B-M)} \delta_B \bigr] \otimes \overline{\delta}_{\pm n_1} \\
  &\qquad\quad\hspace{2pt} + \bigl[ \theta_2 \overline{(M-X)} (\theta_1 \delta_A + (1-\theta_1) \delta_B)    + (1-\theta_2) \overline{(C-X)} \delta_C\bigr] \otimes \overline{\delta}_{\pm n_2} \Bigr\}.
\end{align*}
Consequently, from this MCF we can not only read off the distribution of values and the oscillation directions, but also the \emph{hierarchy} of the laminate.

Finally, if $A,B,C \in \R^{m \times d}$ and we impose the constraint that $\curl \, w_j = 0$, we infer from Theorems~\ref{thm:diff_omega},~\ref{thm:diff_genseq}, also see Example~\ref{ex:osc_curl0}, the necessary and sufficient conditions
\[
  A-B = a \otimes n_1  \quad\text{and}\quad
  M-C = b \otimes n_2  \qquad\text{for some $a,b \in \R^m$,}
\]
for $w_j$ to be asymptotically a sequence of gradients. This conclusion is of course in agreement with Hadamard's jump condition (in conjunction with a boundary adjustment), also see~\cite{BalJam87FPMM}.
\end{example}

However, we note that here the \enquote{jump condition} is only imposed asymptotically, not for all elements of the sequence.

\subsection{Relaxation} \label{ssc:relax}

Just like in Young's original works~\cite{Youn37GCEA,Youn42GSCV_1,Youn42GSCV_2,Youn80LCVO}, we now consider as an application the relaxation of integral functionals. More precisely, we extend integral functionals depending on $\Lrm^p$-gradients to functionals on $p$-growth gradient MCFs and study microstructures of minimizing sequences.

In all of the following, $\Omega \subset \R^d$ is an open Lipschitz domain. Let $f(x,A) = h(x,A) : A \in \Fbf^p(\Omega;\R^{m \times d})$, where \enquote{$:$} denotes the scalar product between matrices (when considered as vectors in $\R^{md}$). We consider the functional
\[
  \Jcal[\nabla u] := \int_\Omega f(x, \nabla u(x)) \dd x,
  \qquad u \in \Wrm^{1,p}(\Omega;\R^m).
\]
For relaxation problems involving non-(quasi)convex $f$, usually the following three questions are of interest:
\begin{itemize}
  \item[(Q1)] What is the \emph{relaxation} of the functional, i.e.\ the $\Wrm^{1,p}$-weakly lower semicontinuous envelope of $\Jcal$ (defined to be the largest $\Wrm^{1,p}$-weakly lower semicontinuous function below $\Jcal$)?
  \item[(Q2)] Can one extend the functional $\Jcal$ to a larger space, in which the minimization problem always has a solution?
  \item[(Q3)] What is the \emph{microstructure} that develops in minimizing sequences to reach that relaxed value?
\end{itemize}

As is well-known, see for example Chapter~9 of~\cite{Daco08DMCV}, the \term{relaxation} is given by
\[
  \Jcal_*[\nabla u] = \int_\Omega Qf(x, \nabla u(x)) \dd x,
  \qquad u \in \Wrm^{1,p}(\Omega;\R^m),
\]
where $Qf \colon \Omega \times \R^{m \times d} \to \R$ denotes the quasiconvex envelope of $f$ with respect to the second argument. We recall that a locally bounded Borel function $g \colon \R^{m \times d} \to \R$ is called \term{quasiconvex} if
\[
  g(A) \leq \dashint_{B(0,1)} g(A + \nabla \phi(y)) \dd y
\]
for all $A \in \R^{m \times d}$ and all $\phi \in \Crm_c^\infty(B(0,1);\R^m)$. It can be shown that the domain $B(0,1)$ can equivalently be replaced by any other bounded Lipschitz domain and, if $g$ has $p$-growth, $\phi$ can be chosen from $\Wrm^{1,p}(B(0,1);\R^m)$ instead. The \term{quasiconvex envelope} $Qg$ is defined as the largest quasiconvex function below $g$. If $g$ is continuous and bounded from below (which we assume), it can be expressed via Dacorogna's formula
\[
  Qg(A) = \inf \, \setBB{ \dashint_{B(0,1)} g(A + \nabla \phi(y)) \dd y }{ \phi \in \Crm_c^\infty(B(0,1);\R^m) }.
\]
We here use the convention that $Qg$ and hence also $\Jcal_*$ are allowed to take the value $-\infty$. See Chapters~5,~6 in~\cite{Daco08DMCV} for details on these notions. This classical relaxation formula answers the first question~(Q1): The relaxation of $\Jcal$ is simply $\Jcal_*$.

To investigate the second question~(Q2), we define $\MCF_\mathrm{gen}^p(\Omega;\R^{m \times d})$ to denote the subset of all $\omega \in \MCF^p(\Omega;\R^{m \times d})$ with the additional property that there exists a sequence $(W_j) \subset \Lrm^p(\Omega;\R^{m \times d})$ generating $\omega$ (this is not automatic, see Example~\ref{ex:MCF_nogen}). Then we consider the following relaxation problem:

\begin{problem}[Relaxation] \label{pr:gamma}
Find an \enquote{extension} $\bar{\Jcal}$ of $\Jcal$ onto the larger space $(\Lrm^p \times \MCF^p)(\Omega;\R^{m \times d})$. More precisely, find a functional $\bar{\Jcal} \colon (\Lrm^p \times \MCF_\mathrm{gen}^p)(\Omega;\R^{m \times d})$ such that
\begin{align*}
  &\min \, \setB{ \bar{\Jcal}[W,\omega] }{ \text{$(W,\omega) \in (\Lrm^p \times \MCF_\mathrm{gen}^p)(\Omega;\R^{m \times d})$, $\curl \, W = 0$, and}\\
  &\hspace{115pt}\text{$\omega$ is a gradient MCF in the sense of~\eqref{eq:Aomega0_ext},~\eqref{eq:A_curl}} } \\
  &\qquad = \inf \, \setB{ \Jcal[\nabla u] }{ u \in \Wrm^{1,p}(\Omega;\R^m) }
\end{align*}
and
\[
  \Jcal[\nabla u_j] \to \bar{\Jcal}[\nabla u,\omega]
  \text{ whenever $\nabla u_j \toweak \nabla u$ and $(\nabla u_j)$ generates the MCF $\omega$.}
\]
\end{problem}

We compute, using Lemma~\ref{lem:bandlim_compact} and employing the (classical) Young measure $(\nu_x)_{x \in \Omega}$ generated by the sequence $(\nabla u_j)$,
\begin{align*}
  \lim_{j\to\infty} \Jcal[\nabla u_j] &= \lim_{R\to\infty} \lim_{j\to\infty} \int_\Omega h(\frarg,\nabla u_j) : \overline{T_{\eta_R}[\nabla u_j]} \dd x \\
  &\qquad + \lim_{R\to\infty} \lim_{j\to\infty} \int_\Omega h(\frarg,\nabla u_j) : \overline{T_{1-\eta_R}[\nabla u_j]} \dd x \biggr] \\
  &= \int_\Omega \int h(x,z) : \overline{\nabla u(x)} \dd \nu_x(z) \dd x
    + \ddprb{f \otimes I,\omega} \\
  &= \int_\Omega \dprb{h(x,\frarg) : \overline{\nabla u(x)}, \nu_x} + \dprb{h(x,\frarg) \otimes I,\omega_x} \dd x \\
  &\qquad + \int_{\cl{\Omega}} \dprb{h^\infty(x,\frarg) \otimes I,\omega_x^\infty} \dd \lambda_\omega(x).
\end{align*}
By Proposition~\ref{prop:MCF_YM}, the Young measure $(\nu_x)$ can be computed from $\nabla u$ and $\omega$, whence we infer the existence of $\bar{\Jcal}$ that only depends on $\nabla u$ and $\omega$. In principle, we can also write down an expression for $\bar{\Jcal}$ only depending on $\nabla u$ and $\omega$, using the procedure from the proof of Proposition~\ref{prop:MCF_YM}. However, in general this is rather cumbersome, so we stick to the above expression involving both the MCF and the Young measure. 

Trivially, $\inf \bar{\Jcal} \leq \inf \Jcal$. On the other hand, the minimization of $\bar{\Jcal}$ cannot yield a value strictly lower than the infimum of $\Jcal$, because, by assumption all admissible MCFs are generated by sequences of gradients (after projecting, see Remark~\ref{rem:Afree_proj}). Finally, if $\nabla u_j \toweak \nabla u$ and $(\nabla u_j)$ generates the MCF $\omega$, then $\Jcal[\nabla u_j] \to \bar{\Jcal}[\nabla u,\omega]$. Thus, for the extended functional $\bar{\Jcal}$ the minimum value is always attained.  This shows that $\bar{\Jcal}$ satisfies all the requirements of Problem~\ref{pr:gamma}

Turning to~(Q3), we follow an approach to study the formation of microstructure when approaching the minimum popularized in M\"{u}ller's influential lecture notes~\cite{Mull99VMMP}, and investigate the following related problem involving \emph{pointwise} minimizers: Define the set of pointwise minimizers $Z(x) := \mathrm{argmin} \, f(x) \subset \sigma\R^{m \times d}$ of $f$, i.e.\
\[
  A_* \in Z(x) \qquad\text{if and only if}\qquad
  f(x,A_*) \leq f(x,A) \quad\text{for all $A \in \sigma\R^{m \times d}$.}
\]
Moreover, assume that $Z(x)$ satisfies the continuity condition (ii) in Theorem~\ref{thm:cc_general}, which is for example trivially satisfied if $Z(x) = Z$ for a.e.\ $x \in \cl{\Omega}$. Now assume that a sequence $(u_j) \subset \Wrm^{1,p}(\Omega;\R^m)$ is not only minimizing, but even satisfies the following stronger differential inclusion:
\begin{equation} \label{eq:min_diff_incl}
  \nabla u_j(x) \in Z(x)  \qquad\text{for $\Lcal^d$-a.e.\ $x \in \Omega$.}
\end{equation}
Alternatively, we can assume the weaker condition~(iii) from Theorem~\ref{thm:cc_general}. By said theorem (also cf.\ Example~\ref{ex:cc_gradients}), we can conclude
\begin{equation} \label{eq:WF_relax}
\begin{aligned}
  \WF_\Omega(\omega) \subset \setb{ (x,z,\xi) &\in \Omega \times \sigma\R^{m \times d} \times \Sbb^{d-1} }{ \text{There exists $a \in \R^m \setminus \{0\}$} \\
  &\text{such that $a \otimes \xi \in \spn_\C Z(x)$} }.
\end{aligned}
\end{equation}
This in particular means that oscillations and concentrations can only occur \emph{in fixed directions} or even not at all if $\spn_\C Z(x)$ does not contain any rank-one matrix.

The considerations so far in particular imply that the theory of microlocal compactness forms retains enough information to study relaxations in the presence of anisotropic effects. This will be demonstrated in the following example, which shows that the MCF allows us to infer directional properties of microstructure:

\begin{example}
Let $n_0 \in \Sbb^{d-1}$ be a unit vector and consider the following integrand with critical anisotropy:
\[
  f(A) := \abs{A}^2 - n_0^T A^T A n_0,  \qquad A \in \R^{m \times d}.
\]
Note that this is the only integrand in the family $\abs{A}^2 - \gamma n_0^T A^T An_0$, $\gamma \in \R$, that is both bounded from below and has non-trivial pointwise minima.
Alternatively, one can write $f$ in the form
\[
  f(A) = A : A - A (n_0 \otimes n_0) : A,
\]
and so it follows that $f \in \Fbf^p(\Omega;\R^{m \times d})$. We can also estimate that $f \geq 0$ and by an elementary computation, $(\partial f/\partial A)(A_0) = 2 A_0 - 2 A_0(n_0 \otimes n_0)$ vanishes if and only if $A_0 = a \otimes n_0$ for some $a \in \R^m$. 

Assume that $(u_j) \subset \Wrm^{1,2}(\Omega;\R^m)$ such that $(\nabla u_j)$ generates the (generalized) Young measure $\nu = (\nu_x,\lambda_\nu,\nu_x^\infty) \in \Ybf^2(\Omega;\R^{m \times d})$ as well as the MCF $\omega \in \MCF^2(\Omega;\R^{m \times d})$. Then, we can write $\bar{\Jcal}$ as a function of $\nu$,
\begin{align*}
  \bar{\Jcal}[\nu] &= \int_\Omega \int_{\R^{m \times d}} \abs{A}^2 - n_0^T A^T A n_0 \dd \nu_x(A) \dd x\\
  &\qquad + \int_{\cl{\Omega}} \int_{\partial \B^{m \times d}} 1 - n_0^T A^T A n_0 \dd \nu_x^\infty(A) \dd \lambda_\nu(x).
\end{align*}

To investigate minimizing microstructure, assume furthermore that $(\nabla u_j)$ satisfies the differential inclusion~\eqref{eq:min_diff_incl} with
\[
  Z := \mathrm{argmin} \, f = \set{ a \otimes n_0 }{ a \in \R^m }.
\]
Then, the wavefront condition~\eqref{eq:WF_relax} entails
\[
  \WF_\Omega(\omega) \subset \setb{ (x,z,\xi) \in \Omega \times \sigma\R^{m \times d} \times \Sbb^{d-1} }{ \xi = n_0 }.
\]
This allows us to draw the following conclusions:
\begin{enumerate}
  \item Oscillations and concentrations asymptotically have to be in direction $n_0$.
  \item Since the Fourier support of a concentration profile (cf.\ Example~\ref{ex:conc}) is asymptotically close to the set $\{\pm n_0\}$, concentrations have to look like \enquote{ridges} with normal $n_0$; single poles for example are not possible.
\end{enumerate}
We conclude the discussion by noting that the theory of MCFs provides a way to assign a precise meaning to statements about the \enquote{asymptotic direction} of oscillations and concentrations, which otherwise is only an intuitive notion.
\end{example}

\subsection{Propagation of singularities} \label{ssc:propagation}

In this last section we briefly explore an application to propagation of singularities in semilinear PDE systems, which was also one of Tartar's main motivations behind H-measures~\cite{Tart90HMNA}, also cf.~\cite{Miel99FPYM,McPaTa85WLSH}.

Define the linear PDE operator
\[
  \Acal := \sum_k A^{(k)} \partial_k,
  \qquad \text{where $A^{(k)} \in \R^{m \times m}$, $k = 1,\ldots,d$,}
\]
and consider for $u \colon (0,T) \times \Omega \to \R^m$ ($\Omega$ a Lipschitz domain) the semilinear system
\begin{equation}  \label{eq:semlin}
  \partial_t u - \Acal u = g(x,u),
\end{equation}
where $g \colon \Omega \times \R^m \to \R^m$ is a given semilinearity, which we assume to be a Carath\'{e}odory function with $\abs{g(x,z)} \leq C(1+\abs{z})$ for a constant $C > 0$ and all $(x,z) \in \Omega \times \R^m$.

The system is \term{hyperbolic} if for all $\xi \in \R^d \setminus \{0\}$ the matrix $\sum_k A^{(k)} \xi_k$ has only real eigenvalues or, equivalently, the symbol $\Abb(\xi)$ has only purely imaginary eigenvalues.

The aim of this section is to \enquote{replace} $u$ by an MCF $\omega$ in the above system and derive the \enquote{equations} that $\omega$ satisfies. This then allows us to understand better the propagation of singularities along the flow governed by~\eqref{eq:semlin}.

Let $f(x,z,q) = h(x,z) \cdot q \in \Fbf^p((0,T) \times \Omega;\R^m)$ be of class $\Crm^1$ and $\Psi \in \Mcal$. Abbreviating the Fourier multiplier $T_{(1-\eta_R) \Psi}$ by $T$, and suppressing the $x$-dependence of various quantities for ease of notation, we compute for $\phi \in \Crm_c^1((0,T) \times \Omega)$:
\begin{align}
  &-\int \!\! \int (\partial_t \phi) \, h(u) \cdot \overline{T[u]} \dd x \dd t \notag\\
  &\qquad = \int \!\! \int \phi \, \bigl( Dh(u) \partial_t u \bigr) \cdot \overline{T[u]} + \phi \, h(u) \cdot \overline{T[\partial_t u]} \dd x \dd t \notag\\
  &\qquad = \int \!\! \int \phi \, \bigl( Dh(u) \Acal u \bigr) \cdot \overline{T[u]} + \phi \, \bigl( Dh(u) g(u) \bigr) \cdot \overline{T[u]}   \label{eq:semlin_PDE_omega_int}\\
  &\qquad\qquad\qquad + \phi \, h(u) \cdot \overline{T[\Acal u]} + \phi \, h(u) \cdot \overline{T[g(u)]} \dd x \dd t.  \notag
\end{align}
For the third term we get
\begin{align*}
  &\int \!\! \int \phi \, h(u) \cdot \overline{T[\Acal u]} \dd x \dd t
    = - \sum_k \int \!\! \int \partial_k \bigl( \phi \, h(u) \bigr) \cdot \overline{A^{(k)} T[u]} \dd x \dd t \\
  &\qquad= - \sum_k \int \!\! \int (\partial_k \phi) \, h(u) \cdot \overline{A^{(k)} T[u]}
    + \phi \, \bigl( Dh(u) \partial_k u \bigr) \cdot \overline{A^{(k)} T[u]} \dd x \dd t.
\end{align*}
At this point of the discussion we further need to assume on $h$ the structural \term{commutation relations}
\begin{equation} \label{eq:Ak_commute}
  [A^{(k)}]^* Dh(u) = Dh(u) A^{(k)}  \qquad\text{for $k = 1,\ldots,d$.}
\end{equation}
Note that these are clearly satisfied if all $A^{(k)}$ are real and scalar, i.e.\ $A^{(k)} = a^{(k)} I_{m \times m}$ for $a^{(k)} \in \R$. Then, with~\eqref{eq:Ak_commute}, we may continue the preceding transformations to arrive at
\[
  \int \!\! \int \phi \, h(u) \cdot \overline{T[\Acal u]} \dd x \dd t
  = - \int \!\! \int (\Acal^* \phi) \, h(u) \cdot \overline{T[u]} + \phi \, \bigl( Dh(u) \Acal u \bigr) \cdot \overline{T[u]} \dd x \dd t,
\]
where
\[
  \Acal^* := \sum_k [A^{(k)}]^* \partial_k.
\]
Plugging this into~\eqref{eq:semlin_PDE_omega_int} and observing the surprising fact that we may cancel terms, we get
\begin{align*}
  &- \int \!\! \int (\partial_t \phi) \, h(u) \cdot \overline{T[u]} \dd x \dd t \\
    &\quad\;\; = \int \!\! \int \phi \, \bigl( Dh(u) g(u) \bigr) \cdot \overline{T[u]} - (\Acal^* \phi) \, h(u) \cdot \overline{T[u]} + \phi \, h(u) \cdot \overline{T[g(u)]} \dd x \dd t.
\end{align*}

Now assume we are given a sequence $(u_j) \subset (\Wrm^{1,p} \cap \Crm^1)((0,T) \times \Omega;\R^m)$ with every $u_j$ solving~\eqref{eq:semlin} in the above sense for a semilinearity $g_j$, that is,
\[
  \partial_t u_j - \Acal u_j = g_j(x,u_j).
\]
We also assume $u_j \toweak u$. In fact, it suffices to assume that the preceding equation is satisfied in the weaker sense that the $\Lrm^p$-norm of $\partial_t u_j - \Acal u_j - g_j(x,u_j)$ vanishes as $j\to\infty$. Up to a subsequence, the pairs $(u_j,g_j(\frarg,u_j))$ generate a microlocal compactness form $\omega \in \MCF^p((0,T) \times \Omega;\C^m \times \C^m)$. Write the preceding equation for $u_j$ and let first $j\to\infty$, then $R\to\infty$, to derive
\begin{align*}
  &-\ddprb{(\partial_t \phi) \, h(z_1) \cdot q_1 \otimes \overline{\Psi}, \omega} \\
  &\qquad = \ddprb{\phi \, \bigl[ Dh(z_1) z_2 \cdot q_1 + h(z_1) \cdot q_2 \bigr] \otimes \overline{\Psi}, \omega} 
  -\ddprb{(\Acal^* \phi) \, h(z_1) \cdot q_1 \otimes \overline{\Psi}, \omega}.
\end{align*}
Here, the splittings $z = (z_1,z_2), q = (q_1,q_2) \in \C^m \times \C^m$ correspond to the two parts of the generating sequence.

Rearranging, we have arrived at:

\begin{theorem} \label{thm:MCF_system}
The MCF $\omega \in \MCF^p((0,T) \times \Omega;\R^m \times \R^m)$ satisfies the \term{extended (MCF) system} associated to~\eqref{eq:semlin}, that is
\[
  \ddprb{(\partial_t \phi - \Acal^* \phi) \, h(z_1) \cdot q_1 \otimes \overline{\Psi}, \omega} = -\ddprb{\phi \, \bigl[ Dh(z_1) z_2 \cdot q_1 + h(z_1) \cdot q_2 \bigr] \otimes \overline{\Psi}, \omega}
\]
for all $\phi \in \Crm_c^1((0,T) \times \Omega)$, $h \in \Crm^1(\R^m;\R^m)$ such that $h^\infty$ exists in the sense of~\eqref{eq:h_infty} and the commutation relations~\eqref{eq:Ak_commute} are satisfied, and all $\Psi \in \Mcal$.
\end{theorem}

Notice that the right-hand side in this definition takes the role of a semilinear interaction term and is zero if $g_j \equiv 0$ for all $j$. This extended system now defines a flow on the level of MCFs and hence describes the propagation of oscillations and concentrations or, equivalently, of (lack of) compactness. Consequently, a solution to the extended system could be called a \enquote{compactness flow}.

If the original system is linear and $g_j = g$ is just a function of $x$, we can simplify the above extended system to
\[
  \ddprb{(\partial_t \phi - \Acal^* \phi) \, h(z) \cdot q \otimes \overline{\Psi}, \omega_u} = -\ddprb{\phi \, Dh(z) g \cdot q \otimes \overline{\Psi}, \omega_u},
\]
where now $\omega_u$ is only the $u$-part of $\omega$. This shows that in this situation oscillations and concentrations are only transported according to a similar law as in the original system and the more complicated interactions between the solution and the semilinearity, which are present in the full extended semilinear system above, do not occur.

\begin{remark}[Comparison to PDEs for measures,~\cite{JoMeRa95TCCN}]
To illustrate the differences between the preceding result and previous works on  \enquote{PDEs for measures} or \enquote{Young measure solutions} to PDEs as for instance in Theorem~1.5 of~\cite{JoMeRa95TCCN}, we write the Young measure system corresponding to  our system~\eqref{eq:semlin}: Consider a sequence of solutions $u_j$ and a corresponding sequence $v_j(x) = g(x,u_j(x))$ of semilinearities, which we further assume to generate Young measures $\mu_{t,x}$ and $\nu_{t,x}$, respectively (this can be achieved after selecting a subsequence). We arrive at the system
\[
  \partial_t \mu_{t,x} - \Acal \mu_{t,x} = \nu_{t,x}  \qquad\text{in $\Mbf(\R^m)$.}
\]
Many works on this topic such as~\cite{JoMeRa95TCCN} now focus on the relationship between $\mu_{t,x}$ and $\nu_{t,x}$: Indeed, it is one of the main achievements of Theorem~1.5 in \textit{loc.\ cit.} that, in their more restricted situation, this system can be written in terms of $\mu_{t,x}$ only, thus expressing a compensated compactness fact (in the sense that products of weakly converging sequences converge to the product of the weak limits, which is the main thrust of~\cite{JoMeRa95TCCN}). Our result, however, goes further than the above Young measure system: We are not only interested in the flow of the \emph{value distributions} expressed in the equations for $\mu_{t,x}$ and $\nu_{t,x}$ (and in fact one can find equivalent equations for the \emph{distribution functions}), but moreover in a flow describing the dynamics of the \emph{frequency} and \emph{directional} properties of asymptotic oscillations. In the MCF system in Theorem~\ref{thm:MCF_system} this is expressed through the \enquote{frequency/directional test function} $\Psi$.
\end{remark}

We leave finer investigations into the propagation of singularities and compactness to future work. For now we only remark in closing that in the context of the above extended (MCF) system, the Compensated Compactness Theorem~\ref{thm:cc_general} allows one to prove statements about the \enquote{ellipticity} present in the system.




\providecommand{\bysame}{\leavevmode\hbox to3em{\hrulefill}\thinspace}
\providecommand{\MR}{\relax\ifhmode\unskip\space\fi MR }
\providecommand{\MRhref}[2]{%
  \href{http://www.ams.org/mathscinet-getitem?mr=#1}{#2}
}
\providecommand{\href}[2]{#2}

\end{document}